\documentclass{amsart}

\usepackage{amsmath,amsfonts,amssymb}

\numberwithin{equation}{section}

\theoremstyle{plain}
\newtheorem{theorem}{Theorem}[section]
\newtheorem{corollary}[theorem]{Corollary}
\newtheorem{lemma}[theorem]{Lemma}
\newtheorem{proposition}[theorem]{Proposition}
\theoremstyle{remark}
\newtheorem{remark}[theorem]{Remark}
\theoremstyle{definition}
\newtheorem{definition}[theorem]{Definition}
\newtheorem{notation}[theorem]{Notation}

\makeatletter

\def\R{\mathbb{R}}
\def\C{\mathbb{C}}
\def\Z{\mathbb{Z}}

\newcommand{\lie}[1]{\mathfrak{#1}}
\newcommand{\lier}[1]{\mathfrak{#1}_{0}}
\newcommand{\roots}[2]{\Delta({#1}, {#2})}
\newcommand{\proots}[2]{\Delta^{+}({#1}, {#2})}

\def\gr{\mathrm{gr}}
\def\Supp{\mathrm{Supp}}
\def\AssVar{\mathcal{AV}}
\def\AssCyc{\mathcal{AC}}
\def\Hom{\mathrm{Hom}}
\def\Ker{\mathrm{Ker}}

\def\Alt{\mathrm{Alt}}

\def\Ind{\mathrm{Ind}}
\def\Ad{\mathrm{Ad}}
\def\ad{\mathrm{ad}}
\def\pr{\mathrm{pr}}
\def\Dim{\mathrm{Dim}}
\def\Deg{\mathrm{Deg}}
\def\WF{\mathrm{WF}}
\def\Wh{\mathrm{Wh}}
\def\sgn{\mathrm{sgn}}

\def\d{\partial}
\newcommand{\vect}[1]{\mathbf{#1}}
\def\I{\iota}

\allowdisplaybreaks[2]

\title{Discrete series Whittaker functions on $Spin(2n,2)$} 
\author{Kenji Taniguchi}
\thanks{Department of Physics and Mathematics, 
Aoyama Gakuin University, 
5-10-1, Fuchinobe, Sagamihara, Kanagawa 229-8558, Japan. 
(taniken@gem.aoyama.ac.jp)}


\begin{document}

\begin{abstract}
Discrete series Whittaker functions on $Spin(2n, 2)$ are studied. 
The dimensions of the space of both algebraic and continuous Whittaker
modelss are explicitly determined. 
They are described by a sum of dimensions of irreducible
representations of $Spin(2n-3, 2)$. 
Also obtained are the Mellin-Barnes type integral formulas of the
Whittaker functions associated with minimal $K$-type vectors. 
\end{abstract}

\maketitle


\section{Introduction}
Let $G_{\R}$ be a real semisimple Lie group and 
$G_{\R} = K_{\R} A_{\R} N_{\R}$ be its Iwasawa decomposition. 
Let $\eta$ be a one dimensional unitary representation of $N_{\R}$. 
Given a representation $\pi$ of $G_{\R}$, a realization of $\pi$ in
the induced representation $\Ind_{N_{\R}}^{G_{\R}} \eta$ is called a
Whittaker model of $\pi$.  
By this realization, a vector $v$ in $\pi$ is expressed by a function
on $G_{\R}$, which we call the Whittaker function associated with $v$. 

Mainly from the number theoretical point of view, Takayuki Oda and
his colleagues have calculated explicit formulas of Whittaker
functions and generalized spherical functions on various low rank
groups and of various representations. 

Besides the number theoretical interest, the theory of Whittaker
models are also interesting from the analytic points of view. 
It is well known that, 
in the case of the principal series of $SL(2, \R)$, 
the Whittaker functions associated with a $K_{\R}$-type vector is
expressed by the classical Whittaker's confluent hypergeometric
function. 
In the case of representations of higher rank groups, 
the image of a Whittaker model gives an example of multi-variable
confluent hypergeometric functions. 
Therefore, we may expect that we can study such functions by means of
representation theory. 

Another interesting aspect is the relationship with the invariants of
representations. 
First, let us consider algebraic Whittaker models, 
namely $(\lie{g}, K_{\R})$-module intertwining operators from a
Harish-Chandra $(\lie{g}, K_{\R})$-module $\pi$ 
to $\Ind_{N_{\R}}^{G_{\R}} \eta$. 
If the one dimensional unitary representation $\eta$ of $N_{\R}$
is non-degenerate (cf.\S \ref{subsection:Whittaker models}), 
we can judge the existence of non-trivial Whittaker models by the
Gelfand-Kirillov dimension of $\pi$, 
and the dimension of the Whittaker models is given by the Bernstein
degree of it. 
Secondly, let us consider continuous Whittaker models, 
namely continuous intertwining operators from the
$C^{\infty}$-globalization $\pi_{\infty}$ of a $(\lie{g},
K_{\R})$-module $\pi$ to the $C^{\infty}$-induced space 
$C^{\infty}$-$\Ind_{N_{\R}}^{G_{\R}} \eta$. 
We can judge the existence of non-trivial continuous Whittaker models
by the wave front set of $\pi$. If $\pi$ is a discrete series, 
the dimension of continuous Whittaker models are expressed by the
Bernstein degree and geometric data of nilpotent orbits. 
These results, due to H. Matumoto (\cite{M1}, \cite{M2}), 
are summarized in \S \ref{section:general theory}.

The author thinks that we can observe the relationship between the
invariants of representations and the structure of Whittaker models more
clearly when we treat Whittaker models of non-quasi-split groups than
that of quasi-split groups. 
Suppose $G_{\R}$ is a non-quasi-split real semisimple Lie group. 
Consider a discrete series representation $\pi$ which has non-trivial
Whittaker models. 
Then it is observed that the degree of an irreducible component of the
associated variety corresponds to the dimension of the solution space
of one system of differential equation. 
It is also observed that the multiplicity of the
associated cycle corresponds to the dimension of the right 
$Z_{M_{\R}}(\eta)$-module structure of the solution space of the
gradient type differential-difference equation 
$\mathcal{D}_{\tilde{\lambda}, \eta} \phi = 0$.  
(For the definition of this equation, 
see \S \ref{subsection:Realization of Whittaker functions}). 
Here, $M_{\R}(\eta)$ is the centralizer of $\eta$ in 
$M_{\R} = Z_{K_{\R}}(A_{\R})$. 
This observation is obtained in a former paper \cite{T} of the
author's. 
He expects that such correspondence holds for general higher rank group
cases, and hopes to explain it in a natural way. 

Until now, there are not so many concrete examples of non-quasi-split
cases. 
For such reasons, we investigates the case when $G_{\R}$ is the
non-quasi-split group $Spin(2n, 2)$, 
$\pi$ is a discrete series of $G_{\R}$, 
and $\eta$ is a non-degenerate character of $N_{\R}$ in this paper. 
This setting is the same as in \cite{HO}, except for the group
$G_{\R}$. 
Since $Spin(4, 2) \simeq SU(2, 2)$, this paper is a generalization of
\cite{HO}.

The sections are organized as follows. 
In \S \ref{section:Spin(2n,2)}, we briefly review the structure
of $Spin(2n, 2)$ and parametrize its discrete series
representations. 
The set of discrete series representations of $Spin(2n,2)$ is divided
into $2n+2$ parts $\Xi_{m,\pm}$, $m = 1, \dots, n+1$. 
\S \ref{section:general theory} is mainly devoted to the
presentation of general theory. 
We review H. Matumoto's results on the existence and the dimension of
Whittaker models in \S \ref{subsection:Whittaker models}, 
J. T. Chang's results on the associated cycles of discrete series in
\S \ref{subsection:Chang's results}, and H. Yamashita's
results on the realization of algebraic Whittaker models of discrete
series in \S \ref{subsection:Realization of Whittaker functions}. 
In \S \ref{subsection:SO case}, we apply Chang's theory to our
$Spin(2n, 2)$ case. 
As a result, combined with Theorem~\ref{thm:last theorem, algebraic}, 
we obtain the dimension of the space of algebraic Whittaker models: 
\begin{theorem}[Corollary~\ref{cor:existing cases},
    Theorem~\ref{thm:Bernstein}]   
Suppose $G_{\R} = Spin(2n, 2)$.   
\begin{enumerate}
\item
The discrete series $\pi_{\Lambda}$ has non-trivial algebraic
Whittaker models 
if and only if $\Lambda \in \Xi_{m,\pm}$, $m = 2, \dots, n$. 
\item
The dimension of the space of algebraic Whittaker models of
$\pi_{\Lambda}$, $\Lambda \in \Xi_{m,\pm}$, $m=2,\dots,n$, 
is 
\[4 \sum_{\genfrac{}{}{0pt}{}{
\lambda_1 \geq \mu_1 \geq \lambda_2 \geq \dots 
\geq \lambda_{m-2} \geq \mu_{m-2} \geq \lambda_{m-1}}
{\lambda_m \geq \mu_1' \geq \lambda_{m+1} \geq \dots 
\geq \lambda_{n-1} \geq \mu_{n-m}' \geq |\lambda_n|} }
\dim
V_{(\mu_1,\dots,\mu_{m-2},\mu_1',\dots,\mu_{n-m}')}^{Spin(2n-3,\C)}. 
\]
\end{enumerate}
\end{theorem}
In \S \ref{section:Radial A part}, we write down the gradient type
differential-difference equations, which describe
discrete series Whittaker functions, by using Gelfand-Tsetlin basis of
the minimal $K_{\R}$-type. 
After that, we put them in order. 
Solving such differential-difference equations reduces to getting 
a coefficient function of some special minimal $K_{\R}$-type vector, 
which we call ``corner vector''. 
This reduction is discussed in
\S \ref{subsection:corner}, and the result is as follows. 
\begin{theorem}[Theorem~\ref{thm:corner}]
Let $Q_{n-1}^{+}$ be the Gelfand-Tsetlin pattern defined in
Definition~\ref{def:Q^pm}. 
Let 
$\phi(a) 
= 
\sum_{Q \in GT(\lambda)} 
c(Q; a) Q$ 
be a solution of the differential-difference equation 
$\mathcal{D}_{\tilde{\lambda}, \eta} \phi = 0$, which describes an
algebraic Whittaker model of the discrete series $\pi_{\Lambda}^{\ast}$, 
$\Lambda \in \Xi_{m,\pm}$, $m = 2, \dots, n$. 
If $c(Q_{n-1}^{+}; a)$ is known, then 
it determines all the $c(Q; a)$ for $Q \in GT(\lambda)$ containing the
same $\vect{q}_{2n-4}$ part as $Q_{n-1}^{+}$. 
\end{theorem}
By this theorem, we may restrict our interest to determine the
coefficient function $c(Q_{n-1}^{+}; a)$. 
In \S \ref{section:solution}, we deduce a system of differential
equations satisfied by the coefficient function $c(Q; a)$ for some
special Gelfand-Tsetlin pattern $Q$. 
The set of such special patterns contains $Q_{n-1}^{+}$. 
Also obtained are the Mellin-Barnes type integral expression of these
functions. 
In order to present these results briefly, 
we use the notation defined in \S\S \ref{section:Radial A part},
\ref{section:solution}. 
See \eqref{eq:K_6(m')} for the definition of $K_{6}(m')$, 
and see \S \ref{subsection:solutions of scalar equations} for
the definition of $\alpha_{j}$, $\beta_{j}$, $t_{j}$, $n(Q, m'; a)$. 
\begin{theorem}[Proposition~\ref{proposition:modefied equations}, 
Proposition~\ref{proposition:local solutions}] 
\label{thm:1.3} 
Assume $Q \in GT(\lambda)$ satisfies \eqref{eq:5.13(1)},
\eqref{eq:5.11} and \eqref{eq:5.13(4)}. 
\begin{enumerate}
\item
Define $f(Q, m'; t) := n(Q, m'; a)^{-1} c(Q; a)$. 
Then, $f(Q, m'; t)$ is a solution of 
\begin{align*}
& 
(\d_{t_{1}}^{2} - \d_{t_{2}}^{2} - t_{1}^{2}) f(Q, m'; t) = 0, 
\\
& 
\left\{
\prod_{p=1}^{N_{2}} 
(\d_{t_{2}} + \alpha_{p}) 
+ 
\frac{t_{2}}{t_{1}} 
(\d_{t_{1}} - \d_{t_{2}}) 
\prod_{p=3}^{N_{2}} 
(\d_{t_{2}} + \beta_{p}) 
\right\} 
f(Q, m'; t) 
= 0. 
\end{align*}
\item
Let $C_{j}$ be a loop starting and ending at $+ \infty$, 
crossing the real axis at $- \alpha_{j} - 1 < s < -\alpha_{j}$, 
and encircling all poles of 
$\Gamma(- \alpha_{p} - s)$, $p = j, \dots, N_{2}$, once in the negative
direction, but none of the poles of 
$\Gamma(\beta_{p} + s)$, $p = 3, \dots, j$. 
Define 
\begin{equation*}
\begin{Bmatrix}
f_{1}^{K}(Q, m'; t) 
\\
f_{1}^{I}(Q, m'; t) 
\end{Bmatrix} 
:= 
\frac{1}{2 \pi \I} 
\int_{C_{1}} 
\frac{\prod_{p=1}^{N_{2}} \Gamma(-\alpha_{p} - s)}
{\prod_{p=3}^{N_{2}} \Gamma(1 - \beta_{p} - s)} 
\begin{Bmatrix} 
t_{2}^{s}\,  
K_{-s}(t_{1})
\\
(-t_{2})^{s}\,  I_{-s}(t_{1})
\end{Bmatrix} 
ds, 
\end{equation*}
and, for $j = 2, \dots, N_{2}$, define 
\begin{align*}
& 
\begin{Bmatrix}
f_{j}^{K}(Q, m'; t) 
\\
f_{j}^{I}(Q, m'; t) 
\end{Bmatrix} 
\\
& 
:= 
\frac{1}{2 \pi \I} 
\int_{C_{j}} 
\frac{\prod_{p=j}^{N_{2}} \Gamma(-\alpha_{p} - s) 
\prod_{p=3}^{j} \Gamma(\beta_{p} + s)}
{\prod_{p=1}^{j-1} \Gamma(1 + \alpha_{p} + s) 
\prod_{p=j+1}^{N_{2}} \Gamma(1 - \beta_{p} - s)} 
\begin{Bmatrix} 
(- t_{2})^{s}\,  
K_{-s}(t_{1}) 
\\
t_{2}^{s}\,  
I_{-s}(t_{1})
\end{Bmatrix} 
ds.  
\end{align*}
Here, $K_{\mu}(z)$ and $I_{\nu}(z)$ are modified Bessel functions, 
and $\I = \sqrt{-1}$. 
These integrals absolutely converge in 
$\C^{2} \setminus (\{t_{1} = 0\} \cup \{t_{2} = 0\})$, 
and they form a basis of the solution space of the system of
differential equations in (1). 
\end{enumerate}
\end{theorem}

An interesting fact is that not all of these functions generates a 
solution of the whole differential-difference equation 
$\mathcal{D}_{\tilde{\lambda}, \eta} \phi = 0$. 
In \S \ref{subsection:whole solutions}, 
we construct shift operators satisfied by the functions $f_{j}^{K}$,
$f_{j}^{I}$ 
(Propositions~\ref{prop:first shift operator}, \ref{prop:S_2}). 
By these shift operators, we know that only $4$ in the $2 N_{2}$
functions $f_{j}^{K}$, $f_{j}^{I}$ {\it do} generate solutions of the
whole differential-difference equations. 
\begin{theorem}[Theorem~\ref{thm:last theorem, algebraic}] 
Let $\Lambda \in \Xi_{m,\pm}$, $m = 2, \dots, n$. 
Then the functions $f_{j}^{L}(Q_{n-1}^{+}, n-1; t)$, 
with $j = 1, 2$, $L = K, I$, 
completely determines the solutions of 
$\mathcal{D}_{\tilde{\lambda}, \eta} \phi = 0$. 
\end{theorem}
Some of the above solutions may define a continuous
intertwining operator from the $C^{\infty}$-globalization
$(\pi_{\Lambda}^{\ast})_{\infty}$ to 
$C^{\infty}(G_{\R}/N_{\R}; \eta)$. 

\begin{theorem}[Theorem~\ref{thm:continuous intertwining space}]
Let $\Lambda \in \Xi_{m,\pm}$, $m = 2, \dots, n$. 
Suppose the character $\eta$ of $N_{\R}$ is the one defined by 
\eqref{eq:character}. 
If $\mp \eta_{2} > 0$, then the dimension of the space of continuous
intertwining operators is 
\[
\sum_{\genfrac{}{}{0pt}{}{
\lambda_1 \geq \mu_1 \geq \lambda_2 \geq \dots 
\geq \lambda_{m-2} \geq \mu_{m-2} \geq \lambda_{m-1}}
{\lambda_m \geq \mu_1' \geq \lambda_{m+1} \geq \dots 
\geq \lambda_{n-1} \geq \mu_{n-m}' \geq |\lambda_n|} }
\dim
V_{(\mu_1,\dots,\mu_{m-2},\mu_1',\dots,\mu_{n-m}')}^{Spin(2n-3,\C)}. 
\]
Each continuous intertwining operator corresponds to $f_{1}^{K}$
defined in Theorem~\ref{thm:1.3} (2). 
On the other hand, if $\mp \eta_{2} < 0$, then this space is zero. 
\end{theorem} 

When $G_{\R} = Spin(2n+1, 2)$, we have similar results on the
Whittaker models of discrete series representations. 
Details are discussed elsewhere about this case. 

Before ending the introduction, we define some notation. 
Suppose $H_{\R}$ is a real Lie group. 
Write $\lier{h}$ the Lie algebra of $H_{\R}$, 
$\lie{h}$ the complexification of $\lier{h}$ 
and $H$ a complexification of $H_{\R}$. 
This notation will be applied to groups denoted by other Roman letters
in the same way without comment. 
For two integers $a < b$, let $[a,b]$ be the interval 
$\{x \in \Z | a \leq x \leq b\}$. 
The imaginary unit is denoted by $\iota$. 

\section{The group $Spin(2n,2)$ and its discrete series} 
\label{section:Spin(2n,2)}

\subsection{Structure of $Spin(2n,2)$} 

Let $G_{\R} = Spin(2n,2) \subset Spin(2n+2, \C)$ ($n \geq 2$) be the
connected two-fold linear cover of $SO_{0}(2n, 2)$, 
whose maximal compact subgroup $K_{\R}$ is 
isomorphic to $Spin(2n) \times Spin(2)$. 
We realize the Lie algebra $\lier{g} = \lie{so}(2n,2)$ as 
\begin{align*}
\mathfrak{so}(2n,2)
&=
\left\{
\left.
\begin{pmatrix}
X_{11} & \I X_{12} 
\\
-\I {}^{t}X_{12} & X_{22} 
\end{pmatrix} 
\right| 
\begin{matrix}
X_{11} \in \Alt_{2n}(\R), X_{22} \in \Alt_{2}(\R), 
\\
X_{12} \in M_{2n,2}(\R)
\end{matrix}
\right\}.
\end{align*}
Let $\theta X = -{}^{t} \bar{X}$ be a Cartan involution of $\lier{g}$
and let $\lier{g}=\lier{k}+\lier{p}$ the corresponding Cartan
decomposition. 
Denote elementary matrices by 
$E_{ij} = (\delta_{ki} \delta_{lj})_{k,l = 1, \dots, 2n+2}$ 
and set $F_{ij} := E_{ij} - E_{ji}$. 
Then 
\begin{align*}
& 
\{
F_{j,k} | 
1 \leq k < j \leq 2n, \mbox{ or } (j, k) = (2n+2, 2n+1)\} 
\quad 
\mbox{and }  
\\
&
\{
\I F_{2n+j,k} | 
1 \leq j \leq 2, 1 \leq k \leq 2n\} 
\end{align*}
are bases of $\lier{k}$ and $\lier{p}$ respectively.

Let 
\[
A_{k} := \I F_{2n+k, 2n-2 + k},  
\qquad 
\lier{a} 
:= 
\R A_{1} + \R A_{2}, 
\]
and define $f_{j} \in \lier{a}^{\ast}$ by $f_{j}(A_{k}) = \delta_{j,k}$. 
Then $\lier{a}$ is a maximal abelian subspace of $\lier{p}$. 
The restricted root system 
$\roots{\lier{g}}{\lier{a}}$ is 
\[
\roots{\lier{g}}{\lier{a}} 
= 
\{\pm f_{1} \pm f_{2}, \pm f_{1}, \pm f_{2}\}. 
\]
Choose a positive system 
\[
\proots{\lier{g}}{\lier{a}} 
= 
\{\pm f_{1} + f_{2}, f_{1}, f_{2}\}, 
\]
and denote the corresponding nilpotent subalgebra 
$\sum_{\alpha \in \proots{\lier{g}}{\lier{a}}} (\lier{g})_{\alpha}$ 
by $\lier{n}$. 
Here $(\lier{g})_{\alpha}$ is the root space corresponding to a root
$\alpha$. 
One obtains an Iwasawa decomposition 
\begin{align*}
& 
\lier{g} = \lier{k} + \lier{a} + \lier{n}, 
& 
& 
G_{\R} = K_{\R} A_{\R} N_{\R}, 
\end{align*}
where $A_{\R} = \exp \lier{a}$ and $N_{\R} = \exp \lier{n}$. 
Let 
\begin{align*}
& 
X_{f_{j}}^{k} 
:= 
F_{2n-2+j,k} + \I F_{2n+j,k},  
\qquad 
(1 \leq j \leq 2, 1 \leq k \leq 2n-2), 
\\
&
X_{f_{1}+f_{2}} 
:= 
F_{2n,2n-1} - \I F_{2n+1,2n} + \I F_{2n+2,2n-1} - F_{2n+2,2n+1},  
\\
& 
X_{-f_{1}+f_{2}} 
:= 
F_{2n,2n-1} + \I F_{2n+1,2n} + \I F_{2n+2,2n-1} + F_{2n+2,2n+1}. 
\end{align*}
Then 
$\{X_{f_{j}}^{k} | 1 \leq k \leq 2n-2\}$ is a basis of
$(\lier{g})_{f_{j}}$, 
and 
$\{X_{\pm f_{1} + f_{2}} \}$ is a basis of 
$(\lier{g})_{\pm f_{1} + f_{2}}$.

\subsection{Parameterization of discrete series}

Let us now parametrize the discrete series of $G_{\R}$. 
Take a compact Cartan subalgebra $\lie{t}$ of $\lie{g}$ defined by 
\[
T_{k} := - \I F_{2k,2k-1},
\qquad 
\lie{t} 
:= 
\bigoplus_{k=1}^{n+1} \C T_{k}. 
\]
Let $\{e_{j} | 1 \leq j \leq n+1\}$ be the dual of $\{T_{k}\}$. 
The root systems $\Delta := \roots{\lie{g}}{\lie{t}}$ and
$\Delta_{c} := \roots{\lie{k}}{\lie{t}}$ are 
\begin{align*}
& 
\Delta
= 
\{ \pm e_{i} \pm e_{j} | 1 \leq i < j \leq n+1\}, 
&
& 
\Delta_{c}
= 
\{ \pm e_{i} \pm e_{j} | 1 \leq i < j \leq n\}, 
\end{align*}
respectively. 
Choose a positive system 
\[
\Delta_{c}^{+}
= 
\{ e_{i} \pm e_{j} | 1 \leq i < j \leq n\} 
\]
of $\Delta_{c}$. 
There are $2n+2$ positive systems $\Delta_{1, \pm}^{+}, \dots,
\Delta_{n+1, \pm}^{+}$ of $\Delta$ containing $\Delta_{c}^{+}$, 
defined by 
\begin{align*}
\Delta_{m, +}^{+} 
&=
\Delta_{c}^{+} \cup 
\{e_i \pm e_{n+1}| 1 \leq i \leq m-1\}
\cup 
\{e_{n+1}\pm e_i| m \leq i \leq n\}, 
\\
& \quad \mbox{if} \quad 1 \leq m \leq n+1,
\\
\Delta_{m, -}^{+} 
&=
\Delta_{c}^{+} \cup 
\{e_i \pm e_{n+1}| 1 \leq i \leq m-1\}
\cup 
\{-e_{n+1}\pm e_i| m \leq i \leq n\}, 
\\
& \quad \mbox{if} \quad 1 \leq m \leq n,
\\
\Delta_{n+1, -}^{+} 
&=
\Delta_{c}^{+} \cup 
\{e_i \pm e_{n+1}| 1 \leq i \leq n-1\}
\cup 
\{-e_n\pm e_{n+1}\}. 
\end{align*}
The discrete series representations of $G_{\R}$ are parametrized by
Harish-Chandra parameters. 
By definition, the set of Harish-Chandra parameters $\Xi^{+}$ is 
\[
\Xi^{+}
=
\{\Lambda \in \lie{t}^{\ast} 
| \Lambda \mbox{ is } \Delta\mbox{-regular, }
\Delta_{c}^{+}\mbox{-dominant, and } 
\Lambda + \rho \mbox{ is } K_{\R}\mbox{-analytically integral} 
\}. 
\]
Here, $\rho$ is half the sum of positive roots for some positive system
of $\Delta$. 
Hereafter, we denote by $\pi_{\Lambda}$ the discrete series
representation corresponding to $\Lambda \in \Xi^{+}$. 
In our case, $\Xi^{+}$ is divided into $2n+2$ parts: 
\[
\Xi^{+} 
= 
\bigcup_{m=1}^{n+1} \Xi_{m, +} 
\cup 
\bigcup_{m=1}^{n+1} \Xi_{m, -}, 
\qquad  
\Xi_{m, \pm} 
:=
\{\Lambda \in \Xi^{+} |  
\langle \Lambda, \beta \rangle \geq 0, 
\forall \beta \in \Delta_{m, \pm}^{+}
\}.
\]

\section{Discrete series Whittaker functions} 
\label{section:general theory} 

\subsection{Whittaker models} 
\label{subsection:Whittaker models}

Let $\eta: N_{\R} \longrightarrow \C^{\times}$ be a non-degenerate
unitary character of $N_{\R}$. 
We denote the differential representation $\lier{n} \to \I \R$ 
of $\eta$ by the same letter $\eta$. 
``Non-degenerate'' means that $\eta$ is
non-trivial on every root space corresponding to a simple root of
$\Delta^{+}(\lier{g}, \lier{a})$. 
The space of analytic and smooth Whittaker functions are defined by 
\begin{align*}
& 
\mathcal{A}(G_{\R}/N_{\R};\eta)
=\{F: G_{\R} \overset{\mbox{\scriptsize analytic}}{\rightarrow} \C| 
F(gn)=\eta(n)^{-1}F(g), 
\mbox{ for } n\in N_{\R}, g\in G_{\R}\},
\\
& 
\enskip C^{\infty}(G_{\R}/N_{\R};\eta)
=\{F: G_{\R} \overset{\mbox{\scriptsize $C^{\infty}$}}{\rightarrow} \C| 
F(gn)=\eta(n)^{-1}F(g), 
\mbox{ for } n\in N_{\R}, g\in G_{\R}\},  
\end{align*}
respectively. 
These are representation spaces of $G_{\R}$ by left translation. 

Let $(\pi, V)$ be a representation of $G_{\R}$ 
or a Harish-Chandra $(\lie{g}, K_{\R})$-module. 
A realization of $(\pi, V)$ in 
$\mathcal{A}(G_{\R}/N_{\R};\eta)$ 
is called a Whittaker model of $(\pi, V)$. 
By this realization, a vector $v \in V$ is expressed by a function on
$G_{\R}$, which we call a Whittaker function associated with $v$.

We review some results, due to Matumoto, on the existence and the
dimension of Whittaker models. 
\begin{theorem}[\cite{M1}]\label{thm:matu-1}
Let $V$ be an irreducible Harish-Chandra $(\lie{g},K_{\R})$-module. 
\begin{enumerate}
\item
$V$ has a non-trivial Whittaker model 
if and only if the Gelfand-Kirillov dimension $\Dim V$ of $V$ is equal
to $\dim\lier{n}$. 
\item
If $V$ has a non-trivial Whittaker model, 
then the dimension of Whittaker models is equal to the Bernstein
degree $\Deg V$ of $V$. 
\end{enumerate}
\end{theorem}
Any irreducible Harish-Chandra $(\lie{g}, K_{\R})$-module $V$ 
admits an infinitesimal character. 
Therefore, any smooth Whittaker function $\psi(v)(g)$, 
with $v \in V$ and 
$\psi \in \Hom_{\lie{g}, K_\R}(V, C^\infty(G_\R/N_{\R}; \eta))$, 
is an eigenfunction of the Casimir operator. 
Since the Casimir operator is an elliptic differential operator, 
$\psi(v)$ is real analytic. 
Therefore we may identify 
$\Hom_{\lie{g}, K_\R}(V, \mathcal{A}(G_\R/N_{\R}; \eta))$ with 
$\Hom_{\lie{g}, K_\R}(V, C^\infty(G_\R/N_{\R}; \eta))$, 
if $V$ is irreducible. 

The next problem is to specify the continuous intertwining operators. 
For a Harish-Chandra $(\lie{g}, K_{\R})$-module $(\pi, V)$, 
let $(\pi_{\infty}, V_{\infty})$ be its $C^{\infty}$-globalization. 
This is a Fr\'echet representation of $G_{\R}$. 
The space of continuous intertwining operators from $V_{\infty}$ to
$C^{\infty}(G_{\R}/N_{\R}; \eta)$ will be denoted by 
$\Hom_{G_{\R}}^{\infty}(V_{\infty}, C^\infty(G_\R/N_{\R}; \eta))$. 
This is a subspace of 
$\Hom_{\lie{g}, K_\R}(V, C^\infty(G_\R/N_{\R}; \eta))$, 
and it is isomorphic to the space 
\[
\Wh_{-\eta}^{\infty}(V) 
= \{\xi \in V_{\infty}' | \pi_{\infty}'(X) \xi = - \eta(X) \xi, 
X \in \lier{n}\}
\]
of $C^{-\infty}$-Whittaker vectors (cf \cite{M2}). 
Here, $(\pi_{\infty}', V_{\infty}')$ is the continuous dual to
$(\pi_{\infty}, V_{\infty})$ with respect to the
$U(\lie{g})$-topology. 

The next theorem, also due to Matumoto, presents a condition for the
existence of non-trivial continuous intertwining operators and the
dimension of the space of such operators. 
In order to state his results, we introduce some conventions. 
Since a unitary character $\psi$ is determined by the value 
on the root spaces corresponding to simple roots, 
we may regard $\psi$ as an element of 
$\I (\lier{n}/[\lier{n}, \lier{n}])^{\ast} 
\subset \I \lier{g}^{\ast}$. 
Using the Killing form, we identify $\lier{g}^{\ast}$ with
$\lier{g}$. 
Note that $\psi$ is non-degenerate if and only if 
$\I^{-1} \Ad(G_{\R}) \psi \subset \lier{g}^{\ast} \simeq \lier{g}$ is
a principal nilpotent $G_{\R}$-orbit. 
We denote by $\mathcal{P}(G_{\R})$ the set of principal nilpotent
$G_{\R}$-orbits. 
The wave front set of $V$ is denoted by $\WF(V)$. 
For the definition of wave front set, see \cite{M2}. 
\begin{theorem}[\cite{M2}] \label{thm:matu-2} 
Let $G_{\R}$ be a connected real reductive linear Lie group and let
$V$ be an irreducible Harish-Chandra $(\lie{g}, K_{\R})$-module. 
Let $\psi$ be a non-degenerate unitary character on $\lier{n}$. 
\begin{enumerate}
\item
$\Wh_{\psi}^{\infty}(V)$ is non-zero if and only if 
$\psi \in \WF(V)$. 
\item
If $V$ is the Harish-Chandra module of a discrete series
representation and $\psi \in \WF(V)$, then 
\[
\dim \Wh_{\psi}^{\infty}(V) 
= 
\frac{\# \mathcal{P}(G_{\R})}{\# W_{G_{\R}}} \Deg V, 
\]
where $W_{G_{\R}}$ is the little Weyl group of $G_{\R}$. 
\end{enumerate}
\end{theorem}

\subsection{Chang's results}
\label{subsection:Chang's results}

Let us recall J. T. Chang's papers \cite{C1}, \cite{C2}, 
in which associated cycles of discrete series are determined for some
cases. 

Let $V$ be a Harish-Chandra $(\lie{g},K_\R)$-module. 
Choose a $K_{\R}$-stable finite dimensional generating subspace
$V_{0}$ of $V$, and define a filtration 
$V_{n}:=U(\lie{g})_{n} V_{0}$ of $V$. 
Here $\{U(\lie{g})_{n}| n = 0, 1, 2, \dots\}$ is the standard
filtration of $U(\lie{g})$. 
By this filtration, $M := \gr V$ admits 
an $S(\lie{g}) \simeq \gr U(\lie{g})$-module structure. 
The associated variety $\AssVar(V)$ of $V$ 
is the support $\Supp M \subset \lie{g}^{\ast}$ of $M$. 
Note that $\AssVar(V)$ is a closure of finite $K$-orbits on $\lie{p}$. 
Let $X_{1}, \dots, X_{k}$ be the irreducible components of
$\AssVar(V)$, and let $m_{i}$ be the length of
$S(\lie{g})_{\mathfrak{P}_{i}}$ module $M_{\mathfrak{P}_{i}}$, where
$\mathfrak{P}_{i}$ is the minimal prime ideal corresponding to the
irreducible variety $X_{i}$. 
The associated cycle $\AssCyc(V)$ of $V$ is the formal sum 
$\sum_{i} m_{i} X_{i}$.

Assume that $G_{\R}$ has discrete series. 
Then $\lie{g}$ has a compact Cartan subalgebra 
$\lie{t} \subset \lie{k}$. 
Let $\lie{b}=\lie{t}+\bar{\lie{u}}$ be a Borel subalgebra of
$\lie{g}$, and let $B\subset G$ be the corresponding Borel subgroup. 
This Borel determines a positive root system $\Delta^{+}$ of
$\roots{\lie{g}}{\lie{t}}$ so that $\bar{\lie{u}}$ corresponds to
negative roots. 
Let $\Delta_c^+\subset \Delta^+$ be the set of compact positive roots
and denote 
$\rho=\frac{1}{2}\sum_{\alpha \in \Delta^+}\alpha$, 
$\rho_c=\frac{1}{2}\sum_{\alpha \in \Delta_c^+}\alpha$, 
$\rho_n=\rho-\rho_c$. 
We may and do regard $\lie{b}$ as a point on the flag variety 
$X := G/B$.  
By the choice of $\lie{b}$, 
the $K$-orbit $Z := K \cdot \lie{b} \simeq K / K \cap B$ is closed. 
Here, ``dot'' means the adjoint action. 

Suppose $\Lambda \in \Xi^{+}$ is $\Delta^{+}$-dominant. 
Let $\tau$ be the character of $T := \exp \lie{t}$ with  
$d\tau=\Lambda-\rho$. 
This character gives rise to a $K$-homogeneous line bundle on $Z$,
and we denote by $\mathcal{L}_{\Lambda - \rho}$ its
sheaf of local sections on $Z$. 
Let $j : Z \to X$ be the embedding map, 
and let $j_{+}(\mathcal{L}_{\Lambda-\rho})$ be the direct image in
the category of $\mathcal{D}$-modules. 
It is well known that the global section space 
$\varGamma(X, j_{+}(\mathcal{L}_{\Lambda-\rho}))$ realizes the
Harish-Chandra module of $\pi_{\Lambda}$. 
Note that the lowest $K_{\R}$-type sheaf for $\pi_{\Lambda}$ is 
$\mathcal{L}_{\Lambda-\rho} \otimes \det \mathcal{N}_{Z|X} 
\simeq \mathcal{L}_{\Lambda-\rho_{c}+\rho_{n}}$, 
where $\mathcal{N}_{Z|X}$ is the normal sheaf of $Z$ in $X$. 

Let $T^{\ast} X \simeq G \times_{B} (\lie{g}/\lie{b})^{\ast} 
\simeq 
G \times_{B} \bar{\lie{u}}$ be the cotangent bundle on $X$. 
Here, we identify $(\lie{g}/\lie{b})^{\ast}$ with 
$\bar{\lie{u}}$ by a fixed invariant bilinear form on $\lie{g}$. 
By this identification and 
$T Z \simeq K \times_{K \cap B} (\lie{k} / \lie{k} \cap \lie{b})$, 
the conormal bundle $T_{Z}^{\ast} X$ is isomorphic to 
$K \times_{K \cap B} (\bar{\lie{u}} \cap \lie{p})$. 
For each point $x \in X$, 
representing a Borel subalgebra $\lie{b}_{x}$, 
the fiber $T_{x}^{\ast} X$ in $T^{\ast} X$ is given by 
$(\lie{g}/\lie{b}_{x})^{\ast} \subset \lie{g}^{\ast}$. 
The moment map 
$\gamma:T^* X\rightarrow \lie{g}^{\ast}$ is given by, on each fiber, 
the inclusion 
$(\lie{g}/\lie{b}_{x})^{\ast} \hookrightarrow \lie{g}^{\ast}$. 
For the discrete series $\pi_{\Lambda}$, the associated variety
$\AssVar(\pi_{\Lambda})$ is a closure of a single $K$-orbit on
$\lie{p}$, and it coincides with the moment map image 
$\gamma(T_{Z}^{\ast} X) \simeq K \cdot (\bar{\lie{u}} \cap \lie{p})$. 
Therefore, the associated cycle $\AssCyc(\pi_{\Lambda})$ of
$\pi_{\Lambda}$ is an integral multiple of $\gamma(T_{Z}^{\ast} X)$. 


%
\begin{theorem}[\cite{C2}]
Let $\xi$ be a generic point 
of $\bar{\lie{u}} \cap \lie{p} \subset \gamma(T_{Z}^{\ast} X)$. 
Then the associated cycle of the discrete series representation
$\pi_{\Lambda}$ is 
\[
\AssCyc(\pi_{\Lambda}) 
= 
\dim H^0(\gamma^{-1}(\xi), 
\mathcal{L}_{\Lambda - \rho_{c} + \rho_{n}}|_{\gamma^{-1}(\xi)})\,  
\gamma(T_{Z}^{\ast} X). 
\]
\end{theorem}

Under some condition, the explicit formula of this coefficient can be
obtained. 
Let $N_K(\xi,\bar{\lie{u}}\cap\lie{p}):=
\{k\in K| k\cdot \xi\in\bar{\lie{u}}\cap \lie{p}\}$.  
Then it is not hard to show that 
$\gamma^{-1}(\xi)\simeq 
N_K(\xi,\bar{\lie{u}}\cap\lie{p})^{-1}\cdot \lie{b} \subset Z$. 
Note that $N_K(\xi,\bar{\lie{u}}\cap\lie{p})$ is not a group.  
In order to calculate $N_K(\xi,\bar{\lie{u}}\cap\lie{p})$, 
let us consider the following groups. 
Let $S$ be the set of all compact simple roots, 
$\langle S \rangle$ the root system generated by $S$. 
Let $\lie{l}:=\lie{t}+\sum_{\alpha\in\langle S\rangle}\lie{g}_\alpha$,
and let $\lie{q}=\lie{l}+\bar{\lie{v}} \supset\lie{b}$ be the
parabolic subalgebra whose Levi part is $\lie{l}$. 
The analytic subgroups with Lie algebras $\lie{l}$, $\lie{q}$ are
denoted by $L$, $Q$ respectively. 
\begin{theorem}[\cite{C2}]
Let $K(\xi)$ be the centralizer of $\xi$ in $K$ and $K(\xi)_r^0$ be
the identity component of the reductive part of $K(\xi)$. 
If 
\begin{equation}\label{eq:Chang's hypothesis}
N_K(\xi,\bar{\lie{u}}\cap\lie{p})=(K\cap Q) K(\xi)
=(K\cap Q) K(\xi)_r^0,
\end{equation}
then 
\[H^0(\gamma^{-1}(\xi), 
\mathcal{L}_{\Lambda-\rho_c+\rho_n}|_{\gamma^{-1}(\xi)})
\simeq \mathrm{Coh}\mbox{-}\mathrm{Ind}
\uparrow_{K(\xi)_r^0\cap Q}^{K(\xi)_r^0} 
\mathrm{Res}\downarrow_{K(\xi)_r^0\cap Q}^{K\cap Q}
V_{\Lambda-\rho_c+\rho_n}^{K\cap Q}.\]
Here $\mathrm{Coh}\mbox{-}\mathrm{Ind}$ is the cohomological induction,
i.e. the operation which makes the irreducible representation of the
large group with the same highest weight as that of the small group. 
\end{theorem}

\subsection{$Spin(2n,2)$ case}\label{subsection:SO case}

Let us apply the general theory to our $Spin(2n, 2)$ case. 

\begin{proposition}\label{prop:GK dim}
\begin{enumerate}
\item
If $\Lambda \in \Xi_{1,\pm}$, then 
$\Dim(\pi_{\Lambda}) = 2n$. 
\item
If $\Lambda \in \Xi_{m,\pm}$, $m = 2, \dots, n$, 
then $\Dim(\pi_{\Lambda}) = 4n-2 = \dim \lier{n}$. 
\item
If $\Lambda \in \Xi_{n+1,\pm}$, then 
$\Dim(\pi_{\Lambda}) = 4n-3$. 
\end{enumerate}
\end{proposition}
\begin{proof}
Let $X_{\alpha} \in \lie{g}_{\alpha}$ be a non-zero root vector for 
$\alpha \in \Delta(\lie{g}, \lie{t})$. 
For each case, we can choose 
\[
\xi = 
\begin{cases}
(1) & X_{\pm e_{n} \mp e_{n+1}} + X_{\mp e_{n} \mp e_{n+1}}
\\
(2) & X_{- e_{1} \pm e_{n+1}} + X_{e_{n} \mp e_{n+1}} 
+ X_{- e_{n} \mp e_{n+1}}
\\
(3) & 
X_{\mp e_{n} \pm e_{n+1}} + X_{- e_{n-1} \mp e_{n+1}}
\end{cases}
\]
as a generic point of $\bar{\lie{u}} \cap \lie{p}$. 
It is not hard to calculate the centralizer $K(\xi)$, and finally we
obtain 
$\Dim(\pi_{\Lambda}) 
= 
\dim (K \cdot \xi) 
= 
\dim K - \dim K(\xi) 
= 2n$, $4n-2$ and $4n-3$, respectively. 
\end{proof}
\begin{corollary}\label{cor:existing cases}
The discrete series $\pi_{\Lambda}$ has a non-trivial algebraic
Whittaker model 
if and only if $\Lambda \in \Xi_{m,\pm}$, $m = 2, \dots, n$. 
\end{corollary} 

Chang has shown that the condition \eqref{eq:Chang's hypothesis} is
satisfied if $G_{\R}$ is a connected real rank one group. 
It is also satisfied in our case. 

\begin{proposition}
The condition \eqref{eq:Chang's hypothesis} is satisfied 
when $\Lambda \in \Xi_{m,\pm}$, $m = 2, \dots, n$. 
\end{proposition}
\begin{proof}
Suppose $\Lambda \in \Xi_{m, \pm}$, $m = 2, \dots, n$. 
In this case, 
\[
\bar{\lie{u}}\cap\lie{p} 
= 
\bigoplus_{i=1}^{m-1} 
(\lie{g}_{-e_i \pm e_{n+1}} \oplus \lie{g}_{-e_{i} \mp e_{n+1}}) 
\oplus 
\bigoplus_{i=m}^{n} 
(\lie{g}_{e_{i} \mp e_{n+1}} 
\oplus
\lie{g}_{-e_{i} \mp e_{n+1}}). 
\]
We choose a generic point $\xi$ of $\bar{\lie{n}}\cap\lie{p}$ 
as in the proof of Proposition~\ref{prop:GK dim}. 
For this $\xi$, 
\begin{align*}
K(\xi)_r&\simeq
\{\pm 1\}\times Spin(2n-3,\C), \\
K(\xi)&=
K(\xi)_r\exp
\left(\sum_{i=2}^{n-1}(a_i X_{-e_1+e_i}+b_i X_{-e_1-e_i})
+a_n(X_{-e_1+e_n}+X_{-e_1-e_n})\right), 
\end{align*}
where $X_{\alpha}$'s are appropriately chosen non-zero root vectors. 
For $\Lambda \in \Xi_{m,\pm}$, the parabolic subgroup $K\cap Q$
corresponds to the set of simple roots 
$\{e_{1} - e_{2}, \dots, e_{m-2} - e_{m-1}, 
e_{m} - e_{m+1}, \dots, e_{n-1}-e_n, e_{n-1}+e_n\}$. 
Let $Q_{1} \supset K \cap B$ be the parabolic subgroup of $K$, 
whose Levi subalgebra corresponds to the set of
simple roots 
$\{e_2-e_3, \dots, e_{n-1}-e_n, e_{n-1}+e_n\}$. 
Note that $Q_{1}$ contains $K(\xi)$. 
Let 
$k=qwq_{1}$ 
be the Bruhat decomposition of 
$k\in N_K(\xi,\bar{\lie{u}}\cap\lie{p})$ with respect to 
$(K \cap Q) \backslash K / Q_{1}$, 
where $q\in K\cap Q$, 
$w\in (\mathfrak{S}_{m-1}\times
(\mathfrak{S}_{n-m+1}\ltimes \Z_2^{n-m}))
\backslash \mathfrak{S}_n\ltimes \Z_2^{n-1}/
(\mathfrak{S}_{n-1}\ltimes \Z_2^{n-2})$, $q_{1}\in Q_{1}$.  
Since $K\cap Q$ normalizes $\bar{\lie{u}}\cap\lie{p}$, 
$k\cdot \xi \in \bar{\lie{u}}\cap \lie{p}$ is equivalent to  
$q_{1} \cdot \xi\in w^{-1}(\bar{\lie{u}}\cap \lie{p})$. 
By direct computation, this is true only if $w = e$, 
and we can show that, if  
$q_{1} \cdot \xi\in \bar{\lie{u}}\cap \lie{p}$, 
then 
$q_{1} \in (K\cap Q) K(\xi)$. 
Thus we get 
$N_K(\xi,\bar{\lie{u}}\cap \lie{p}) \subset 
(K\cap Q) K(\xi)$. 
Since the inverse inclusion is trivial, 
we get 
$N_K(\xi,\bar{\lie{u}}\cap \lie{p}) 
=
(K\cap Q) K(\xi)$. 
Finally, since each connected component of $K(\xi)$ meets 
$\exp\lie{t} \subset K \cap Q$ and the unipotent radical of $K(\xi)$ is
contained in $K \cap Q$, we know 
$N_K(\xi,\bar{\lie{u}}\cap \lie{p}) 
=
(K\cap Q) K(\xi) 
= 
(K\cap Q) K(\xi)_{r}^{0}$. 
\end{proof}

Let $V_{\nu}^{F}$ be the irreducible finite dimensional representation
of a reductive group $F$ with the highest weight $\nu$. 
For $\Lambda \in \Xi_{m, \pm}$, $m = 2, \dots, n$, 
let 
$\tilde{\lambda} = (\lambda; \lambda_{n+1}) 
= (\lambda_1,\dots,\lambda_n; \lambda_{n+1}) =\Lambda-\rho_{c}+\rho_n$ 
be the Blattner parameter of $\pi_{\Lambda}$. 
Since 
\begin{align*}
K(\xi)_r^0&\simeq 
Spin(2n-3,\C), \\
K\cap Q&\simeq
GL(m-1,\C)\times Spin(2(n-m+1),\C) \times Spin(2,\C) 
\\
& \qquad 
\ltimes \mbox{\rm (unipotent radical)}, \\
K(\xi)_r^0\cap Q&\simeq
GL(m-2,\C)\times Spin(2n-2m+1,\C) 
\ltimes \mbox{\rm (unipotent radical)}, 
\end{align*}
we have 
\begin{align*}
& V_{\tilde{\lambda}}^{K\cap Q}
=
V_{(\lambda_1,\dots,\lambda_{m-1})}^{GL(m-1,\C)}
\boxtimes 
V_{(\lambda_m,\dots,\lambda_n)}^{Spin(2(n-m+1),\C)}
\boxtimes 
V_{\lambda_{n+1}}^{Spin(2,\C)}, 
\\
& \mathrm{Res}\downarrow_{K(\xi)_r^0\cap Q}^{K\cap Q}
V_{\tilde{\lambda}}^{K\cap Q}
\\
& 
\qquad 
\simeq\bigoplus_{
\genfrac{}{}{0pt}{}{\lambda_1 \geq \mu_1 \geq \lambda_2 \geq \dots 
\geq \lambda_{m-2} \geq \mu_{m-2} \geq \lambda_{m-1}}
{\lambda_m \geq \mu_1' \geq \lambda_{m+1} \geq \dots 
\geq \lambda_{n-1} \geq \mu_{n-k}' \geq |\lambda_n|} }
V_{(\mu_1,\dots,\mu_{m-2})}^{GL(m-2,\C)}
\boxtimes 
V_{(\mu_1',\dots,\mu_{n-m}')}^{Spin(2n-2m+1,\C)}, 
\\
& 
\mathrm{Coh}\mbox{-}\mathrm{Ind}
\uparrow_{K(\xi)_r^0\cap Q}^{K(\xi)_r^0}
\mathrm{Res}\downarrow_{K(\xi)_r^0\cap Q}^{K\cap Q}
V_{\tilde{\lambda}}^{K\cap Q}
\\
& 
\qquad 
\simeq\bigoplus_{\genfrac{}{}{0pt}{}{
\lambda_1 \geq \mu_1 \geq \lambda_2 \geq \dots 
\geq \lambda_{m-2} \geq \mu_{m-2} \geq \lambda_{m-1}}
{\lambda_m \geq \mu_1' \geq \lambda_{m+1} \geq \dots 
\geq \lambda_{n-1} \geq \mu_{n-m}' \geq |\lambda_n|} }
V_{(\mu_1,\dots,\mu_{m-2},\mu_1',\dots,\mu_{n-m}')}^{Spin(2n-3,\C)}.
\end{align*}

\begin{theorem}\label{thm:Bernstein}
The Bernstein degree of $\pi_{\Lambda}$, 
$\Lambda \in \Xi_{m,\pm}$, $m=2,\dots,n$, 
is 
\[C \sum_{\genfrac{}{}{0pt}{}{
\lambda_1 \geq \mu_1 \geq \lambda_2 \geq \dots 
\geq \lambda_{m-2} \geq \mu_{m-2} \geq \lambda_{m-1}}
{\lambda_m \geq \mu_1' \geq \lambda_{m+1} \geq \dots 
\geq \lambda_{n-1} \geq \mu_{n-m}' \geq |\lambda_n|} }
\dim
V_{(\mu_1,\dots,\mu_{m-2},\mu_1',\dots,\mu_{n-m}')}^{Spin(2n-3,\C)}, 
\]
where $C$ is a general constant independent of $\Lambda$. 
\end{theorem}

\subsection{Realization of Whittaker functions}
\label{subsection:Realization of Whittaker functions}

Embedding of a discrete series into an induced representation is
realized by gradient type differential-difference equations. 
We review Yamashita's results (cf. \cite{Yamashita}). 

Let $\eta$ be a non-degenerate unitary character of $N_{\R}$. 
For a finite dimensional representation $(\tau, V_{\tau})$ of
$K_{\R}$, define 
\begin{align*}
& 
C^{\infty}_{\tau}(K_{\R} \backslash G_{\R} /N_{\R}; \eta)
\\
&:=\{F: G_{\R} \overset{C^{\infty}}{\longrightarrow} V_{\tau}| 
F(kgn)=\eta(n)^{-1}\tau(k) F(g), 
\mbox{ for } n\in N_{\R}, g\in G_{\R}, k \in K_{\R}\}. 
\end{align*}
Let, as before, $\tilde{\lambda}=\Lambda-\rho_{c}+\rho_{n}$ be the
Blattner parameter corresponding to a Harish-Chandra parameter
$\Lambda$. 
Let $(\tau_{\tilde{\lambda}}, V_{\tilde{\lambda}})$ be the irreducible
finite dimensional representation of $K_{\R}$ with the highest weight
$\tilde{\lambda}$. 
Let $(\Ad, \lie{p})$ be the adjoint representation of $K_{\R}$ on
$\lie{p}$. 

Fix an invariant bilinear form $\langle \enskip, \enskip \rangle$ on
$\lier{g}$ and choose an orthonormal basis $\{X_{i}\}$ of $\lier{p}$. 
Define a differential-difference operator 
$\nabla_{\tilde{\lambda}, \eta}$ by 
\begin{align*}
& \nabla_{\tilde{\lambda}, \eta}:
C^\infty_{\tau_{\tilde{\lambda}}}(K_\R\backslash G_\R /N_{\R}; \eta)
\to 
C^\infty_{\tau_{\tilde{\lambda}} \otimes \Ad}
(K_\R\backslash G_\R /N_{\R}; \eta), 
\\
& \nabla_{\tilde{\lambda}, \eta} \phi(g)
:=
\sum_i L_{X_i} \phi(g)\otimes X_i. 
\end{align*}
Here, $L_{X_i}$ is the left translation. 

Let $\Delta_{n}^{+}$ be the set of non-compact roots $\alpha$ 
with $\langle \alpha, \Lambda \rangle > 0$. 
Then the irreducible decomposition of 
$\tau_{\tilde{\lambda}} \otimes \Ad$ is 
$\oplus_{\alpha\in \Delta_n^+ \cup (-\Delta_{n}^{+})}
m_\alpha \tau_{\tilde{\lambda} + \alpha}$, $m_\alpha\in\{0, 1\}$. 
Let 
$\tau_{\tilde{\lambda}}^-
:=
\oplus_{\alpha\in\Delta_n^+}
m_{-\alpha} \tau_{\tilde{\lambda} - \alpha}$ 
be the negative part and 
let 
$\pr^-: \tau_{\tilde{\lambda}} \otimes \Ad 
\rightarrow \tau_{\tilde{\lambda}}^-$ be the
natural projection. 
Define a differential-difference operator 
$\mathcal{D}_{{\tilde{\lambda}}, \eta}$ by 
\[
\mathcal{D}_{\tilde{\lambda}, \eta}
:=\pr^-\circ \nabla_{\tilde{\lambda}, \eta}
 : C^\infty_{\tau_{\tilde{\lambda}}}(K_\R\backslash G_\R / N_{\R}; \eta)
\rightarrow 
C^\infty_{\tau_{\tilde{\lambda}}^-}(K_\R\backslash G_\R / N_{\R}; \eta).
\]
\begin{theorem}[\cite{Yamashita}] 
Let $\pi_{\Lambda}^{\ast}$ be the dual Harish-Chandra module of
$\pi_{\Lambda}$. 
If the Blattner $\tilde{\lambda}$ of $\pi_{\Lambda}$ is far from the
  walls, then 
\[
\Hom_{\lie{g}, K_\R}(\pi_{\Lambda}^*, 
C^\infty(G_\R/N_{\R}; \eta))
\simeq 
\Ker(\mathcal{D}_{\tilde{\lambda}, \eta}).
\]
\end{theorem}

\begin{remark}\label{rem:contragredient}
Suppose $G_\R=Spin(2n,2)$ and $\Lambda=\sum_{i=1}^{n+1}\Lambda_i e_i$. 
The Harish-Chandra parameter of 
$\pi_{\Lambda}^{\ast}$ is 
$\Lambda^{\ast}
:=
\sum_{i=1}^{n-1} \Lambda_{i} e_{i} + (-1)^n\Lambda_n e_n 
-\Lambda_{n+1} e_{n+1}$. 
Therefore, if $\Lambda \in \Xi_{m,\pm}$ and $m \leq n$, 
then $\Lambda^{\ast} \in \Xi_{m,\mp}$. 
The Blattner parameter $\tilde{\lambda}$ corresponding to 
$\Lambda = \sum_{i=1}^{n+1} \Lambda_{i} e_{i} \in \Xi_{m,\pm}$, 
$m \leq n$, is 
\[
\tilde{\lambda} 
= 
\sum_{i=1}^{m-1} (\Lambda_{i} - n + i + 1) e_{i} 
+ 
\sum_{i=m}^{n} (\Lambda_{i} - n + i) e_{i} 
+(\Lambda_{n+1} \pm (n-m+1)) e_{n+1}. 
\]
Especially, the $n$-th components $\lambda_{n}$, $\lambda_{n}^{\ast}$
of the Blattner parameters of $\pi_{\Lambda}$, $\pi_{\Lambda}^{\ast}$
are $\Lambda_{n}$, $(-1)^{n} \Lambda_{n}$, respectively. 
Therefore, if $\Lambda \in \Xi_{m,\pm}$, $m = 2, \dots, n$, 
then the Bernstein degrees of $\pi_{\Lambda}$ and
$\pi_{\Lambda}^{\ast}$ are identical because of
Theorem~\ref{thm:Bernstein}. 
\end{remark}
 
\section{Radial $A_{\R}$ part of 
$\mathcal{D}_{\tilde{\lambda}, \eta}$}
\label{section:Radial A part}

In this section, we write down the differential-difference equation
$\mathcal{D}_{\tilde{\lambda}, \eta} \phi =0$ explicitly. 
After that, we reduce our computation to getting a coefficient function
of some special vector.

\subsection{Irreducible decomposition of tensor product
  representation}
\label{section:3.4}

The Lie algebra $\lie{k}$ is isomorphic to 
$\lie{so}(2n, \C) \oplus \lie{so}(2, \C)$ and, as a vector space,
$\lie{p}$ is isomorphic to the matrix space $M_{2n,2}(\C)$. 
The adjoint representation  $(\ad, \lie{p})$ of $\lie{k}$ on $\lie{p}$
is 
\begin{align*}
& 
\lie{so}(2n,\C) \oplus \lie{so}(2,\C) 
\curvearrowright 
M_{2n,2}(\C), 
\\
& 
(A, B) \cdot X 
= AX-XB, 
\quad 
\mbox{for } A \in \lie{so}(2n,\C), B \in \lie{so}(2,\C), 
X \in M_{2n,2}(\C).  
\end{align*}

Let $(\tau^{k}, \C^{k})$ be the natural representation of 
$\lie{so}(k,\C)$ on $\C^{k}$. 
By the identification $\C^{2n} \otimes \C^{2} \simeq M_{2n,2}(\C)$, 
$u \otimes v \mapsto u {}^{t}v$, 
\[
(\ad, \lie{p}) 
\simeq 
(\tau^{2n}, \C^{2n}) \boxtimes (\tau^{2}, \C^{2}).
\]
Note that $\I F_{2n+i,j} \in \lie{p}$ corresponds to the vector 
$v_{j}^{2n} \boxtimes v_{i}^{2} \in 
\tau^{2n} \boxtimes \tau^{2}$, 
where 
$v_{i}^{k} 
= 
{}^{t} (0, \dots 0, \overset{i}{1}, 0, \dots, 0)$ 
is a standard basis of $\C^{k}$.

We realize the representation $\tau_{\tilde{\lambda}}$ of $K_{\R}$ by
using the Gelfand-Tsetlin basis. 

\begin{definition}
Let $\lambda = (\lambda_{1}, \dots, \lambda_{n})$ be a dominant
integral weight of $Spin(2n)$. 
A {\it ($\lambda$-)Gelfand-Tsetlin pattern} is a set of vectors 
$Q=(\vect{q}_{1}, \dots, \vect{q}_{2n-1})$ such that 
\begin{enumerate}
\item
$\vect{q}_i=(q_{i,1}, q_{i,2}, \dots, q_{i,\lfloor(i+1)/2\rfloor})$. 
\item
The numbers $q_{i,j}$ are all integers or all half integers. 
\item
$q_{2i+1,j} \geq q_{2i,j} \geq q_{2i+1,j+1}$, 
for any $j=1, \dots, i-1$.  
\item
$q_{2i+1,i} \geq q_{2i,i} \geq |q_{2i+1,i+1}|$.  
\item
$q_{2i,j} \geq q_{2i-1,j} \geq q_{2i,j+1}$, 
for any $j=1, \dots, i-1$. 
\item
$q_{2i,i} \geq q_{2i-1,i} \geq -q_{2i,i}$.  
\item
$q_{2n-1,j}=\lambda_{j}$. 
\end{enumerate}
Here, $\lfloor a \rfloor$ is the largest integer not greater than
$a$. 
The set of all $\lambda$-Gelfand-Tsetlin patterns is denoted by
$GT(\lambda)$. 
\end{definition}
\begin{notation}
For any set or number $\ast$ depending on $Q \in GT(\lambda)$, 
we denote it by $\ast(Q)$, if we need to specify $Q$. 
For example, $q_{i,j}(Q)$ is the $q_{i,j}$ part of $Q \in GT(\lambda)$. 
\end{notation}
\begin{theorem}[\cite{GT}] 
Let $\lambda$ be a dominant integral weight of $Spin(2n)$ and
let $(\tau_{\lambda}, V_{\lambda}^{Spin(2n)})$ be the irreducible
representation of $Spin(2n)$ with the highest weight $\lambda$. 
Then $GT(\lambda)$ is a basis of 
$(\tau_{\lambda}, V_{\lambda}^{Spin(2n)})$. 
\end{theorem}
The action of elements $F_{p,q} \in \lie{so}(2n)$ is expressed as follows. 
For $j > 0$, let 
\begin{align*}
& 
l_{2i-1,j} := q_{2i-1,j} + i - j, 
& 
& 
l_{2i-1,-j} := - l_{2i-1,j}, 
\\
& 
l_{2i,j} := q_{2i,j} + i + 1 - j, 
& 
& 
l_{2i,-j} := - l_{2i,j} + 1,  
\end{align*}
and let $l_{2i,0} = 0$. 
Define $a_{p,q}(Q)$ by 
\begin{align*} 
a_{2i-1,j}(Q) 
&= 
\sgn j \, 
\sqrt{-
\frac{\prod_{1 \leq |k| \leq i-1}(l_{2i-1,j} + l_{2i-2, k}) 
\prod_{1 \leq |k| \leq i} (l_{2i-1,j} + l_{2i, k})}
{4 \prod_{\genfrac{}{}{0pt}{}{1 \leq |k| \leq i,}{k \not= \pm j}} 
(l_{2i-1,j} + l_{2i-1,k}) (l_{2i-1,j} + l_{2i-1,k} + 1)}
},
\intertext{for $j = \pm 1, \dots, \pm i$, and}  
a_{2i,j}(Q) 
&= 
\epsilon_{2i,j}(Q) 
\sqrt{-
\frac{\prod_{1 \leq |k| \leq i}(l_{2i,j} + l_{2i-1, k}) 
\prod_{1 \leq |k| \leq i+1} (l_{2i,j} + l_{2i+1, k})}
{(4 l_{2i,j}^{2} - 1) 
\prod_{\genfrac{}{}{0pt}{}{0 \leq |k| \leq i}{k \not= \pm j}} 
(l_{2i,j} + l_{2i,k}) (l_{2i,j} - l_{2i,k})}
},
\end{align*}
for $j = 0, \pm 1, \dots, \pm i$, 
where $\epsilon_{2i,j}(Q)$ is $\sgn j$ if $j \not= 0$, 
and $\sgn(q_{2i-1,i} q_{2i+1,i+1})$ if $j = 0$. 

Let $\sigma_{a,b}$ be the shift operator, sending $\vect{q}_a$ to 
$\vect{q}_{a} + (0, \dots, \overset{|b|}{\sgn(b)}, 0 ,\dots, 0)$. 
For notational convenience, we often write 
$\tau_{i,j} := \sigma_{2n-3,i} \sigma_{2n-2,j}$. 

\begin{theorem}[\cite{GT}]
Under the above notation, the Lie algebra action is expressed as 
\begin{align*}
\tau_\lambda(F_{2i+1,2i})Q
&=
\sum_{1 \leq |j| \leq i} a_{2i-1,j}(Q) \sigma_{2i-1,j}Q, 
\\
\tau_\lambda(F_{2i+2,2i+1})Q
&=
\sum_{0 \leq |j| \leq i} a_{2i,j}(Q)\sigma_{2i,j}Q. 
\end{align*}
\end{theorem}

In the following of this paper, we assume that the Blattner parameter
$\tilde{\lambda}$ is far from the walls. 
It follows that the numbers $\lambda_{j}$ satisfy  
\begin{equation}\label{eq:order of lambda}
\lambda_{1} > \lambda_{2} > \dots > \lambda_{n-1} 
> |\lambda_{n}| 
\end{equation}
and the differences between adjacent numbers are sufficiently large. 

Let $e_{k} = (0, \dots, \overset{k}{1}, 0, \dots, 0) \in \C^{n}$ and 
$e_{-k} := -e_{k}$. 
Let 
\begin{align*}
\pr_{k}:& V_\lambda^{Spin(2n)} \otimes \C^{2n} 
\to 
V_{\lambda+e_k}^{Spin(2n)}, \quad \mbox{for} \quad 
k = \pm 1, \dots, \pm n 
\end{align*}
be the projection operator along the irreducible decomposition 
\begin{align*}
V_\lambda^{Spin(2n)} \otimes \C^{2n} 
&\simeq
\bigoplus_{1 \leq |k| \leq n} 
V_{\lambda+e_k}^{Spin(2n)}. 
\end{align*}
In order to describe this operator explicitly, 
we identify $Q \in GT(\lambda)$ with the Gelfand-Tsetlin pattern
$(Q, \vect{q}_{2n})$ of a representation of $Spin(2n+1)$, 
where 
$\vect{q}_{2n}:=(\lambda_1+1,\dots,\lambda_{n-1}+1,|\lambda_n|+1)$. 
Just the same way as in the proof of \cite[Proposition~4.3]{Kr} 
and \cite[Lemma~3.1.3]{T}, we get the following formulas. 
\begin{lemma}\label{lemma:projection-1}
For $Q \in GT(\lambda)$ and $k = \pm 1, \dots, \pm n$, 
\begin{align}
& 
\pr_{k}(Q \otimes v_{2n}^{2n}) 
= 
a_{2n-1,k}(Q) \sigma_{2n-1,k} Q, 
\label{eq:tensor v_{2n}}
\\
& 
\pr_{k}(Q \otimes v_{2n-1}^{2n}) 
= 
\sum_{0 \leq |j| \leq n-1} 
\frac{a_{2n-2,j}(Q) a_{2n-1,k}(\tau_{0,j}Q)}
{l_{2n-2,j} - l_{2n-1,k}} 
\tau_{0,j} \sigma_{2n-1,k} Q, 
\label{eq:tensor v_{2n-1}}
\\
& 
\pr_{k}(Q \otimes v_{2n-2}^{2n}) 
\label{eq:tensor v_{2n-2}}
\\
& 
= 
\sum_{1 \leq |i| \leq n-1} \sum_{0 \leq |j| \leq n-1} 
\frac{a_{2n-3,i}(Q) a_{2n-2,j}(\tau_{i,0} Q) 
a_{2n-1,k}(\tau_{i,j} Q)}
{(l_{2n-3,i}-l_{2n-2,j}+1) 
(l_{2n-2,j} - l_{2n-1,k})} 
\tau_{i,j} \sigma_{2n-1,k} Q. 
\notag
\end{align}
\end{lemma}
\begin{remark}\label{remark:null points of a_ij}
For $Q \in GT(\lambda)$, it is not hard to see that 
$a_{2i-1,j}(Q) = 0$ if and only if 
$\tau_{0,j}Q \not \in GT(\lambda)$,  
and that 
$a_{2i,j}(Q) = 0$ if and only if 
either 
(i) $j \not= 0$ and $\sigma_{2i,j} Q \not\in GT(\lambda)$ 
or 
(ii) $j = 0$ and $q_{2i-1,i}$ or $q_{2i+1,i+1} = 0$. 
Moreover, 
\begin{enumerate}
\item 
if $j, k \not= 0$, the coefficient of 
$\tau_{0,j} \sigma_{2n-1,k} Q$ in
\eqref{eq:tensor v_{2n-1}} is non-zero if and only if 
$\tau_{0,j} \sigma_{2n-1,k} Q \in GT(\lambda + e_{k})$, 
and 
\item
if $i, j, k \not= 0$, the coefficient 
of $\tau_{i,j} \sigma_{2n-1,k} Q$ in
\eqref{eq:tensor v_{2n-2}} is non-zero if and only if 
$\tau_{i,j} \sigma_{2n-1,k} Q 
\in GT(\lambda + e_{k})$. 
\end{enumerate}
\end{remark}

We know that 
$V_{\tilde{\lambda}}^{K_{\R}}  
= V_{\lambda}^{Spin(2n)} \boxtimes V_{\lambda_{n+1}}^{Spin(2)}$ and 
$\lie{p} 
\simeq 
\C^{2n} \boxtimes \C^{2}$. 
Therefore, the irreducible decomposition of 
$V_{\tilde{\lambda}}^{K_{\R}} 
\otimes 
\lie{p}$ is 
\begin{equation}\label{eq:decomposition of V*p}
V_{\tilde{\lambda}}^{K_{\R}} 
\otimes 
\lie{p} 
\simeq 
\bigoplus_{1 \leq |k| \leq n} 
V_{\lambda + e_{k}}^{Spin(2n)} 
\boxtimes 
V_{\lambda_{n+1}+1}^{Spin(2)} 
\oplus 
\bigoplus_{1 \leq |k| \leq n} 
V_{\lambda + e_{k}}^{Spin(2n)} 
\boxtimes 
V_{\lambda_{n+1}-1}^{Spin(2)}. 
\end{equation}

Since the irreducible representations of $Spin(2)$ are one
dimensional, we identify $Q \in GT(\lambda)$ with a vector of
$V_{\tilde{\lambda}}^{K_{\R}}$, on which $T_{n+1}$ acts by the scalar
$\lambda_{n+1}$. 
Such an identification is also applied 
to $V_{\lambda + e_{k}}^{Spin(2n)} 
\boxtimes 
V_{\lambda_{n+1} \pm 1}^{Spin(2)}$. 
More precisely, for $Q \in GT(\lambda)$, 
regard 
$Q_{k}^{\pm} 
:= 
\sigma_{2n-1,k} Q$ as a vector in 
$V_{\lambda + e_{k}}^{Spin(2n)} 
\boxtimes 
V_{\lambda_{n+1} \pm 1}^{Spin(2)}$, on which $T_{n+1}$ acts by the
scalar $\lambda_{n+1} \pm 1$.  

Let $\pr_{k,\pm} = \pr_{k} \boxtimes \pr_{\pm}$ be the projection
operator from 
$V_{\tilde{\lambda}}^{K_{\R}} 
\otimes 
\lie{p}$ 
to 
$V_{\lambda + e_{k}}^{Spin(2n)} 
\boxtimes 
V_{\lambda_{n+1}\pm 1}^{Spin(2)}$ along the irreducible decomposition
\eqref{eq:decomposition of V*p}. 
Then, by normalizing vectors appropriately, we have the following
explicit formulas from Lemma~\ref{lemma:projection-1}. 
\begin{lemma}\label{lemma:projection-2}
For $Q \in GT(\lambda)$ and $k = \pm 1, \dots, \pm n$, 
\begin{align*}
& 
\pr_{k,\pm}(Q \otimes (v_{2n}^{2n} \boxtimes v_{2}^{2})) 
= 
a_{2n-1,k}(Q) Q_{k}^{\pm}, 
\\
& 
\pr_{k,\pm}(Q \otimes (v_{2n-1}^{2n} \boxtimes v_{2}^{2})) 
= 
\sum_{0 \leq |j| \leq n-1}
\frac{a_{2n-2,j}(Q) a_{2n-1,k}(\tau_{0,j} Q)}
{l_{2n-2,j} - l_{2n-1,k}}  
\tau_{0,j} Q_{k}^{\pm}, 
\\
&
\pr_{k,\pm}(Q \otimes (v_{2n-2}^{2n} \boxtimes v_{2}^{2})) 
\\
& 
\qquad 
=
\sum_{1 \leq |i| \leq n-1}
\sum_{0 \leq |j| \leq n-1}
\frac{a_{2n-3,i}(Q) a_{2n-2, j}(\tau_{i,0} Q) 
a_{2n-1,k}(\tau_{i,j} Q)}
{(l_{2n-3,i} - l_{2n-2,j} + 1)
(l_{2n-2,j} - l_{2n-1,k})}
\tau_{i,j} Q_{k}^{\pm}, 
\\
& 
\pr_{k,\pm}(Q \otimes (v_{j}^{2n} \boxtimes v_{1}^{2})) 
= 
\mp \I 
\pr_{k,\pm}(Q \otimes (v_{j}^{2n} \boxtimes v_{2}^{2})) 
\quad \mbox{for} \quad j = 2n, 2n-1, 2n-2. 
\end{align*}
\end{lemma}

The actions of Casimir elements of $\lie{so}(k) \subset \lie{so}(2n)$ 
on the Gelfand-Tsetlin bases are as follows. 
Let $C_{k} 
:= 
\sum_{1 \leq j < i \leq k} F_{ij}^{2}$. 
Since this is a constant multiple of the Casimir element 
of $\lie{so}(k)$, it acts on an irreducible
representation of $\lie{so}(k)$ by a scalar. 
Especially it acts on $Q \in GT(\lambda)$ by a scalar, 
since $Q$ is contained in the $V_{\vect{q}_{k-1}}^{Spin(k)}$-isotropic
  subspace. 
This scalar is calculated in \cite[\S 5.1]{T}. 
\begin{lemma}\label{lemma:Casimir}
Let $2 \rho_{k} := (k-2, k-4, \dots, k-2 \lfloor k/2 \rfloor)$. 
For $Q \in GT(\lambda)$ and $k = 2, \dots, 2n$, 
\[
\tau_{\lambda}(C_{k}) Q 
= 
-(|\vect{q}_{k-1}|^{2} + 2 \langle \vect{q}_{k-1}, \rho_{k} \rangle)
Q. 
\]
\end{lemma}

\subsection{Differential-difference equation 
$\mathcal{D}_{\tilde{\lambda}, \eta} \phi = 0$}

Under some appropriate normalization of the invariant bilinear form
$\langle \enskip, \enskip \rangle$ on $\lier{g}$, 
$\{\I F_{2n+i,j} | 1 \leq i \leq 2, 1 \leq j \leq 2n\}$ forms an
orthonormal basis of $\lier{p}$. 
The Iwasawa decompositions of these vectors are given by 
\begin{align*}
& 
\I F_{2n+i,k} = X_{f_{i}}^{k} - F_{2n-2+i,k}, 
\quad \mbox{for} \quad  
1 \leq i \leq 2, 1\leq k \leq 2n-2, 
\\
& 
\I F_{2n+i, 2n-2+i} = A_{i}, 
\quad \mbox{for} \quad  
1 \leq i \leq 2, 
\\
& 
\I F_{2n+1, 2n} 
= 
\frac{1}{2}(X_{-f_{1}+f_{2}} - X_{f_{1}+f_{2}}) - \I T_{n+1},
& 
& 
\\
& 
\I F_{2n+2, 2n-1} 
= 
\frac{1}{2}(X_{-f_{1}+f_{2}} + X_{f_{1}+f_{2}}) - F_{2n,2n-1}. 
& 
& 
\end{align*}

Let $M_{\R}$ be the centralizer of $A_{\R}$ in  $K_\R$. 
Since $M_{\R}$ acts on $\lier{n}/[\lier{n}, \lier{n}]$, it acts on the
set of unitary characters of $N_{\R}$. 
Therefore, when we calculate 
$\Ker(\mathcal{D}_{\tilde{\lambda}, \eta})$, 
we may choose ``manageable'' non-degenerate unitary character $\eta$ in
its $M_{\R}$-orbit. 
Let $M_{\R}(\eta)$ be the centralizer of $\eta$ in $M_{\R}$. 
Since $M_{\R}(\eta)$ acts on 
$\Ker(\mathcal{D}_{\tilde{\lambda}, \eta})$ by the right translation,
the space of Whittaker models has $M_{\R}(\eta)$-module structure. 

As a ``manageable'' character, we choose $\eta$ satisfying 
\begin{align}
& 
\eta(X_{f_{i}}^{k}) 
= 
\left\{
\begin{matrix} 
\I \eta_{1} & \mbox{if $(i,k) = (1, 2n-2)$,}
\\
0 & \mbox{else,}
\end{matrix}
\right. 
&
&
\eta(X_{-f_{1}+f_{2}}) 
= 
\I \eta_{2}, 
&
& 
\eta(X_{f_{1}+f_{2}}) = 0 
\label{eq:character}
\end{align}
with $\eta_{1} > 0, \eta_{2} \not= 0$. 
Here, we denote the differential of $\eta$ by the same symbol $\eta$.

Because of the Iwasawa decomposition, an element of 
$C^\infty_{\tau_{\tilde{\lambda}}}(K_\R\backslash G_\R /N_{\R}; \eta)$
is determined by its restriction to $A_{\R}$. 
Thus, we consider the restriction of  
$\phi 
\in 
C^\infty_{\tau_{\tilde{\lambda}}}(K_\R\backslash G_\R /N_{\R}; \eta)$ 
to $A_{\R}$. 

Introduce a coordinate in $A_{\R}$ by  
\[
(\R_{>0})^{2} \ni 
(a_{1}, a_{2}) \mapsto 
\exp ((\log a_{1}) A_{1} + (\log a_{2}) A_{2}) \in A_{\R}
\]
and define $\d_{i} := a_{i} \d/\d a_{i}$. 
The left action of Lie algebra elements on $\phi$ is as follows: 
\begin{align*}
& \mbox{If } 
X \in \lier{k}, \mbox{ then} 
& 
& 
L_{X} \phi(a) 
= 
\left. 
\frac{d}{dt}
\right|_{t=0} 
\phi(\exp(-t X) a) 
= 
- d \tau_{\tilde{\lambda}}(X) \phi(a), 
\\
& 
\mbox{if } 
X \in \lier{n}, \mbox{ then}
& 
& 
L_{X} \phi(a) 
= 
\left. 
\frac{d}{dt}
\right|_{t=0} 
\phi(a \exp(-t \Ad(a^{-1}) X)) 
= 
\eta(\Ad(a^{-1})X) \phi(a), 
\\
& 
\mbox{and for } A_{i}, 
& 
& 
L_{A_{i}} \phi(a) 
= 
\left. 
\frac{d}{dt}
\right|_{t=0} 
\phi(\exp(-t A_{i}) a) 
= 
- \d_{i} \phi(a). 
\end{align*}
Since $\{\I F_{2n+i,j} | 1 \leq i \leq 2, 1 \leq j \leq 2n\}$ is an
orthonormal basis of $\lier{p}$, 
\begin{align*}
& \hspace{-5mm} 
\nabla_{\tilde{\lambda}, \eta} \phi 
\\
=& 
\sum_{i=1}^{2} \sum_{j=1}^{2n} 
L_{\I F_{2n+i,j}} \phi \otimes \I F_{2n+i,j} 
\\
=& 
\sum_{i=1}^{2} \sum_{k=1}^{2n-2} 
L_{X_{f_{i}}^{k} - F_{2n-2+i,k}} \phi \otimes \I F_{2n+i,k} 
+ 
L_{(X_{-f_{1}+f_{2}} - X_{f_{1}+f_{2}})/2 - \I T_{n+1}} 
\phi \otimes \I F_{2n+1,2n} 
\\
& 
+ 
\sum_{i=1}^{2} 
L_{A_{i}} \phi \otimes \I F_{2n+i,2n-2+i} 
+ 
L_{(X_{-f_{1}+f_{2}} + X_{f_{1}+f_{2}})/2 - F_{2n,2n-1}}
\phi \otimes \I F_{2n+2,2n-1} 
\\
=& 
- 
\sum_{i=1}^{2} 
\d_{i} \phi \otimes \I F_{2n+i,2n-2+i} 
+ 
\I \frac{\eta_{1}}{a_{1}} \phi \otimes \I F_{2n+1,2n-2} 
\\
& 
+ 
\I \frac{\eta_{2} a_{1}}{2 a_{2}} 
(\phi \otimes \I F_{2n+2,2n-1} 
+ 
\phi \otimes \I F_{2n+1,2n}) 
\\
& 
+ 
\sum_{i=1}^{2} \sum_{k=1}^{2n-3+i} 
d \tau_{\tilde{\lambda}}(F_{2n-2+i,k}) \phi \otimes \I F_{2n+i,k} 
+ 
\I d \tau_{\tilde{\lambda}}(T_{n+1}) \phi \otimes \I F_{2n+1,2n}.  
\end{align*}
Since
\begin{align*}
& (d \tau_{\tilde{\lambda}} \otimes \ad)(F_{2n-2+i,k})^{2} 
(v \otimes \I F_{2n+i,2n-2+i}) 
\\
&= 
d \tau_{\tilde{\lambda}}(F_{2n-2+i,k})^{2} v \otimes \I F_{2n+i,2n-2+i} 
\\
& \qquad 
- 2 
d \tau_{\tilde{\lambda}}(F_{2n-2+i,k}) v \otimes \I F_{2n+i,k} 
- 
v \otimes \I F_{2n+i,2n-2+i}
\end{align*}
for $v \in V_{\tilde{\lambda}}^{K_{\R}}$, 
we have 
\begin{align*}
& d \tau_{\tilde{\lambda}}(F_{2n-2+i,k}) v \otimes \I F_{2n+i,k} 
\\
&= 
\frac{1}{2} 
\{
(d \tau_{\tilde{\lambda}} \otimes 1)(F_{2n-2+i,k})^{2} 
- 
(d \tau_{\tilde{\lambda}} \otimes \ad)(F_{2n-2+i,k})^{2} 
- 1\}(v \otimes \I F_{2n+i,2n-2+i}). 
\end{align*}
Here, $1$ is the trivial representation of $\lie{k}$. 
Since $\I F_{2n+i, j}$ corresponds to $v_{j}^{2n} \boxtimes v_{i}^{2}$
under $\lie{p} \simeq \C^{2n} \boxtimes \C^{2}$ 
and 
$\sum_{k=1}^{2n-3+i} F_{2n-2+i,k}^{2} 
= C_{2n-2+i} - C_{2n-3+i}$, 
we have 
\begin{align}
\nabla_{\tilde{\lambda}, \eta} \phi 
=& 
\frac{1}{2} 
\sum_{i=1}^{2} 
\{
-2 \d_{i} 
+ (d \tau_{\tilde{\lambda}} \otimes 1)(C_{2n-2+i} - C_{2n-3+i}) 
\label{eq:nabla}\\
& \qquad \quad 
- (d \tau_{\tilde{\lambda}} \otimes \ad)(C_{2n-2+i} - C_{2n-3+i}) 
-2n+3-i
\}
\notag
\\
& \qquad \qquad 
\times \phi \otimes (v_{2n-2+i}^{2n} \boxtimes v_{i}^{2}) 
+ \I \lambda_{n+1} \phi \otimes (v_{2n}^{2n} \boxtimes v_{1}^{2})
\notag 
\\
& 
+ 
\I \frac{\eta_{1}}{a_{1}} \phi \otimes (v_{2n-2}^{2n} \boxtimes v_{1}^{2})  
+ 
\I \frac{\eta_{2} a_{1}}{2 a_{2}} 
(\phi \otimes (v_{2n-1}^{2n} \boxtimes v_{2}^{2})  
+ 
\phi \otimes (v_{2n}^{2n} \boxtimes v_{1}^{2})). 
\notag
\end{align}

Let us calculate the projection of 
$\nabla_{\tilde{\lambda}, \eta} \phi$ to 
each irreducible component 
of $V_{\tilde{\lambda}}^{K_{\R}} \otimes \lie{p}$. 
By the identification 
$GT(\lambda) \subset V_{\tilde{\lambda}}^{K_{\R}}$ explained in
\S \ref{section:3.4}, we write 
\[
\phi(a) 
= 
\sum_{Q \in GT(\lambda)} 
c(Q; a) Q.  
\]
We need to calculate the projection 
\begin{align}
\pr_{k}
&
(
\{(d \tau_{\tilde{\lambda}} \otimes 1)(C_{2n-2+i} - C_{2n-3+i}) 
\notag 
\\
& \qquad 
- (d \tau_{\tilde{\lambda}} \otimes \ad)(C_{2n-2+i} - C_{2n-3+i})\} 
Q \otimes v_{2n-2+i}^{2n}
)
\label{eq:Casimir shift}
\\
&= 
\pr_{k}
(
(d \tau_{\tilde{\lambda}} \otimes 1)(C_{2n-2+i} - C_{2n-3+i})
Q \otimes v_{2n-2+i}^{2n}
)
\notag 
\\
& \qquad 
- 
d \tau_{\lambda + e_{k}}(C_{2n-2+i} - C_{2n-3+i}) 
\pr_{k}
(Q \otimes v_{2n-2+i}^{2n}). 
\notag
\end{align}
When $i=2$, this is a multiple 
of $a_{2n-1,k}(Q) \sigma_{2n-1,k} Q$. 
By \eqref{eq:tensor v_{2n}} and Lemma~\ref{lemma:Casimir}, 
its coefficient is 
\begin{align*} 
-&(|\vect{q}_{2n-1}|^{2} + 2 \langle \vect{q}_{2n-1}, \rho_{2n} \rangle) 
+(|\vect{q}_{2n-2}|^{2} + 2 \langle \vect{q}_{2n-2}, \rho_{2n-1} \rangle) 
\\ 
& \qquad 
+(|\vect{q}_{2n-1} + e_{k}|^{2} 
+ 2 \langle \vect{q}_{2n-1} + e_{k}, \rho_{2n} \rangle) 
-(|\vect{q}_{2n-2}|^{2} + 2 \langle \vect{q}_{2n-2}, \rho_{2n-1}
\rangle) 
\\
&= 
2(\sgn k) \lambda_{|k|} + 1 + (\sgn k)(2n-2|k|) 
\\
&= 
2 l_{2n-1,k} +1. 
\end{align*}
Analogously, when $i=1$, \eqref{eq:Casimir shift} is a linear
combination of 
\[
\frac{a_{2n-2,j}(Q) a_{2n-1,k}(\sigma_{2n-2,j}Q)}
{l_{2n-2,j} - l_{2n-1,k}} 
\sigma_{2n-2,j} \sigma_{2n-1,k} Q, 
\]
whose coefficient is 
$2 l_{2n-2,j}$. 

By these and Lemma~\ref{lemma:projection-2}, the projection of
\eqref{eq:nabla} to 
$V_{\lambda + e_{k}}^{Spin(2n)} 
\boxtimes
V_{\lambda_{n+1} \pm 1}^{Spin(2)}$ 
is 
\begin{align}
&\pr_{k,\pm}(\nabla_{\tilde{\lambda}, \eta} \phi) 
\label{eq:projection of nabla}
\\
&= 
- 
\sum_{Q \in GT(\lambda)} 
a_{2n-1,k}(Q) 
\left(
\d_{2} - l_{2n-1,k} + n - 1 \mp \lambda_{n+1} \mp 
\frac{\eta_{2} a_{1}}{2 a_{2}} 
\right) 
c(Q; a) Q_{k}^{\pm} 
\notag
\\
& \quad 
\pm \I 
\sum_{Q \in GT(\lambda)} 
\sum_{0 \leq |j| \leq n-1}
\frac{a_{2n-2,j}(Q) a_{2n-1,k}(\tau_{0,j} Q)}
{l_{2n-2,j} - l_{2n-1,k}} 
\notag 
\\
& \hspace{20mm} 
\times 
\left(
\d_{1} - l_{2n-2,j} + n - 1 \pm \frac{\eta_{2} a_{1}}{2 a_{2}} 
\right) 
c(Q; a) 
\tau_{0,j} Q_{k}^{\pm}
\notag
\\
& \quad 
\pm 
\frac{\eta_{1}}{a_{1}} 
\sum_{Q \in GT(\lambda)} 
\sum_{1 \leq |i| \leq n-1}
\sum_{0 \leq |j| \leq n-1}
\notag 
\\
& \hspace{20mm} 
\times 
\frac{a_{2n-3,i}(Q) a_{2n-2, j}(\tau_{i,0} Q) 
a_{2n-1,k}(\tau_{i,j} Q)}
{(l_{2n-3,i} - l_{2n-2,j} + 1)
(l_{2n-2,j} - l_{2n-1,k})} 
c(Q; a) 
\tau_{i,j} Q_{k}^{\pm}. 
\notag
\end{align}

In order to rewrite \eqref{eq:projection of nabla} in a simple form, 
we need to know the zero points of coefficients in the right hand. 
By Remark~\ref{remark:null points of a_ij}, we have the following
lemma: 
\begin{lemma}
Suppose $Q \in GT(\lambda)$. 
\begin{enumerate}
\item
For $k = \pm 1, \dots, \pm n$, 
$a_{2n-1,k}(Q)$ is not zero if and only if
$\sigma_{2n-1,k} Q \in GT(\lambda + e_{k})$.
\item
For $k = \pm 1, \dots, \pm n$ and 
$j = \pm 1, \dots, \pm (n-1)$, 
\[
\frac{a_{2n-2,j}(Q) a_{2n-1,k}(\tau_{0,j} Q)}
{l_{2n-2,j} - l_{2n-1,k}} 
= 
\frac{a_{2n-2,j}(\sigma_{2n-1,k} Q) a_{2n-1,k}(Q)}
{l_{2n-2,j} - l_{2n-1,k} - 1} 
\] 
is not zero if and only if
$\tau_{0,j} \sigma_{2n-1,k} Q \in GT(\lambda + e_{k})$. 
\item
For $k = \pm 1, \dots, \pm n$, 
$j = \pm 1, \dots, \pm (n-1)$ 
and 
$i = \pm 1, \dots, \pm (n-1)$, 
\begin{align*}
& 
\frac{a_{2n-3,i}(Q) a_{2n-2, j}(\tau_{i,0} Q) 
a_{2n-1,k}(\tau_{i,j} Q)}
{(l_{2n-3,i} - l_{2n-2,j} + 1)
(l_{2n-2,j} - l_{2n-1,k})} 
\\
& \qquad 
= 
\frac{a_{2n-3,i}(\tau_{0,j} Q) a_{2n-2, j}(\sigma_{2n-1,k} Q) 
a_{2n-1,k}(Q)}
{(l_{2n-3,i} - l_{2n-2,j})
(l_{2n-2,j} - l_{2n-1,k} - 1)} 
\end{align*} 
is not zero if and only if
$\tau_{i,j} \sigma_{2n-1,k} Q \in GT(\lambda + e_{k})$. 
\end{enumerate}
\end{lemma}

By this lemma, we can write down \eqref{eq:projection of nabla} in a
simple form. 
\begin{proposition}\label{proposition:5.3}
For $k = \pm 1, \dots, \pm n$, 
$\pr_{k, \pm}(\nabla_{\tilde{\lambda}, \eta} \phi) = 0$ is equivalent
to the following equations. 
\begin{enumerate}
\item
If $Q \in GT(\lambda)$ and 
$\sigma_{2n-1,k} Q \in GT(\lambda + e_{k})$, 
then 
\begin{align*}
& 
\left(
\d_{2} - l_{2n-1,k} 
+ n -1 
\mp \lambda_{n+1} 
\mp 
\frac{\eta_{2} a_{1}}{2 a_{2}} 
\right) 
c(Q; a) 
\\
& 
\pm \I 
\sum_{0 \leq |j| \leq n-1} 
\frac{a_{2n-2,-j}(\tau_{0,j} Q)}
{l_{2n-2,j} + l_{2n-1,k}} 
\left(
\d_{1} + l_{2n-2,j} 
+ n - 1 
\pm 
\frac{\eta_{2} a_{1}}{2 a_{2}} 
\right) 
c(\tau_{0,j} Q; a) 
\\
& 
\mp 
\frac{\eta_{1}}{a_{1}}  
\sum_{1 \leq |i| \leq n-1}
\sum_{0 \leq |j| \leq n-1}
\frac{a_{2n-3,-i}(\tau_{i,j} Q) a_{2n-2,-j}(\tau_{0,j} Q)}
{(l_{2n-3,i} - l_{2n-2,j})
(l_{2n-2,j} + l_{2n-1,k})} 
c(\tau_{i,j} Q; a) 
\\
&= 0.
\end{align*}
\item
If 
$Q \in GT(\lambda)$, 
$\tau_{0,j} \sigma_{2n-1,k} Q \in GT(\lambda + e_{k})$ and 
$\tau_{0,j} Q \not\in GT(\lambda)$, 
then 
\begin{align*}
& 
\left(
\d_{1} - l_{2n-2,j} 
+ n - 1 
\pm 
\frac{\eta_{2} a_{1}}{2 a_{2}} 
\right) 
c(Q; a) 
\\
& \qquad 
+ \I  
\frac{\eta_{1}}{a_{1}} 
\sum_{1 \leq |i| \leq n-1} 
\frac{a_{2n-3,-i}(\tau_{i,0} Q)}
{l_{2n-3,i} + l_{2n-2,j}} 
c(\tau_{i,0} Q; a) 
\\
& = 0.
\end{align*}
\item
If 
$Q \in GT(\lambda)$, 
$\tau_{i,j} \sigma_{2n-1,k} Q \in GT(\lambda + e_{k})$, 
$\tau_{i,j} Q \not\in GT(\lambda)$ 
and 
$\tau_{i,0} Q \not\in GT(\lambda)$, 
then 
\[
c(Q;a) = 0. 
\]
\end{enumerate}
\end{proposition}
\begin{proof}
(1) 
Replace $\tau_{0,j} Q$ in the second sum of the right hand of
\eqref{eq:projection of nabla} with $Q$ and after that replace $j$
with $-j$. 
Replace $\tau_{i,j} Q$ in the third sum of 
\eqref{eq:projection of nabla} with 
$Q$ and after that replace $(i,j)$ with $(-i,-j)$. 
Then we have the above formula. 
The equations in (2) and (3) are obtained in the same way. 
\end{proof}

We investigate the conditions in
Proposition~\ref{proposition:5.3}. 
Note that, $\tilde{\lambda}$ satisfies \eqref{eq:order of lambda}. 
Note also that, 
if $j=0$, then no $Q \in GT(\lambda)$ satisfies the
conditions in Proposition~\ref{proposition:5.3} (2), (3). 

\begin{lemma}\label{lemma:5.5}
\begin{enumerate}
\item
$Q \in GT(\lambda)$ satisfies 
$\sigma_{2n-1,k} Q \in GT(\lambda + e_{k})$ 
if and only if one of the following conditions holds. 
\begin{enumerate}
\item
$k = 1$. 
\item
$k \in [2, n-1]$ and $q_{2n-2,k-1} > \lambda_{k}$．
\item
$k = -(\sgn \lambda_{n}) n$. 
\item
$k \in [-n+1, -1]$ and $q_{2n-2,-k} < \lambda_{-k}$．
\item
$k = (\sgn \lambda_{n}) n$ and $q_{2n-2,n-1} > |\lambda_{n}|$．
\end{enumerate}
\item
$Q \in GT(\lambda)$ satisfies 
$\tau_{0,j} \sigma_{2n-1,k} Q \in GT(\lambda + e_{k})$ and 
$\tau_{0,j} Q \not\in GT(\lambda)$ 
if and only if one of the following conditions holds. 
\begin{enumerate}
\item
$k \in [1, n-1]$, $j = k$, $q_{2n-2,k} = \lambda_{k}$ and 
$q_{2n-3,k-1} > \lambda_{k}$．
The last condition is unnecessary if $k = 1$. 
\item
$k \in [-n+1, -2]$, $j = -(-k-1) = k + 1$, 
$q_{2n-2,-k-1} = \lambda_{-k}$ and $q_{2n-3,-k-1} < \lambda_{-k}$．
\item
$k = -(\sgn \lambda_{n}) n$, $j = -(n-1)$, 
$q_{2n-2,n-1} = |\lambda_{n}|$ and $|q_{2n-3,n-1}| < |\lambda_{n}|$．
\end{enumerate}
\item
$Q \in GT(\lambda)$ satisfies 
$\tau_{i,j} \sigma_{2n-1,k} Q \in GT(\lambda + e_{k})$, 
$\tau_{i,j} Q \not\in GT(\lambda)$ 
and 
$\tau_{i,0} Q \not\in GT(\lambda)$ 
if and only if one of the following conditions holds. 
\begin{enumerate}
\item
$k \in [1, n-1]$, $i = j = k$, 
$q_{2n-3,k} = q_{2n-2,k} = \lambda_{k}$ and 
$q_{2n-4,k-1} > \lambda_{k}$．
The last condition is unnecessary if $k = 1$. 
\item
$k \in [-n+1, -3]$, 
$i = -(-k-2) = k+2$, $j = -(-k-1) = k+1$, 
$q_{2n-3,-k-2} = q_{2n-2,-k-1} = \lambda_{-k}$ and 
$q_{2n-4,-k-2} < \lambda_{-k}$．
\item
$k = -(\sgn \lambda_{n}) n$, 
$i = -(n-2)$, $j = -(n-1)$, 
$q_{2n-3,n-2} = q_{2n-2,n-1} = |\lambda_{n}|$ and 
$q_{2n-4,n-2} < |\lambda_{n}|$．
\end{enumerate}
\end{enumerate}
\end{lemma}

We write down the differential-difference
  equation $\mathcal{D}_{\tilde{\lambda}, \eta} \phi = 0$ in the case
  $\Lambda \in \Xi_{m,\pm}$. 
In order to write equations briefly, 
we introduce some notation. 
\begin{notation}
\begin{align*}
& 
\mathcal{D}_{1}^{\pm}(Q) 
:= 
\d_{1} + n - 1 \pm \frac{\eta_{2} a_{1}}{2 a_{2}}, 
\\
& V_{j}^{\pm}(Q; a) 
:= 
(\mathcal{D}_{1}^{\pm}(Q) + l_{2n-2,j}) 
c(\tau_{0,j} Q; a) 
\\
& \hspace{20mm} 
+ \I 
\frac{\eta_{1}}{a_{1}} 
\sum_{1 \leq |i| \leq n-1} 
\frac{a_{2n-3,-i}(\tau_{i,j} Q)} 
{l_{2n-3,i} - l_{2n-2,j}} 
c(\tau_{i,j} Q; a),   
\\
& 
\mathcal{D}_{2}^{\pm}(Q) 
:= 
\d_{2} 
+ 
\sum_{p=1}^{m-1} 
l_{2n-1, p} 
- 
\sum_{p=1}^{m-2} 
l_{2n-2, p} 
+ m - 2 
\pm \lambda_{n+1} 
\pm 
\frac{\eta_{2} a_{1}}{2 a_{2}},
\\
& 
A_{j}(Q) 
:= 
a_{2n-2,-j}(\tau_{0,j} Q) 
\frac{
\prod_{p = 1}^{m-2} 
(l_{2n-2,j} - l_{2n-2,p})}
{\prod_{p=1}^{m-1} (l_{2n-2,j} + l_{2n-1,-p})} 
\\
& 
B_{j'}(j, Q) 
:= 
\frac{
a_{2n-2, -j'}(\tau_{0,j'} Q)
\prod_{p \in [-n+1, -m+1] \cup [0, n-1] \setminus \{j\}} 
(l_{2n-2, j'} - l_{2n-2, p})
}
{
\prod_{p \in [-n, -1] \cup [m, n]} 
(l_{2n-2, j'} + l_{2n-1, p})
}. 
\end{align*}
\end{notation}

Let $\Lambda \in \Xi_{m,\pm}$, $m = 2, \dots, n$. 
Then $\phi$ satisfies 
\begin{align*}
& 
\pr_{-k, +} (\nabla_{\tilde{\lambda}, \eta} \phi) = 0, 
&
&
\pr_{-k, -} (\nabla_{\tilde{\lambda}, \eta} \phi) = 0 
&
&
\mbox{for} \quad 1 \leq k \leq m-1, 
\\
& 
\pr_{k, \mp} (\nabla_{\tilde{\lambda}, \eta} \phi) = 0, 
&
&
\pr_{-k, \mp} (\nabla_{\tilde{\lambda}, \eta} \phi) = 0 
&
&
\mbox{for} \quad m \leq k \leq n. 
\end{align*}
By Proposition~\ref{proposition:5.3} and
Lemma~\ref{lemma:5.5}, 
we have the following lemma. 
\begin{lemma}\label{lemma:5.8} 
Suppose $\Lambda \in \Xi_{m,\pm}$, $m \in [2, n]$.  
Then $\mathcal{D}_{\tilde{\lambda}, \eta} \phi = 0$ is equivalent to
the followings. 
\begin{enumerate}
\item
For $k \in [1, m-1]$, if $Q \in GT(\lambda)$ satisfies 
$q_{2n-2,k} < \lambda_{k}$, then 
\begin{align}
& 
(\d_{2} - l_{2n-1,-k} + n - 1) 
c(Q;a) 
+ 
\I \frac{\eta_{2} a_{1}}{2 a_{2}} 
\sum_{0 \leq |j| \leq n-1} 
\frac{a_{2n-2,-j}(\tau_{0,j} Q)}
{l_{2n-2,j} + l_{2n-1,-k}} 
c(\tau_{0,j} Q;a) 
=0. 
\label{eq:5.8(1)a} 
\end{align}
\item
For $k \in [-n, -1] \cup [m, n]$, 
if $Q \in GT(\lambda)$ satisfies one of the condition 
Lemma~\ref{lemma:5.5}(1), then 
\begin{align}
& 
\left(
\d_{2} - l_{2n-1,k} + n - 1 \pm \lambda_{n+1}  
\pm \frac{\eta_{2} a_{1}}{2 a_{2}} 
\right) 
c(Q;a) 
\label{eq:5.8(2)} 
\\
& \qquad 
\mp \I  
\sum_{0 \leq |j| \leq n-1} 
\frac{a_{2n-2,-j}(\tau_{0,j} Q)}
{l_{2n-2,j} + l_{2n-1,k}} 
V_{j}^{\mp}(Q; a) 
\notag 
\\
&=0.
\notag
\end{align}
\item
For $k \in [2, m-1]$, if $Q \in GT(\lambda)$ satisfies 
$q_{2n-2,k-1} = \lambda_{k}$ and 
$q_{2n-3,k-1} < \lambda_{k}$, then 
\begin{align}
& 
c(Q;a) = 0, 
& 
& 
\sum_{1 \leq |i| \leq n - 1} 
\frac{a_{2n-3,-i}(\tau_{i,0} Q)}
{l_{2n-3,i} + l_{2n-2,k-1}} 
c(\tau_{i,0} Q;a) 
= 0. 
\label{eq:5.8(3)} 
\end{align}
\item
Suppose $Q \in GT(\lambda)$ satisfies the condition in
Lemma~\ref{lemma:5.5}(2) for 
$k \in [-n+1, -2] \cup [m, n-1] \cup \{-(\sgn \lambda_{n+1}) n\}$. 
Define $j(k)$ by 
\[
j(k) 
= 
\begin{cases}
k & \mbox{if } k \in [m, n-1], 
\\
k+1 & \mbox{if } k \in [-n+1,-2], 
\\
-n+1 & \mbox{if } k = -(\sgn \lambda_{n}) n. 
\end{cases} 
\]
Then, 
\begin{align}
& 
(\mathcal{D}_{1}^{\mp}(Q) - l_{2n-2,j(k)}) 
c(Q;a) 
+ \I \frac{\eta_{1}}{a_{1}} 
\sum_{1 \leq |i| \leq n - 1} 
\frac{a_{2n-3,-i}(\tau_{i,0} Q)}
{l_{2n-3,i} + l_{2n-2,j(k)}} 
c(\tau_{i,0} Q;a) 
= 0.  
\label{eq:5.8(4)}
\end{align}
\item
Suppose 
$k \in [-n+1, -3] \cup [m, n-1] \cup \{-(\sgn \lambda_{n+1}) n\}$. 
If $Q \in GT(\lambda)$ satisfies the condition in
Lemma~\ref{lemma:5.5}(3), i.e. if 
\begin{align*}
& 
q_{2n-3,k} = q_{2n-2,k} = \lambda_{k}, 
q_{2n-4,k-1} > \lambda_{k}, 
& & 
\mbox{when } k \in [m, n-1], 
\\
& 
q_{2n-3,-k-2} = q_{2n-2,-k-1} = \lambda_{-k}, 
q_{2n-4,-k-2} < \lambda_{-k}, 
& & 
\mbox{when } k \in [-n+1, -3], 
\\
& 
q_{2n-3,n-2} = q_{2n-2,n-1} = |\lambda_{n}|, 
q_{2n-4,n-2} < |\lambda_{n}|, 
& & 
\mbox{when } k = -(\sgn \lambda_{n}) n, 
\end{align*}
then 
\begin{equation*}
c(Q;a) = 0. 
\end{equation*}
\end{enumerate}
\end{lemma}

By eliminating some terms of the equations in Lemma~\ref{lemma:5.8}, 
we obtain the following equations. 

\begin{corollary}
\begin{enumerate}
\item
Suppose $j \in [1, m-1]$ satisfies $\tau_{0,j} Q \in GT(\lambda)$. 
Then 
\begin{align}
& 
\left(\mathcal{D}_{2}^{\pm}(Q) - l_{2n-2,m-1} + l_{2n-2, j} 
\mp \Lambda_{n+1} 
\mp \frac{\eta_{2} a_{1}}{2 a_{2}}
\right) 
c(Q; a) 
\label{eq:5.9(2)} 
\\
&
+ \I 
\frac{\eta_{2} a_{1}}{2 a_{2}} 
\sum_{j' \in \{j\} \cup [m, n-1] \atop \cup [-n+1, 0]} 
\frac{
a_{2n-2, -j'}(\tau_{0,j'} Q)
\prod_{1 \leq p \leq m-1, p \not= j} (l_{2n-2, j'} - l_{2n-2, p})
}
{
\prod_{p = 1}^{m-1} 
(l_{2n-2, j'} + l_{2n-1, -p})
}
\notag 
\\
& \hspace{35mm} 
\times 
c(\tau_{0,j'} Q; a) 
\notag
\\
&= 0. 
\notag
\end{align}
Here, $\Lambda_{n+1} = \lambda_{n+1} \mp (n - m + 1)$ is the
$(n+1)$-st part of the Harish-Chandra parameter 
$\Lambda \in \Xi_{m,\pm}$. 
\item
Suppose $j \in [-n+1, -m+1] \cup [0, n-1]$. 
If there exists $i \in [-n+1, n-1]$ such that 
$\tau_{i, j} Q \in GT(\lambda)$, then 
\begin{align}
& 
(\mathcal{D}_{2}^{\pm}(Q) 
+ 
l_{2n-2, j})  
c(Q; a) 
\mp \I
\sum_{j' \in \{j\} \cup [-m+2, -1]} 
B_{j'}(j, Q) 
V_{j'}^{\mp}(Q; a) 
= 0. 
\label{eq:5.10(2)} 
\end{align}
\item
Especially, when $j = 0$，the equation \eqref{eq:5.10(2)} is 
\begin{align}
& 
\mathcal{D}_{2}^{\pm}(Q) 
c(Q; a) 
\mp \I  
\sum_{j' = -m+2}^{-1}  
B_{j'}(0, Q) V_{j'}^{\mp}(Q; a) 
\label{eq:5.10(3)} 
\\
& 
\mp 
\frac{
\prod_{p=1}^{n-1} l_{2n-3, p}}
{\prod_{p=m}^{n} l_{2n-1,p} \prod_{p=1}^{m-2} (-l_{2n-2,-p})} 
\notag \\
& \quad \times 
\left(
\mathcal{D}_{1}^{\mp}(Q) 
c(Q; a) 
+ 
\I 
\frac{\eta_{1}}{a_{1}} 
\sum_{1 \leq |i| \leq n-1} 
\frac{a_{2n-3, -i}(\tau_{i, 0} Q)}
{l_{2n-3, i}} 
c(\tau_{i, 0} Q; a) 
\right)
\notag
\\
&= 0 
\notag
\end{align}
\end{enumerate}
\end{corollary}
\begin{proof}
\eqref{eq:5.9(2)} can be obtained from the equations
\eqref{eq:5.8(1)a}, 
applied to $k \in [1, m-1]$, 
by eliminating $c(\tau_{0,j'} Q; a)$, 
$j' \in [1, m-1] \setminus \{j\}$. 
\eqref{eq:5.10(2)} can be obtained from the equations
\eqref{eq:5.8(2)}, applied to $k \in [-n, -1] \cup [m, n]$, 
by eliminating $V_{j'}(Q; a)$, 
$j' \in  [-n+1, -m+1] \cup [0, n-1] \setminus \{j\}$. 
\end{proof}

Suppose $Q \in GT(\lambda)$ and $\tau_{0,-j} Q \in GT(\lambda)$ 
for some $j \in [-n+1, -m+1] \cup [0, n-1]$. 
Replace $Q$ in \eqref{eq:5.10(2)} by $\tau_{0,-j} Q$. 
Then we obtain 
\begin{align}
& 
\mp 
\frac{\I}
{B_{j}(j, \tau_{0,-j} Q)} 
(\mathcal{D}_{2}^{\pm}(Q) 
- 
l_{2n-2, -j})  
c(\tau_{0,-j}Q; a) 
\label{eq:5.13(2)} 
\\
& + 
\sum_{j' \in [-m+2, -1]}
\frac{B_{j'}(j, \tau_{0,-j} Q)}
{B_{j}(j, \tau_{0,-j} Q)} 
V_{j'}^{\mp}(\tau_{0,-j} Q; a)  
\notag
\\
&
+
(\mathcal{D}_{1}^{\mp}(Q) - l_{2n-2,-j}) 
c(Q; a) 
+ \I 
\frac{\eta_{1}}{a_{1}} 
\sum_{1 \leq |i| \leq n-1} 
\frac{a_{2n-3,-i}(\tau_{i,0} Q)} 
{l_{2n-3,i} + l_{2n-2,-j}} 
c(\tau_{i,0} Q;a) 
\notag
\\
&= 0.
\notag
\end{align}

If $Q \in GT(\lambda)$ satisfies 
\begin{enumerate}
\item
$\tau_{0,-j} Q \not\in GT(\lambda)$ and 
$\tau_{0,-j} \sigma_{2n-1,-j} Q \in GT(\lambda + e_{-j})$ 
for some $j \in [-n+1, -m]$; 
or 
\item
$\tau_{0,-j} Q \not\in GT(\lambda)$ and 
$\tau_{0,-j} \sigma_{2n-1,-j-1} Q \in GT(\lambda - e_{j+1})$ 
for some $j \in [2, n-2]$; 
or 
\item
$\tau_{0,-n+1} Q \not\in GT(\lambda)$ and 
$\tau_{0,-n+1} \sigma_{2n-1,-(\sgn \lambda_{n}) n} Q 
\in GT(\lambda - (\sgn \lambda_{n}) e_{n})$, 
\end{enumerate}
then 
\eqref{eq:5.8(4)} holds for $k$ such that $j(k) = -j$ 
(or $j(k) = n-1$ in the case (3)). 
We may, and do, regard this equation \eqref{eq:5.8(4)} as a 
special case of \eqref{eq:5.13(2)}, whose first line is zero.

\begin{lemma}
If \eqref{eq:5.13(2)} holds for $j = j_{1}, \dots, j_{s}$ and 
$\tau_{i,0} Q \in GT(\lambda)$ for $i = i_{1}, \dots, i_{s-1}$, 
then 
\begin{align}
& 
\pm
\sum_{\mu=1}^{s} 
\frac{\I}
{B_{j_{\mu}}(j_{\mu}, \tau_{0,-j_{\mu}} Q)} 
\frac{
\prod_{\nu=1}^{s-1} (l_{2n-2,-j_{\mu}} + l_{2n-3, i_{\nu}})}
{\prod_{\nu=1, \not= \mu}^{s} (l_{2n-2,-j_{\mu}} - l_{2n-2, -j_{\nu}})}
\label{eq:5.13(3)}\\
& \hspace{35mm} 
\times (\mathcal{D}_{2}^{\pm}(Q) 
- 
l_{2n-2, -j_{\mu}})  
c(\tau_{0,-j_{\mu}}Q; a) 
\notag
\\
& 
+ 
\sum_{\mu=1}^{s} 
\sum_{j' \in [-m+2, -1]}
\frac{
\prod_{\nu=1}^{s-1} (l_{2n-2,-j_{\mu}} + l_{2n-3, i_{\nu}})}
{\prod_{\nu=1, \not= \mu}^{s} (l_{2n-2,-j_{\mu}} - l_{2n-2, -j_{\nu}})}
\frac{B_{j'}(j_{\mu}, \tau_{0,-j_{\mu}} Q)}
{B_{j_{\mu}}(j_{\mu}, \tau_{0,-j_{\mu}} Q)} 
\notag\\
& \hspace{35mm} 
\times 
V_{j'}^{\mp}(\tau_{0,-j_{\mu}} Q; a)  
\notag\\
&
+
\left(
\mathcal{D}_{1}^{\mp}(Q) 
- \sum_{\mu=1}^{s} l_{2n-2,-j_{\mu}} 
- \sum_{\nu=1}^{s-1} l_{2n-3,i_{\nu}}
\right) 
c(Q; a) 
\notag\\
& 
+ \I 
\frac{\eta_{1}}{a_{1}} 
\sum_{1 \leq |i| \leq n-1} 
\frac{a_{2n-3,-i}(\tau_{i,0} Q) 
\prod_{\nu=1}^{s-1} (l_{2n-3,i} - l_{2n-3,i_{\nu}})} 
{\prod_{\mu=1}^{s} (l_{2n-3,i} + l_{2n-2,-j_{\mu}})} 
c(\tau_{i,0} Q;a) 
\notag
\\
&= 0.
\notag
\end{align}
\end{lemma}
\begin{proof}
This is obtained from the equations \eqref{eq:5.13(2)}, 
applied to $j = j_{1}, \dots, j_{s}$, 
by eliminating $c(\tau_{i, 0} Q; a)$, $i = i_{1}, \dots, i_{s-1}$. 
\end{proof}

We apply this lemma to some special cases. 
For $m' \in [m-1,n-1]$ and $Q \in GT(\lambda)$, 
let 
\begin{align*}
K_{1}(m') 
&:= 
\{
p \in [1,m'-1] | 
\tau_{p,0} Q \in GT(\lambda)
\},
\\
K_{2}(m') 
&:= [1,m'-1] \setminus K_{1}(m'), 
\\
K_{3}(m') 
&:= 
\{
p \in [-n+2, -m'] | 
\tau_{p,0} Q \in GT(\lambda)
\}, 
\\
K_{4}(m') 
&:= 
[-n+2, -m'] \setminus K_{3}(m'), 
\\
K_{5}(m') 
&:= 
\{p-1 | p \in K_{3}(m')\}. 
\end{align*}
If $p \in K_{1}(m')$, then 
$\tau_{0,-p} Q \in GT(\lambda)$ or  
$\lambda_{p+1} = q_{2n-2,p} > q_{2n-3,p} \geq q_{2n-2,p+1}$. 
In the latter case, 
$\tau_{0,-p} \sigma_{2n-1,-p-1} Q \in GT(\lambda - e_{p+1})$. 
Thus \eqref{eq:5.13(2)} holds for $j = p$ in each case. 
If $p \in K_{3}(m')$, then 
$\tau_{0,-p+1} Q \in GT(\lambda)$ or 
$q_{2n-2,-p} \geq q_{2n-3,-p} > q_{2n-2,-p+1} = \lambda_{-p+1}$. 
In the latter case, 
$\tau_{0,-p+1} \sigma_{2n-1,-p+1} Q \in GT(\lambda + e_{-p+1})$. 
Thus \eqref{eq:5.13(2)} holds for $j = p-1$ in each case. 

Let $i \in K_{1}(m') \cup K_{3}(m')$, or $i = -n+1$ if 
$\tau_{-n+1,0} Q \in GT(\lambda)$.  
Define 
\begin{align*}
& 
I_{1}(m') 
:= 
K_{1}(m') \cup K_{5}(m') \cup \{0\}, 
\\ 
& 
I_{2}(m')
:= 
\begin{cases}
K_{1}(m') \cup K_{3}(m') \cup \{-n+1\} 
\quad 
\mbox{if } \tau_{-n+1,0} Q \in GT(\lambda) 
\\
K_{1}(m') \cup K_{3}(m') \cup \{n-1\} 
\quad 
\mbox{if } \tau_{-n+1,0} Q \not\in GT(\lambda), 
\end{cases}
\\
& 
I_{2}(m', i) := I_{2}(m') \setminus \{i\}. 
\end{align*}
Then \eqref{eq:5.13(3)} holds for 
$\{j_{1}, \dots, j_{s}\} = I_{1}(m')$ and 
$\{i_{1}, \dots, i_{s-1}\} = I_{2}(m', i)$, 
so the following formula is obtained.  
\begin{corollary}
For $i \in K_{1}(m') \cup K_{3}(m')$ or $i = -n+1$ if 
$\tau_{-n+1,0} Q \in GT(\lambda)$, 
then 
\begin{align}
-\I &
\frac{\eta_{1}}{a_{1}} 
\frac{a_{2n-3,-i}(\tau_{i,0} Q) 
\prod_{i' \in I_{2}(m', i)} (l_{2n-3,i} - l_{2n-3,i'})}
{\prod_{j \in I_{1}(m')} (l_{2n-3,i} + l_{2n-2,-j})}
c(\tau_{i,0} Q;a) 
\label{eq:5.13(3)-1}
\\
=& 
\pm 
\sum_{j \in I_{1}(m')}
\frac{\I}
{B_{j}(j, \tau_{0,-j} Q)} 
\frac{
\prod_{i' \in I_{2}(m', i)} (l_{2n-2,-j} + l_{2n-3, i'})}
{\prod_{j' \in I_{1}(m') \setminus \{j\}} 
(l_{2n-2,-j} - l_{2n-2, -j'})}
\notag \\ 
& \hspace{35mm} 
\times (\mathcal{D}_{2}^{\pm}(Q) 
- 
l_{2n-2, -j})  
c(\tau_{0,-j}Q; a) 
\notag
\\
& 
+ 
\sum_{j \in I_{1}(m')}
\sum_{j' \in [-m+2, -1]}
\frac{
\prod_{i' \in I_{2}(m', i)} (l_{2n-2,-j} + l_{2n-3, i'})}
{\prod_{j'' \in I_{1}(m') \setminus \{j\}} 
(l_{2n-2,-j} - l_{2n-2, -j''})}
\frac{B_{j'}(j, \tau_{0,-j} Q)}
{B_{j}(j, \tau_{0,-j} Q)} 
\notag \\ 
& \hspace{35mm} 
\times V_{j'}^{\mp}(\tau_{0,-j} Q; a)  
\notag\\
&
+
\left(
\mathcal{D}_{1}^{\mp}(Q) 
- \sum_{j \in I_{1}(m')} l_{2n-2,-j}
- \sum_{i' \in I_{2}(m', i)} l_{2n-3,i'}
\right) 
c(Q; a) 
\notag
\\
& 
+ \I 
\frac{\eta_{1}}{a_{1}} 
\sum_{i' \in [-m'+1,-1] \atop \cup [m',n-1]} 
\frac{a_{2n-3,-i'}(\tau_{i',0} Q) 
\prod_{i'' \in I_{2}(m', i)} (l_{2n-3,i'} - l_{2n-3,i''})}
{\prod_{j \in I_{1}(m')} (l_{2n-3,i'} + l_{2n-2,-j})}
c(\tau_{i',0} Q;a). 
\notag
\end{align}
\end{corollary}

Suppose $Q \in GT(\lambda)$ satisfies 
\begin{align}
& 
q_{2n-2,p} = q_{2n-3,p} = q_{2n-4,p} 
\quad 
\mbox{for any } 
p \in [1, m-2]. 
\label{eq:5.13(1)} 
\end{align}
Define 
\begin{equation}\label{eq:K_6(m')}
K_{6}(m') 
:=\{j \in [m-1,m'] \cup [-n+1,-m'] | 
\tau_{0,j} Q \in GT(\lambda)
\}. 
\end{equation}
For $j \in K_{6}(m')$, define 
\[
I_{1}(m', j) 
:= 
I_{1}(m') \cup \{-j\}. 
\]
Then \eqref{eq:5.13(3)} holds for 
$\{j_{1}, \dots, j_{s}\} = I_{1}(m', j)$ and 
$\{i_{1}, \dots, i_{s-1}\} = I_{2}(m')$. 
By \eqref{eq:5.13(1)} and the definition of $V_{j}^{\pm}(Q; a)$, 
the term $V_{j''}(\tau_{0,-j'} Q)$ is zero for 
$j' \in [-m+2,-1]$ and $j'' \in I_{1}(m', j)$. 
Thus we obtain the following formula. 
\begin{corollary}
Suppose $Q \in GT(\lambda)$ satisfies \eqref{eq:5.13(1)}. 
If $j \in K_{6}(m')$, then 
\begin{align}
\mp&
\frac{\I}
{B_{-j}(-j, \tau_{0,j} Q)} 
\frac{
\prod_{i \in I_{2}(m')} (l_{2n-2,j} + l_{2n-3, i})}
{\prod_{j' \in I_{1}(m')} 
(l_{2n-2,j} - l_{2n-2, -j'})}
(\mathcal{D}_{2}^{\pm}(Q) 
- 
l_{2n-2, j})  
c(\tau_{0,j}Q; a) 
\label{eq:5.13(3)-2}
\\
=& 
\pm 
\sum_{j' \in I_{1}(m')}
\frac{\I}
{B_{j'}(j', \tau_{0,-j'} Q)} 
\frac{
\prod_{i \in I_{2}(m')} (l_{2n-2,-j'} + l_{2n-3, i})}
{\prod_{j'' \in I_{1}(m', j) \setminus \{j'\}} 
(l_{2n-2,-j'} - l_{2n-2, -j''})}
\notag
\\
& 
\quad 
\times 
(\mathcal{D}_{2}^{\pm}(Q) 
- 
l_{2n-2, -j'})  
c(\tau_{0,-j'}Q; a) 
\notag
\\
&
+
\left(
\mathcal{D}_{1}^{\mp}(Q) 
- \sum_{j' \in I_{1}(m', j)} l_{2n-2,-j'}
- \sum_{i \in I_{2}(m')} l_{2n-3,i}
\right) 
c(Q; a) 
\notag
\\
& 
+ \I 
\frac{\eta_{1}}{a_{1}} 
\sum_{i \in [-m'+1,-m+1] \atop \cup [m',n-1]} 
\frac{a_{2n-3,-i}(\tau_{i,0} Q) 
\prod_{i' \in I_{2}(m')} (l_{2n-3,i} - l_{2n-3,i'})}
{\prod_{j' \in I_{1}(m', j)} (l_{2n-3,i} + l_{2n-2,-j'})}
c(\tau_{i,0} Q;a). 
\notag
\end{align}
\end{corollary}

\subsection{Condition for $\vect{q}_{2n-4}$}

In this subsection, we deduce necessary conditions for
$\vect{q}_{2n-4}$ so that the differential-difference equation has a
non-trivial solution. 
In the first place, we deduce conditions for 
$q_{2n-4,k}$, $k \in [1,  m-2]$. 
\begin{lemma}
Suppose $k \in [1, m-2]$. 
If $Q \in GT(\lambda)$ satisfies $q_{2n-4,k} < \lambda_{k+1}$, 
then $c(Q; a) = 0$. 
\end{lemma}
\begin{proof}
\par
\noindent
\textbf{Step 1.} 
Among the Gelfand-Tsetlin patterns with $q_{2n-4,k} < \lambda_{k+1}$, 
there is a $Q$ satisfying 
$q_{2n-2,k} = \lambda_{k+1}$ and $q_{2n-3,k} < \lambda_{k+1}$. 
By the first equation in \eqref{eq:5.8(3)}, $c(Q; a) = 0$ for this
$Q$. 
Therefore, by the second equation in \eqref{eq:5.8(3)}, 
$c(Q; a) = 0$ for $Q \in GT(\lambda)$ satisfying 
$q_{2n-2,k} = q_{2n-3,k} = \lambda_{k+1}$. 
Thus, $c(Q; a) = 0$ for all $Q \in GT(\lambda)$ such that 
$q_{2n-2,k} = \lambda_{k+1}$ and $q_{2n-3,k} \leq \lambda_{k+1}$. 
\par
\noindent
\textbf{Step 2.}
Suppose $Q \in GT(\lambda)$ satisfies 
$q_{2n-4,k} < \lambda_{k+1}$ and 
$q_{2n-4,k} \leq q_{2n-3,k} \leq q_{2n-2,k} = \lambda_{k+1}$. 
In this case, if $j' \in [m, n-1] \cup [-n+1, 0] \setminus \{k\}$,
then $c(\tau_{0,j'} Q; a) = 0$ by Step 1, 
since $\tau_{0,j'} Q$ satisfies the condition in it. 
Then by \eqref{eq:5.9(2)}, applied to this $Q$ and $j = k$, 
we know that $c(Q; a) = 0$ for $Q \in GT(\lambda)$ satisfying 
$q_{2n-2,k} = \lambda_{k+1} + 1$ and $q_{2n-3,k} \leq \lambda_{k+1}$. 
By repeating this discussion, we know that $c(Q; a) = 0$ 
for $Q \in GT(\lambda)$ satisfying  
$\lambda_{k} \geq q_{2n-2,k} \geq \lambda_{k+1}$ and 
$q_{2n-3,k} \leq \lambda_{k+1}$. 
\par
\noindent
\textbf{Step 3.}
Suppose $Q \in GT(\lambda)$ satisfies 
$q_{2n-2,k} = q_{2n-3,k} = \lambda_{k+1}$. 
By Step 2, $c(Q; a) = 0$. 
If $\tau_{i,j} Q \in GT(\lambda)$ for $(i,j) \not= (k,k)$, 
then it satisfies 
$\lambda_{k} \geq q_{2n-2,k} \geq \lambda_{k+1}$ and 
$q_{2n-3,k} \leq \lambda_{k+1}$; so it satisfies the condition in the
conclusion of Step 2. 
It follows that $c(\tau_{i,j} Q; a) = 0$ for $(i, j) \not= (k, k)$. 
Then by \eqref{eq:5.10(2)}, applied to this $Q$ and $j = k$, 
we have $c(\tau_{k,k} Q; a) = 0$, that is to say 
$c(Q; a) = 0$ for $Q \in GT(\lambda)$ satisfying 
$q_{2n-4,k} < \lambda_{k+1}$ and 
$q_{2n-2,k} = q_{2n-3,k} = \lambda_{k+1} + 1$. 
\par
\noindent
\textbf{Step 4.} 
Suppose $Q \in GT(\lambda)$ satisfies the condition in the conclusion of
Step 3. Then, by applying the discussion in Step 2 for this $Q$, 
we know that $c(Q; a) = 0$ for $Q \in GT(\lambda)$ satisfying 
$q_{2n-3,k} \leq \lambda_{k+1} + 1$. 

By repeating the shift operations as in Step 2 and Step 3, the lemma is
shown. 
\end{proof}

In the second place, we deduce conditions for 
$q_{2n-4,k}$, $k \in [m-1, n-2]$. 
\begin{lemma} 
Suppose $k \in [m-1, n-2]$. 
If $Q \in GT(\lambda)$ satisfies $q_{2n-4,k} > \lambda_{k+1}$, then 
$c(Q; a) = 0$. 
\end{lemma}
\begin{proof}
\par
\noindent
\textbf{Step 1.} 
Among the Gelfand-Tsetlin patterns with $q_{2n-4,k} > \lambda_{k+1}$, 
there is a $Q$ satisfying 
$q_{2n-3,k+1} = q_{2n-2,k+1} = \lambda_{k+1}$. 
By Lemma~\ref{lemma:5.8} (5), 
$c(Q; a) = 0$ for this $Q$. 
\par
\noindent
\textbf{Step 2.}
We know $c(Q; a) = 0$ for those $Q$ with 
$q_{2n-3,k+1} = q_{2n-2,k+1} = \lambda_{k+1}$. 
Shift $q_{2n-3,k+1}$ downward by using \eqref{eq:5.10(3)}. 
Then we know that $c(Q; a) = 0$ for all $Q \in GT(\lambda)$ satisfying 
$q_{2n-4,k} > \lambda_{k+1}$ and 
$\lambda_{k+1} = q_{2n-2,k+1} \geq q_{2n-3,k+1}$. 
\par
\noindent
\textbf{Step 3.}
Suppose $Q \in GT(\lambda)$ satisfies the condition in Step 1. 
Then by \eqref{eq:5.10(2)}, applied to this $Q$ and $j = -k-1$, 
we know that $c(Q; a) = 0$ for all $Q \in GT(\lambda)$ satisfying 
$q_{2n-4,k} > \lambda_{k+1}$ and 
$q_{2n-2,k+1} = q_{2n-3,k+1} = \lambda_{k+1} - 1$. 

By repeating the shift operations as in Step 2 and Step 3, 
the lemma is shown. 
\end{proof}

\begin{lemma} 
If $Q \in GT(\lambda)$ satisfies $q_{2n-4,k} < \lambda_{k+2}$ 
for $k \in [m-1, n-3]$, or $q_{2n-4,n-2} < |\lambda_{n}|$, 
then 
$c(Q; a) = 0$. 
\end{lemma}
\begin{proof}
The proof of this lemma is just the same as that of the previous one. 
In the first place, if $Q$ satisfies the above condition and 
$q_{2n-3,k} = q_{2n-2,k+1} = \lambda_{k+2}$ 
(or $=|\lambda_{n}|$ if $k = n-2$), we have 
$c(Q; a) = 0$ by Lemma~\ref{lemma:5.8} (5). 
Shift $q_{2n-3,k}$ upward by using \eqref{eq:5.10(3)}. 
Then we know that $c(Q; a) = 0$ for all $Q \in GT(\lambda)$ satisfying 
$q_{2n-4,k} < \lambda_{k+2}$ and $q_{2n-2,k+1} = \lambda_{k+2}$. 
In the third place, 
if $Q \in GT(\lambda)$ satisfies 
$q_{2n-3,k} = q_{2n-2,k+1} = \lambda_{k+2}$, 
then shift upward them to 
$q_{2n-3,k} = q_{2n-2,k+1} = \lambda_{k+2} + 1$ by using 
\eqref{eq:5.10(2)}. 
By repeating these shift operations, the lemma is shown. 
\end{proof}
\begin{corollary}
If $Q \in GT(\lambda)$ does not satisfy 
\begin{equation}\label{eq:5.11}
\begin{cases} 
\lambda_{k} \geq q_{2n-4,k} \geq \lambda_{k+1} 
\quad 
\mbox{for } k \in [1, m-2], 
\\
\lambda_{k+1} \geq q_{2n-4,k} \geq \lambda_{k+2} 
\quad 
\mbox{for } k \in [m-1, n-3], 
\\
\lambda_{n-1} \geq q_{2n-4,n-2} \geq |\lambda_{n}|, 
\end{cases}
\end{equation}
then $c(Q; a) = 0$. 
\end{corollary}

\subsection{Reduction to the ``corner'' vectors}
\label{subsection:corner}

Choose a $\vect{q}_{2n-4}$ satisfying the condition \eqref{eq:5.11}. 
In this subsection, we show that, 
if $c(Q_{0}; a)$ is known for some $Q_{0} \in GT(\lambda)$
containing this $\vect{q}_{2n-4}$, then it completely determines 
the other $c(Q; a)$'s for $Q \in GT(\lambda)$ containing the same
$\vect{q}_{1}, \dots, \vect{q}_{2n-4}$ parts. 
We call $Q_{0}$ with this property a ``corner vector''. 

As will be seen in \S \ref{section:solution}, there is a set of
Gelfand-Tsetlin bases $Q$, including corner vectors, such that we can
explicitly write down the scalar differential equations satisfied by
$c(Q; a)$. 
Such bases $Q \in GT(\lambda)$ satisfy the following conditions: 
For an $m' \in [m-1, n-1]$, 
\begin{equation}\label{eq:5.13(4)} 
\begin{cases}
q_{2n-2,p} = \lambda_{p+1}, \enskip q_{2n-3,p} = q_{2n-4,p} 
\quad 
\mbox{for } p \in [m-1, m'-1], 
\\
q_{2n-2,p} = \lambda_{p}, \enskip q_{2n-3,p} = q_{2n-4,p-1} 
\quad 
\mbox{for } p \in [m'+1, n-1], 
\\
q_{2n-2,m'} = q_{2n-3,m'} \in 
[\lambda_{m'+1}, q_{2n-4,m'-1}]
\enskip 
\mbox{if } m' \geq m, 
\\
\hspace{38mm}
\mbox{or} \enskip  
[\lambda_{m}, \lambda_{m-1}] 
\enskip \mbox{if } m' = m - 1.
\end{cases}
\end{equation}
\begin{definition}\label{def:Q^pm}
$Q_{m'}^{-} \in GT(\lambda)$ is the Gelfand-Tsetlin pattern which
  satisfies \eqref{eq:5.13(1)}, \eqref{eq:5.11}, \eqref{eq:5.13(4)}
  and $q_{2n-2,m'} = q_{2n-3,m'} = \lambda_{m'+1}$. 
Analogously, 
$Q_{m'}^{+} \in GT(\lambda)$ is the Gelfand-Tsetlin pattern which
  satisfies \eqref{eq:5.13(1)}, \eqref{eq:5.11}, \eqref{eq:5.13(4)} and 
$q_{2n-2,m'} = q_{2n-3,m'} = q_{2n-4,m'-1}$ 
(or $=\lambda_{m-1}$ if $m'=m-1$). 
\end{definition}
\begin{theorem}\label{thm:corner}
If $c(Q_{n-1}^{+}; a)$ is known, then 
it determines all the $c(Q; a)$ for $Q \in GT(\lambda)$ containing the
same $\vect{q}_{1}, \dots, \vect{q}_{2n-4}$ parts as $Q_{n-1}^{+}$. 
\end{theorem}
\begin{proof}
In this proof, we assume all $Q \in GT(\lambda)$ contain the same
$\vect{q}_{1}, \dots, \vect{q}_{2n-4}$ as $Q_{n-1}^{+}$. 
Note that this $\vect{q}_{2n-4}$ satisfies \eqref{eq:5.11}. 
For simplicity, we prove the case when $\lambda_{n} > 0$ and $m < n$. 
The case when $\lambda_{n} < 0$ is proved just in the same way. 
The case when $m = n$ is also proved just in the same way, and the
proof is a little easier. 

For 
$Q = (\vect{q}_{1}, \dots, \vect{q}_{2n-1}) 
\in
GT(\lambda)$, 
define lengths of $Q$ by 
\begin{align*}
\Vert Q \Vert_{2n-2,1} 
:=& 
\sum_{j=1}^{m-2} 
(q_{2n-2,j} - \lambda_{j+1}), 
\qquad \qquad 
\Vert Q \Vert_{2n-2,2} 
:= 
\sum_{j=m-1}^{n-2} 
(q_{2n-2,j} - \lambda_{j+1}), 
\\
\Vert Q \Vert_{2n-3,1} 
:=& 
\sum_{j=1}^{m-2} 
(q_{2n-3,j} - q_{2n-4,j}),  
\\
\Vert Q \Vert_{2n-3,2} 
:=& 
\sum_{j=m-1}^{n-3} 
(q_{2n-3,j} - q_{2n-4,j}) 
\\
& \quad 
+ 
q_{2n-3,n-2} - q_{2n-3,n-1} 
- |q_{2n-2,n-1} - q_{2n-4,n-2}|.
\end{align*}
Note that, since 
$q_{2n-3,n-2} \geq q_{2n-2,n-1} \geq q_{2n-3,n-1}$ and 
$q_{2n-3,n-2} \geq q_{2n-4,n-2} \geq q_{2n-3,n-1}$, 
the term 
\begin{align*}
q_{2n-3,n-2} &- q_{2n-3,n-1} 
- |q_{2n-2,n-1} - q_{2n-4,n-2}| 
\\
&= 
q_{2n-3,n-2} - q_{2n-3,n-1} 
\\
& \quad 
- \max\{q_{2n-2,n-1}, q_{2n-4,n-2}\} 
+ 
\min\{q_{2n-2,n-1}, q_{2n-4,n-2}\} 
\end{align*}
is zero if and only if either (i) 
$q_{2n-3,n-2} = q_{2n-2,n-1}$ and $q_{2n-3,n-1} = q_{2n-4,n-2}$ or 
(ii) $q_{2n-3,n-2} = q_{2n-4,n-2}$ $q_{2n-3,n-1} = q_{2n-2,n-1}$. 
Define a partial order $Q' \prec Q$ in $GT(\lambda)$ by 
\begin{align*}
Q' \preceq Q
\quad \Leftrightarrow \quad 
& 
Q' = Q \mbox{ or } Q' \prec Q, 
\\
Q' \prec Q
\quad \Leftrightarrow \quad 
& 
\Vert Q' \Vert_{2n-2,1} < \Vert Q \Vert_{2n-2,1}; 
\\
&
\mbox{ or }
\Vert Q' \Vert_{2n-2,1} = \Vert Q \Vert_{2n-2,1} 
\mbox{ and }  
\Vert Q' \Vert_{2n-3,1} < \Vert Q \Vert_{2n-3,1}; 
\\
&
\mbox{ or }
\Vert Q' \Vert_{2n-2,1} = \Vert Q \Vert_{2n-2,1},  
\Vert Q' \Vert_{2n-3,1} = \Vert Q \Vert_{2n-3,1} 
\\
& \qquad 
\mbox{ and }  
\Vert Q' \Vert_{2n-2,2} < \Vert Q \Vert_{2n-2,2}; 
\\
&
\mbox{ or }
\Vert Q' \Vert_{2n-2,1} = \Vert Q \Vert_{2n-2,1},  
\Vert Q' \Vert_{2n-3,1} = \Vert Q \Vert_{2n-3,1},  
\\
& \qquad 
\Vert Q' \Vert_{2n-2,2} = \Vert Q \Vert_{2n-2,2} 
\mbox{ and }  
\Vert Q' \Vert_{2n-3,2} < \Vert Q \Vert_{2n-3,2}. 
\end{align*}
We will show Theorem~\ref{thm:corner} by induction on this
partial order. 

\noindent
\textbf{Step 1.} 
If $Q$ satisfies \eqref{eq:5.13(1)} and 
\begin{equation}\label{eq:q', q'' min-2}
\begin{cases}
q_{2n-2,p} = \lambda_{p+1} \enskip \mbox{ for } p \in [m-1, n-2],  
\\
q_{2n-3,p} = q_{2n-4,p} \enskip \mbox{ for } p \in [m-1, n-3], 
\\
\lambda_{n-1} > q_{2n-3,n-2} = q_{2n-2,n-1} \geq q_{2n-4,n-2},
\\
q_{2n-3,n-1} = q_{2n-4,n-2},
\end{cases}
\end{equation}
then the equation \eqref{eq:5.10(2)}, 
applied to this $Q$ and $j = n-1$, 
is 
\begin{align*}
& (\mathcal{D}_{2}(Q) + l_{2n-2,n-1}) c(Q; a)
\\
& 
+ 
\frac{\eta_{1}}{a_{1}} 
\frac{B_{n-1}(n-1, Q) a_{2n-3,-n+2}(\tau_{n-2,n-1} Q)}
{l_{2n-3,n-2} - l_{2n-2,n-1}} 
c(\tau_{n-2,n-1} Q; a) 
\\
& 
= 0.
\end{align*}
If $Q$ satisfies \eqref{eq:5.13(1)} and 
\begin{equation}\label{eq:q', q'' min-1}
\begin{cases}
q_{2n-2,p} = \lambda_{p+1}, \enskip \mbox{ for } p \in [m-1, n-2], 
\\
q_{2n-3,p} = q_{2n-4,p} \enskip \mbox{ for } p \in [m-1, n-2], 
\\ 
q_{2n-4,n-2} \geq q_{2n-2,n-1} = q_{2n-3,n-1} > \lambda_{n},
\end{cases}
\end{equation}
then the equation \eqref{eq:5.10(2)} for this $Q$ and $j = -(n-1)$ is 
\begin{align*}
& 
(\mathcal{D}_{2}(Q) + l_{2n-2,-n+1}) c(Q; a)
\\
& 
+ 
\frac{\eta_{1}}{a_{1}} 
\frac{B_{-n+1}(-n+1, Q) a_{2n-3,n-1}(\tau_{-n+1,-n+1} Q)}
{l_{2n-3,-n+1} - l_{2n-2,-n+1}} 
c(\tau_{-n+1,-n+1} Q; a) 
\\
&= 0.
\end{align*}
By repeating these shift operations, 
we know that $c(Q_{n-1}^{+}; a)$ determines the $c(Q; a)$ 
for $Q \in GT(\lambda)$ which satisfies 
(i) \eqref{eq:5.13(1)} and \eqref{eq:q', q'' min-2} or 
(ii) \eqref{eq:5.13(1)} and \eqref{eq:q', q'' min-1}, 
in other words, for $Q$ which is minimal with respect to the partial
order $\preceq$. 

\noindent
\textbf{Step 2.} 
Suppose $Q \in GT(\lambda)$ in question satisfies 
$\tau_{i,0} Q \in GT(\lambda)$ for an 
$i \in [1,n-2] \cup \{-n+1\}$. 
Then $\tau_{i,0} Q \succ Q$ since 
$\Vert \tau_{i,0} Q \Vert_{2n-2,p} 
=
\Vert Q \Vert_{2n-2,p}$ for $p = 1, 2$ but 
$\Vert \tau_{i,0} Q \Vert_{2n-3,p} > \Vert Q \Vert_{2n-3,p}$ 
for $p = 1$ or $2$. 
Consider the equation \eqref{eq:5.13(3)-1} applied to $m'=n-1$. 
If $j \in I_{1}(n-1) \setminus \{0\} = K_{1}(n-1) \subset [1, n-2]$, then 
$\tau_{0,-j} Q \prec Q$ since 
$\Vert \tau_{0,-j} Q \Vert_{2n-2,p} < \Vert Q \Vert_{2n-2,p}$ for 
$p  =1$ or $2$. 
If $j \in I_{1}(n-1)$, $j' \in [-m+2, -1]$ and $i' \in [-n+1, n-1]$, 
then $\tau_{i',j'} \tau_{0,j} Q \prec Q$ since 
$\Vert \tau_{i',j'} \tau_{0,j} Q \Vert_{2n-2,1} 
< 
\Vert Q \Vert_{2n-2,1}$. 
Moreover, if $i' \in [-n+2, -1] \cup \{n-1\}$, 
then $\tau_{i',0} Q \prec Q$ since 
$\Vert \tau_{i',0} Q \Vert_{2n-3,p} < \Vert Q \Vert_{2n-3,p}$ for 
$p = 1$ or $2$. 
It follows that all the  $c(Q'; a)$'s appearing in the right hand of
\eqref{eq:5.13(3)-1} satisfy $Q' \preceq Q$. 
Therefore, $c(\tau_{i,0} Q; a)$ can be expressed as 
\begin{equation*}
c(\tau_{i,0} Q; a)  
= 
\sum_{Q' \preceq Q(\prec \tau_{i,0} Q)} 
\mbox{\rm (differential of $c(Q'; a)$)}, 
\end{equation*}
and it is determined by $c(Q_{n-1}^{+}; a)$ by the hypothesis of
induction. 
Especially, if $Q \in GT(\lambda)$ satisfies 
\eqref{eq:5.13(1)} and $q_{2n-2,p} = \lambda_{p+1}$ for any 
$p \in [m-1, n-2]$, then we know from the result of Step 1 
that $c(Q_{n-1}^{+}; a)$ determines $c(Q; a)$.

\noindent
\textbf{Step 3.} 
Suppose $Q \in GT(\lambda)$ satisfies \eqref{eq:5.13(1)}. 
Let $k \in K_{6}(n-1)$. 
Consider the equation \eqref{eq:5.13(3)-2} applied to 
$m'=n-1$ and $j = k$. 
Since $j' \in I_{1}(n-1)$ implies 
$\tau_{0,-j'} Q 
\preceq Q$ and 
$i \in [-n+2,-m+1] \cup \{n-1\}$ implies 
$\tau_{i,0} Q \prec Q$,  
all the $c(Q'; a)$ appearing in the right hand of \eqref{eq:5.13(3)-2} 
satisfies $Q' \preceq Q$. 
Therefore, 
$c(\tau_{0,k} Q; a)$ satisfies 
\begin{equation}\label{eq:prec-2}
(\mathcal{D}_{2}^{\pm}(Q) 
- 
l_{2n-2, k})  
c(\tau_{0,k}Q; a) 
= 
\sum_{Q' \preceq Q} 
\mbox{\rm (differential of $c(Q'; a)$)}. 
\end{equation}
The equation \eqref{eq:5.9(2)} for $j = m-1$ implies 
\begin{equation}\label{eq:prec-3}
\sum_{k \in K_{6}(n-1)} 
A_{k}(Q) c(\tau_{0,k} Q; a)  
=
\sum_{Q' \preceq Q} 
\mbox{(differential of $c(Q'; a)$)} 
\end{equation}
since $\tau_{0,j'} Q \prec Q$ for $j' \in [-n+2, -m+1]$. 

In order to eliminate extra terms in the left hand of
\eqref{eq:prec-3}, consider the following differential operator. 
For any $m' \in [m-1,n-1]$ and $j, k \in K_{6}(m')$, 
the fraction in the right hand of 
\begin{align}
& 
\prod_{p \in K_{6}(m')\setminus \{j\}}
(\mathcal{D}_{2}^{\pm}(Q) - l_{2n-2,p})
\label{eq:trick}
\\
&= 
\frac{
\prod_{p \in K_{6}(m')\setminus \{j\}}
(\mathcal{D}_{2}^{\pm}(Q) - l_{2n-2,p})
-  
\prod_{p \in K_{6}(m')\setminus \{j\}}
(l_{2n-2,k} - l_{2n-2,p})
}
{\mathcal{D}_{2}^{\pm}(Q) - l_{2n-2,k}}
\notag
\\
& \quad 
\times 
(\mathcal{D}_{2}^{\pm}(Q) - l_{2n-2,k})
\notag
\\
& \quad
+ 
\prod_{p \in K_{6}(m')\setminus \{j\}}
(l_{2n-2,k} - l_{2n-2,p})
\notag
\end{align}
is a polynomial in $\mathcal{D}_{2}^{\pm}(Q)$, 
so it is a differential operator.  
Note that the last term in the right hand is zero if $k \not= j$. 

Choose $j \in K_{6}(n-1) \cap [m-1, n-2]$. 
We have $\tau_{0,j} Q \succ Q$, since $j \in [m-1,n-2]$ implies 
$\Vert \tau_{0,j} Q \Vert_{2n-2,2} > \Vert Q \Vert_{2n-2,2}$. 
Differentiate both sides of \eqref{eq:prec-3} by 
$\prod_{p \in K_{6}(n-1) \setminus \{j\}}
(\mathcal{D}_{2}^{\pm}(Q) - l_{2n-2,p})$ 
and use \eqref{eq:prec-2} and \eqref{eq:trick}. 
Then we get 
\begin{equation*}
c(\tau_{0,j} Q; a)  
=
\sum_{Q' \prec \tau_{0,j} Q} 
\mbox{\rm (differential of $c(Q'; a)$)}.
\end{equation*}
By the hypothesis of induction, $c(\tau_{0,j} Q; a)$ is determined by
$c(Q_{n-1}^{+}; a)$, and so is $c(\tau_{i,j} Q; a)$ for 
$i \in [m-1,n-2]$ because of Step 2. 
Then we know from the result of Step 2 that $c(Q_{n-1}^{+}; a)$
determines the $c(Q; a)$ for all $Q \in GT(\lambda)$ satisfying
\eqref{eq:5.13(1)}. 

\noindent
\textbf{Step 4.} 
If $\tau_{0,j} Q \in GT(\lambda)$ for $j \in [1,m-2]$, then the
equation \eqref{eq:5.9(2)}, applied to this $j$, implies 
\begin{equation*}
c(\tau_{0,j} Q; a) 
= 
\sum_{Q' \prec \tau_{0,j} Q} 
\mbox{\rm (differential of $c(Q'; a)$)}.
\end{equation*}
By the hypothesis of induction, $c(\tau_{0,j} Q; a)$ is determined by 
$c(Q_{n-1}^{+}; a)$, and so is $c(\tau_{i,j} Q; a)$ for 
$i \in [1,m-1]$ because of Step 2. 
Then we know from the result of Step 3 that $c(Q_{n-1}^{+}; a)$
determines all the $c(Q; a)$ for $Q \in GT(\lambda)$. 
\end{proof}

\section{Determination of Whittaker models}
\label{section:solution} 

In this section, we first deduce a system of differential equations
which are satisfied by $c(Q; a)$ for $Q \in GT(\lambda)$ satisfying
\eqref{eq:5.13(1)}, \eqref{eq:5.11} and \eqref{eq:5.13(4)}. 
Secondly obtained are the Mellin-Barnes type integral formulas of the
solutions of this system of equations. 
Though the dimension of the solution space is high, only a few of the
solutions satisfy the whole differential-difference equations
$\mathcal{D}_{\tilde{\lambda}, \eta} \phi = 0$. 
Lastly, continuous intertwining operators are determined. 

\subsection{Scalar differential equations}

In this subsection, we assume that $Q$ satisfies \eqref{eq:5.13(1)}, 
\eqref{eq:5.11} and \eqref{eq:5.13(4)}. In this case, 
\begin{align*}
K_{1}(m')  
&= 
\{
p \in [m-1,m'-1] | 
q_{2n-3,p} < q_{2n-2,p}\},
\\
K_{2}(m') 
&= [1,m'-1] \setminus K_{1} 
= 
[1, m-2] \cup 
\{
p \in [m-1,m'-1] | 
q_{2n-3,p} = q_{2n-2,p}\},
\\
K_{3}(m') 
&= 
\{
p \in [-n+2, -m'] | 
q_{2n-3,-p} > q_{2n-2,-p+1} = \lambda_{-p+1}\}, 
\\
K_{4}(m') 
&= 
[-n+2, -m'-1] \setminus K_{3} 
\\
&= 
\{
p \in [-n+2, -m'-1] | 
q_{2n-3,-p} = q_{2n-2,-p+1} = \lambda_{-p+1}\}, 
\\
K_{5}(m') 
&= 
\{p-1 | p \in K_{3}(m')\} 
\\
&= 
\{p \in [-n+1, -m' -1] | q_{2n-3,-p-1} > q_{2n-2,-p} = \lambda_{-p}\}, 
\end{align*}
by definition. 

\begin{lemma}
Let 
\begin{align*}
& 
J(m') := [-n+1, -m'-1] \cup [m-1,m'-1], 
\\
& 
d_{m'}(Q) := 
\sum_{p \in J(m') \cup \{m'\}} l_{2n-2,-p} 
+ 
\sum_{p \in J(m')} l_{2n-3,p}. 
\end{align*}
Suppose $Q \in GT(\lambda)$ satisfies \eqref{eq:5.13(1)},
\eqref{eq:5.11} and \eqref{eq:5.13(4)}. 
\begin{enumerate}
\item
If  $\tau_{-m',0} Q \in GT(\lambda)$, then 
\begin{align}
-& 
\I \frac{\eta_{1}}{a_{1}} 
\frac{a_{2n-3,m'}(\tau_{-m',0} Q) 
\prod_{p \in J(m')}(l_{2n-3,-m'} - l_{2n-3,p})}
{\prod_{p \in J(m') \cup \{0\}} 
(l_{2n-3,-m'} + l_{2n-2,-p})} 
c(\tau_{-m',0} Q; a) 
\label{eq:5.14(1)} 
\\
&= 
\left(
\pm 
\frac{l_{2n-1,n}}{l_{2n-2,-m'}} 
\mathcal{D}_{2}^{\pm}(Q) + \mathcal{D}_{1}^{\mp}(Q) 
- d_{m'}(Q) + l_{2n-2,-m'} 
\right) 
c(Q; a). 
\notag 
\end{align}
\item
If $q_{2n-2,m'} = q_{2n-3,m'} > \lambda_{m'+1}$,
then 
\begin{align}
\mp
& 
\frac{\I}{B_{m'}(m', \tau_{-m', -m'} Q)} 
\frac{(2 l_{2n-2,-m'}+ 1) 
\prod_{p \in J(m')} (l_{2n-2,-m'} + l_{2n-3,p})}
{\prod_{p \in J(m') \cup \{0\}} (l_{2n-2,-m'} - l_{2n-2,-p})} 
\label{eq:5.15(1)} 
\\
& \qquad 
\times (\mathcal{D}_{2}^{\pm}(Q) - l_{2n-2,-m'}) 
c(\tau_{-m',-m'} Q; a) 
\notag 
\\
& =
\left(
\mp 
\frac{l_{2n-1,n}}{l_{2n-2,-m'}} \mathcal{D}_{2}^{\pm}(Q) 
+ \mathcal{D}_{1}^{\mp}(Q) - d_{m'}(Q) - l_{2n-2,-m'} - 1 
\right) 
c(\tau_{-m',0} Q; a) 
\notag
\\
& \qquad 
+ 
\I \frac{\eta_{1}}{a_{1}} 
\frac{a_{2n-3,-m'}(Q) 
(2 l_{2n-3,m'} - 2) 
\prod_{p \in J(m')}(l_{2n-3,m'} - l_{2n-3,p} - 1)}
{\prod_{p \in J(m') \cup \{0, m'\}} 
(l_{2n-3,m'} + l_{2n-2,-p} - 1)} 
c(Q; a). 
\notag 
\end{align}
\item
If $q_{2n-2,m'} = q_{2n-3,m'} > \lambda_{m'+1}$,
then 
\begin{align}
c(\tau_{-m',-m'} Q; a) 
&= 
\mp 
\frac{a_{1}}{\eta_{1}} 
\frac{l_{2n-3,-m'} - l_{2n-2,-m'} + 1}
{a_{2n-2,m'}(\tau_{-m',-m'} Q) a_{2n-3,m'}(\tau_{-m',0} Q)} 
\label{eq:5.15(2)} 
\\
& \qquad 
\times 
\frac{\prod_{p \in [-n,-1] \cup [m,n]}
(l_{2n-2,-m'} + l_{2n-1,p})}
{\prod_{p \in [-n+1, -m+1] \cup [0,n-1] \setminus \{-m'\}}
(l_{2n-2,-m'} - l_{2n-2,p})} 
\notag\\
& \qquad 
\times 
(\mathcal{D}_{2}^{\pm}(Q) + l_{2n-2,-m'}) c(Q; a). 
\notag
\end{align}
\item
For $k \in K_{6}(m')$ 
\begin{align}
\mp& 
\frac{\I}{B_{-k}(-k, \tau_{0, k} Q)} 
\frac{\prod_{p \in J(m') \cup \{-m'\}} (l_{2n-2,k} + l_{2n-3,p})}
{\prod_{p \in J(m') \cup \{0\}} (l_{2n-2,k} - l_{2n-2,-p})} 
(\mathcal{D}_{2}^{\pm}(Q) - l_{2n-2,k}) c(\tau_{0,k} Q; a) 
\label{eq:5.19(1)} 
\\
&= 
\left(
\mp\frac{l_{2n-1,n}}{l_{2n-2,k}} \mathcal{D}_{2}^{\pm}(Q) 
+ \mathcal{D}_{1}^{\mp}(Q) - d_{m'}(Q) - l_{2n-2,k}  
\right) 
c(Q; a) 
\notag 
\end{align}
\end{enumerate}
\end{lemma}
\begin{proof} 
(1) 
Since $-m' \in K_{3}(m')$, 
the equation \eqref{eq:5.13(3)-1} holds for $i = -m'$. 
If $j \in I_{1}(m') \setminus \{0\} = K_{1}(m') \cup K_{5}(m')$, 
then $\tau_{0,-j} Q \not\in GT(\lambda)$ because of
\eqref{eq:5.13(4)}. 
If $j' \in [-m+2, -1]$ and $j \in I_{1}(m')$, 
then $\tau_{i',j'} \tau_{0,-j} Q \not\in GT(\lambda)$ for all $i'$
because of \eqref{eq:5.13(1)}. 
If $i' \in [-m'+1,-1] \cup [m',n-1]$, then $\tau_{i',0} Q \not\in
GT(\lambda)$ because of \eqref{eq:5.13(1)} and \eqref{eq:5.13(4)}. 
Therefore, the terms in the right hand of \eqref{eq:5.13(3)-1} vanish
except for the third line and the 
$(\mathcal{D}_{2}^{\pm}(Q) - l_{2n-2,0}) c(Q; a)$ term in the first
line. 
If we arrange the coefficients by using $l_{2n-3,p} = - l_{2n-2,-p}$
for $q \in K_{2}(m')$ and 
$l_{2n-3,p} = - l_{2n-2,-p+1}$ for $q \in K_{4}(m')$, 
we get \eqref{eq:5.14(1)}.  

(2) 
Since $\tau_{0,-m'} \tau_{-m',0} Q = \tau_{-m',-m'} Q \in GT(\lambda)$, 
the equation \eqref{eq:5.13(3)-2}, with $Q$ replaced by 
$\tau_{-m',0} Q \in GT(\lambda)$ and with $j=-m'$, holds. 
Since 
$\tau_{0,-j'} \tau_{-m',0} Q \not\in GT(\lambda)$ 
for $j' \in K_{1}(m') \cup K_{5}(m')$
and 
$\tau_{i',0} \tau_{-m',0} Q \not\in GT(\lambda)$ for 
$i' \in [-m'+1,-m+1] \cup [m'+1, n-1]$, 
the terms in the right hand of \eqref{eq:5.13(3)-2} vanish except for
the second line, 
the $(\mathcal{D}_{2}^{\pm}(Q) - l_{2n-2,0}) c(\tau_{-m',0} Q; a)$ term
in the first line and $c(\tau_{m',0} \tau_{-m',0} Q; a)$ term in the
third line. 
Thus we get \eqref{eq:5.15(1)}. 

(3)
For this $Q$, $\tau_{i,-m'} Q \in GT(\lambda)$ if and only if 
$i = -m'$. 
\eqref{eq:5.15(1)} is the equation \eqref{eq:5.10(2)} applied 
to $j = -m'$. 

(4) 
Consider the equation \eqref{eq:5.13(3)-2} applied to $j = k$. 
Since $\tau_{0,-j'} Q \not\in GT(\lambda)$ if 
$j' \in K_{1}(m') \cup K_{5}(m')$ and 
$\tau_{i',0} Q \not\in GT(\lambda)$ if 
$i' \in [-m'+1,-m+1] \cup [m', n-1]$, 
the terms in the right hand of \eqref{eq:5.13(3)-2} vanish except for
the second line and the 
$(\mathcal{D}_{2}^{\pm}(Q) - l_{2n-2,0}) c(Q; a)$ term 
in the first line. 
Thus we get \eqref{eq:5.19(1)}. 
\end{proof}

\begin{proposition}
If $Q \in GT(\lambda)$ satisfies \eqref{eq:5.13(1)}, \eqref{eq:5.11}
and \eqref{eq:5.13(4)}, then $c(Q; a)$ is a solution of 
\begin{equation}\label{eq:5.16(1)} 
\left\{
(\mathcal{D}_{1}^{\mp}(Q) - d_{m'}(Q))^{2} 
- (\mathcal{D}_{2}^{\pm}(Q) \pm l_{2n-1,n})^{2} 
- 
\left(
\frac{\eta_{1}}{a_{1}}\right)^{2} 
\right\} 
c(Q; a) 
= 0. 
\end{equation}
\end{proposition}
\begin{proof}
Suppose $q_{2n-2,m'} = q_{2n-3,m'} > \lambda_{m'+1}$. 
Then \eqref{eq:5.16(1)} is obtained 
from 
\eqref{eq:5.14(1)}, \eqref{eq:5.15(1)} and \eqref{eq:5.15(2)}, 
by eliminating 
$c(\tau_{-m',0} Q; a)$ and $c(\tau_{-m',-m'} Q; a)$. 

Next, suppose $q_{2n-2,m'} = q_{2n-3,m'} = \lambda_{m'+1}$. 
The equations \eqref{eq:5.15(1)} and \eqref{eq:5.15(2)}, with $Q$
replaced by $\tau_{m',m'} Q$, are 
difference equations
for $Q$, $\tau_{m',m'} Q$ and $\tau_{0,m'} Q$. 
The equation \eqref{eq:5.19(1)}, applied to $k=m'$, is a
difference equation for $Q$ and $\tau_{0,m'} Q$. 
By eliminating $\tau_{m',m'} Q$ and $\tau_{0,m'} Q$ from these, 
we get the equation \eqref{eq:5.16(1)}. 
\end{proof}

In order to deduce another differential equation, 
we prepare two identities. 

\begin{lemma}\label{lemma:5.18}
For $x_{1}, \dots, x_{k}$, $y_{1}, \dots, y_{k}$ and $z$, 
\begin{align}
&
\sum_{i=1}^{k} 
\frac{
\prod_{j=1}^{k-1} (x_{i} + y_{j}) 
\prod_{i'=1, \not= i}^{k} (z - x_{i'})}
{\prod_{i'=1, \not= i}^{k} (x_{i} - x_{i'})} 
= 
\prod_{i=1}^{k-1} (z+y_{i}). 
\label{eq:summation-1}
\\
&
\sum_{i=1}^{k} 
\frac{\prod_{j=1}^{k} (x_{i} + y_{j})}
{(z+x_{i}) 
\prod_{i'=1, \not= i}^{k} (x_{i} - x_{i'})} 
= 
1 - 
\frac{\prod_{i=1}^{k} (z-y_{i})}
{\prod_{i=1}^{k} (z+x_{i})}. 
\label{eq:summation-2}
\end{align}
\end{lemma}

\begin{proof} 
\eqref{eq:summation-1}  
The left hand is a polynomial in $z$ of degree $k-1$, which we denote
by $f(z)$. 
For $l \in [1, k]$, 
$
f(x_{l}) 
= 
\prod_{i=1}^{k-1} (x_{l} + y_{i})$. 
On the other hand, the polynomial $\prod_{i=1}^{k-1} (z+y_{i})$ of
degree $k-1$ has the same values at $k$ points $x_{l}$, 
so they are identical. 

\eqref{eq:summation-2} 
Let $g(z) := \mbox{(left hand)} \times \prod_{i=1}^{k} (z+x_{i})$,
which is a polynomial in $z$ of degree $k-1$. 
For $l \in [1, k]$, 
$
g(-x_{l}) 
= 
(-1)^{k-1} 
\prod_{i=1}^{k} (x_{l} + y_{i})$. 
On the other hand, the polynomial 
$\prod_{i=1}^{k} (z+x_{i}) - \prod_{i=1}^{k} (z-y_{i})$ of degree
$k-1$ has the same values at $k$ points $z = -x_{l}$, 
so they are identical. 
\end{proof}

Suppose $Q \in GT(\lambda)$ satisfies \eqref{eq:5.13(1)}, 
\eqref{eq:5.11}, \eqref{eq:5.13(4)} and $q_{2n-2,m-1} < \lambda_{m-1}$. 
In this case, $K_{6}(m')$ is not empty since $m-1 \in K_{6}(m')$. 
The equation \eqref{eq:5.9(2)}, applied to this $Q$ and $j=m-1$, is 
\begin{align} 
& 
\left\{
\mathcal{D}_{2}^{\pm}(Q) \mp \Lambda_{n+1} 
+ \frac{\eta_{2} a_{1}}{2 a_{2}} 
(\I A_{0}(Q) \mp 1) 
\right\} 
c(Q; a) 
%
+ 
\I 
\frac{\eta_{2} a_{1}}{2 a_{2}}
\sum_{k \in K_{6}(m')} 
A_{k}(Q) c(\tau_{0,k} Q; a)  
\label{eq:5.19(2)} 
\\
&= 0.
\notag
\end{align}
Let $j \in K_{6}(m')$. 
After applying 
$\prod_{p \in K_{6}(m') \setminus \{j\}}
(\mathcal{D}_{2}^{\pm}(Q) - l_{2n-2,p} + 1)$ 
to both sides of \eqref{eq:5.19(2)}, 
use \eqref{eq:trick} and \eqref{eq:5.19(1)}. 
Then we get 
\begin{align}
& - 
\I 
\frac{\eta_{2} a_{1}}{2 a_{2}}
A_{j}(Q) 
\left(
\prod_{p \in K_{6}(m') \setminus \{j\}}
(l_{2n-2,j} - l_{2n-2,p})
\right) 
c(\tau_{0,j} Q; a)  
\label{eq:5.19(5)}
\\ 
&= 
\Biggl[
\prod_{p \in K_{6}(m') \setminus \{j\}}
(\mathcal{D}_{2}^{\pm}(Q) - l_{2n-2,p} + 1)
\left\{
\mathcal{D}_{2}^{\pm}(Q) \mp \Lambda_{n+1} 
+ \frac{\eta_{2} a_{1}}{2 a_{2}} 
(\I A_{0}(Q) \mp 1) 
\right\} 
\notag
\\
& \quad 
\pm \I 
\frac{\eta_{2} a_{1}}{2 a_{2}}
\sum_{k \in K_{6}(m')} 
A_{k}(Q) 
\I B_{-k}(-k, \tau_{0, k} Q) 
\frac{\prod_{p \in J(m') \cup \{0\}} (l_{2n-2,k} - l_{2n-2,-p})} 
{\prod_{p \in J(m') \cup \{-m'\}} (l_{2n-2,k} + l_{2n-3,p})}
\notag
\\
& \hspace{10mm} \times 
\frac{
\prod_{p \in K_{6}(m') \setminus \{j\}}
(\mathcal{D}_{2}^{\pm}(Q) - l_{2n-2,p})
-  
\prod_{p \in K_{6}(m') \setminus \{j\}}
(l_{2n-2,k} - l_{2n-2,p})
}
{\mathcal{D}_{2}^{\pm}(Q) - l_{2n-2,k}}
\notag
\\
& \hspace{10mm} \times 
\left(
\mp \frac{l_{2n-1,n}}{l_{2n-2,k}} \mathcal{D}_{2}^{\pm}(Q) 
+ \mathcal{D}_{1}^{\mp}(Q) - d_{m'}(Q) - l_{2n-2,k} 
\right) 
\Biggr]
c(Q; a). 
\notag
\end{align}
Let 
\begin{align}
K_{7}(m') 
& := 
\{p \in [m-1, m'-1] | p+1 \in K_{6}(m')\} 
\label{eq:K_7} 
\\
& \qquad \cup 
\{p \in [-n+1, -m'-1] | p \in K_{6}(m')\}.
\notag
\end{align}
Note that $\# K_{6}(m') = \# K_{7}(m') +1$ if 
$q_{2n-2,m-1} < \lambda_{m-1}$. 
This is because $m-1 \in K_{6}(m')$ if $q_{2n-2,m-1} < \lambda_{m-1}$,
while there is no element in $K_{7}(m')$ corresponding 
to $m-1 \in K_{6}(m')$. 
It is not hard to see that 
\[
\I A_{0}(Q) 
= 
\frac{\prod_{p=1}^{n-1} l_{2n-3,p} \prod_{p=1}^{n} l_{2n-1,p}}
{\prod_{p=1}^{n-1} l_{2n-2,p} (l_{2n-2,p} - 1)} 
\frac{\prod_{p=1}^{m-2} (-l_{2n-2,p})}
{\prod_{p=1}^{m-1} (-l_{2n-1,p})}
= 
\frac{l_{2n-1,n} \prod_{p \in K_{7}(m')} l_{2n-3,p}}
{\prod_{p \in K_{6}(m')} l_{2n-2,p}} 
\]
and 
\begin{align*}
& 
\I A_{k}(Q)\, 
\I B_{-k}(-k, \tau_{0, k} Q) 
\frac{\prod_{p \in I_{1}(m')} (l_{2n-2,k} - l_{2n-2,-p})} 
{\prod_{p \in I_{2}(m')} (l_{2n-2,k} + l_{2n-3,p})}
\\
& \hspace{45mm} 
= 
\frac{
\prod_{p \in K_{7}(m')} (l_{2n-2,k} - l_{2n-3,p})
}
{\prod_{p \in K_{6}(m') \setminus \{k\}} 
(l_{2n-2,k} - l_{2n-2,p})}.  
\end{align*}
It follows that the last term in \eqref{eq:5.19(5)} is 
$\pm \eta_{2} a_{1}/ 2 a_{2}$ times 
\begin{align*}
& 
\sum_{k \in K_{6}(m')} 
\frac{
\prod_{p \in K_{7}} (l_{2n-2,k} - l_{2n-3,p})
}
{\prod_{p \in K_{6}(m') \setminus \{k\}} 
(l_{2n-2,k} - l_{2n-2,p})} 
\\
& \hspace{6mm} 
\times 
\frac{
\prod_{p \in K_{6}(m') \setminus \{j\}}
(\mathcal{D}_{2}^{\pm}(Q) - l_{2n-2,p})
-  
\prod_{p \in K_{6}(m') \setminus \{j\}}
(l_{2n-2,k} - l_{2n-2,p})
}
{\mathcal{D}_{2}^{\pm}(Q) - l_{2n-2,k}}
\notag
\\ 
& \hspace{6mm}
\times 
\left(
\left(
1 \mp \frac{l_{2n-1,n}}{l_{2n-2,k}} 
\right)
(\mathcal{D}_{2}^{\pm}(Q) - l_{2n-2,k})
+ \mathcal{D}_{1}^{\mp}(Q) - d_{m'}(Q) - \mathcal{D}_{2}^{\pm}(Q) 
\mp l_{2n-1,n}
\right) 
\\
&= 
\left(
\sum_{k \in K_{6}(m')} 
\frac{l_{2n-2,k} \mp l_{2n-1,n}}{l_{2n-2,k}} 
\frac{
\prod_{p \in K_{7}(m')} (l_{2n-2,k} - l_{2n-3,p})
}
{\prod_{p \in K_{6}(m') \setminus \{k\}} 
(l_{2n-2,k} - l_{2n-2,p})} 
\right) 
\\
& \hspace{6mm}
\times 
\prod_{p \in K_{6}(m') \setminus \{j\}}
(\mathcal{D}_{2}^{\pm}(Q) - l_{2n-2,p})
\notag
\\
& \hspace{6mm}
+ 
\sum_{k \in K_{6}(m')} 
\frac{
\prod_{p \in K_{7}(m')} (l_{2n-2,k} - l_{2n-3,p})
}
{\prod_{p \in K_{6}(m') \setminus \{k\}} 
(l_{2n-2,k} - l_{2n-2,p})} 
\frac{
\prod_{p \in K_{6}(m') \setminus \{k\}}
(\mathcal{D}_{2}^{\pm}(Q) - l_{2n-2,p})}
{\mathcal{D}_{2}^{\pm}(Q) - l_{2n-2,j}}
\notag
\\
& \hspace{6mm} \times 
\left(
\mathcal{D}_{1}^{\mp}(Q) - d_{m'}(Q) - \mathcal{D}_{2}^{\pm}(Q) 
\mp l_{2n-1,n}
\right) 
\\
& \hspace{6mm}
- 
\prod_{p \in K_{7}(m')} (l_{2n-2,j} - l_{2n-3,p})
\frac{1}{\mathcal{D}_{2}^{\pm}(Q) - l_{2n-2,j}}
\notag
\\
& \hspace{6mm} \times 
\left(
\left(
1 \mp \frac{l_{2n-1,n}}{l_{2n-2,j}} 
\right)
(\mathcal{D}_{2}^{\pm}(Q) - l_{2n-2,j})
+ \mathcal{D}_{1}^{\mp}(Q) - d_{m'}(Q) - \mathcal{D}_{2}^{\pm}(Q) 
\mp l_{2n-1,n}
\right) 
\notag
\\
&
=
\left(
1 
\mp \frac{l_{2n-1,n} \prod_{p \in K_{7}(m')} l_{2n-3,p}}
{\prod_{p \in K_{6}(m')} l_{2n-2,p}} 
\right)
\left(
\prod_{p \in K_{6}(m') \setminus \{j\}}
(\mathcal{D}_{2}^{\pm}(Q) - l_{2n-2,p})
\right)
\\
& \hspace{4mm}
+ 
\frac{
\prod_{p \in K_{7}(m')}
(\mathcal{D}_{2}^{\pm}(Q) - l_{2n-3,p}) 
- 
\prod_{p \in K_{7}(m')}
(l_{2n-2,j} - l_{2n-3,p})}
{\mathcal{D}_{2}^{\pm}(Q) - l_{2n-2,j}}
\\
& \hspace{8mm}
\times 
\left(
\mathcal{D}_{1}^{\mp}(Q) - d_{m'}(Q) - \mathcal{D}_{2}^{\pm}(Q) 
\mp l_{2n-1,n}
\right) 
\\
& \hspace{4mm}
- 
\frac{l_{2n-2,j} \mp l_{2n-1,n}}{l_{2n-2,j}} 
\prod_{p \in K_{7}(m')} (l_{2n-2,j} - l_{2n-3,p}).
\end{align*}
Here, we use Lemma~\ref{lemma:5.18} to get the last equality. 
Then \eqref{eq:5.19(5)} becomes the next equation. 
\begin{lemma}
If $Q \in GT(\lambda)$ satisfies \eqref{eq:5.13(1)}, \eqref{eq:5.11}
and \eqref{eq:5.13(4)}, 
then for $j \in K_{6}(m')$, 
\begin{align} 
-\I 
\frac{\eta_{2} a_{1}}{2 a_{2}}
& A_{j}(Q) 
\left(
\prod_{p \in K_{6}(m') \setminus \{j\}}
(l_{2n-2,j} - l_{2n-2,p})
\right) 
c(\tau_{0,j} Q; a)  
\label{eq:5.22}
\\
=&
\left(
\prod_{p \in K_{6}(m') \setminus \{j\}}
(\mathcal{D}_{2}^{\pm}(Q) - l_{2n-2,p} + 1)
\right) 
(
\mathcal{D}_{2}^{\pm}(Q) \mp \Lambda_{n+1} 
)
c(Q; a) 
\notag
\\
& \quad 
\pm 
\frac{\eta_{2} a_{1}}{2 a_{2}}
\Bigg[
\frac{
\prod_{p \in K_{7}(m')}
(\mathcal{D}_{2}^{\pm}(Q) - l_{2n-3,p}) 
- \prod_{p \in K_{7}(m')} (l_{2n-2,j} - l_{2n-3,p})
}
{\mathcal{D}_{2}^{\pm}(Q) - l_{2n-2,j}}
\notag\\
& \qquad \qquad \quad 
\times 
\left(
\mathcal{D}_{1}^{\mp}(Q) - d_{m'}(Q) - \mathcal{D}_{2}^{\pm}(Q) 
\mp l_{2n-1,n}
\right) 
\notag
\\
& \qquad \qquad \quad
- 
\frac{l_{2n-2,j} \mp l_{2n-1,n}}
{l_{2n-2,j}}
\prod_{p \in K_{7}(m')} (l_{2n-2,j} - l_{2n-3,p})
\Bigg]
c(Q; a) 
\notag
\end{align}
\end{lemma}

\begin{proposition}
If $Q \in GT(\lambda)$ satisfies \eqref{eq:5.13(1)}, \eqref{eq:5.11},
\eqref{eq:5.13(4)} and $q_{2n-2,m-1} < \lambda_{m-1}$, 
then $c(Q; a)$ is a solution of 
\begin{align}
& 
\left\{
(
\mathcal{D}_{2}^{\pm}(Q) \mp \Lambda_{n+1} 
) 
\prod_{k \in K_{6}(m')} 
(\mathcal{D}_{2}^{\pm}(Q) - l_{2n-2,k} + 1) 
\right.
\label{eq:5.20(3)} 
\\
& \hspace{6mm} 
\left. 
\pm 
\frac{\eta_{2} a_{1}}{2 a_{2}}
(\mathcal{D}_{1}^{\mp}(Q) - d_{m'}(Q) - \mathcal{D}_{2}^{\pm}(Q) 
\mp l_{2n-1,n})  
\prod_{p \in K_{7}(m')} 
(\mathcal{D}_{2}^{\pm}(Q) - l_{2n-3,p}) 
\right\}
\notag
\\
& 
\hspace{4mm} 
\times c(Q; a) 
\notag
\\
&= 0. 
\notag
\end{align} 
If $m' = m-1$ and $q_{2n-2,m-1} = \lambda_{m-1}$, 
then $c(Q; a)$ satisfies the equation \eqref{eq:5.20(3)} with
$K_{6}(m')$ replaced by $K_{6}(m-1) \cup \{m-1\}$. 
\end{proposition} 
\begin{proof}
Suppose $q_{2n-2,m-1} < \lambda_{m-1}$. 
In this case $K_{6}(m')$ is not empty since $m-1 \in K_{6}(m')$. 
Differentiate \eqref{eq:5.22} by 
$\mathcal{D}_{2}^{\pm}(Q) - l_{2n-2,j} + 1$, and use \eqref{eq:5.19(1)}. 
Then we get the equation \eqref{eq:5.20(3)}. 
Suppose $m' = m-1$ and $q_{2n-2,m-1} = \lambda_{m-1}$. 
Then $K_{6}(m-1)(\tau_{-m+1,-m+1} Q) = K_{6}(m-1)(Q) \cup \{m-1\}$. 
It follows that the equation \eqref{eq:5.22} holds for $j = m-1$, 
if we replace $Q$ by $\tau_{-m+1,-m+1} Q$. 
From this equation and the equations \eqref{eq:5.14(1)} and
\eqref{eq:5.15(2)}, applied to $m' = m-1$, 
we obtain \eqref{eq:5.20(3)} by eliminating $c(\tau_{-m+1,0}Q; a)$ 
and $c(\tau_{-m+1,-m+1} Q; a)$. 
\end{proof}


\subsection{Solutions of \eqref{eq:5.16(1)} and \eqref{eq:5.20(3)}}
\label{subsection:solutions of scalar equations}

Suppose $\Lambda \in \Xi_{m,\pm}$, $m = 2, \dots, n$, 
and $m' \in [m-1,n-1]$. 
Assume $Q \in GT(\lambda)$ satisfies \eqref{eq:5.13(1)},
\eqref{eq:5.11} and \eqref{eq:5.13(4)}. 
Define $j_{i}$ and $k_{i}$ by 
\begin{align*}
& 
\{j_{3}, \dots, j_{N_{1}}\} 
= 
K_{6}(m') \cap [m, m'], 
\quad 
0 < j_{3} < \dots < j_{N_{1}}, 
\\
& 
\{k_{N_{1}+1}, \dots, k_{N_{2}}\} 
= 
K_{6}(m') \cap [-n+1, -m'-1], 
\quad 
k_{N_{1}+1} < \dots < k_{N_{2}} < 0.  
\end{align*}
If $K_{6}(m') \cap [m, m']$ (resp. $K_{6}(m') \cap [-n+1, -m'-1]$) is
empty, then we set $N_{1} = 2$ (resp. $N_{2} = N_{1}$). 

When $m' \geq m$, define 
\begin{equation*}
\alpha_{p} = \alpha_{p}(m', Q) 
:= \left\{
\begin{array}{ll}
\pm \Lambda_{n} \pm \Lambda_{n+1}, & \mbox{ for } p=1, 
\\
\pm \Lambda_{n} + l_{2n-2,m-1} - 1, & \mbox{ for } p=2, 
\\
\pm \Lambda_{n} + l_{2n-2,j_{p}} - 1, 
& \mbox{ for } 3 \leq p \leq N_{1}, 
\\
\pm \Lambda_{n} + l_{2n-2,k_{p}} - 1, 
& \mbox{ for } N_{1}+1 \leq p \leq N_{2}. 
\end{array}
\right. 
\end{equation*}
Suppose $m' = m-1$. 
In this case, (i) $K_{6}(m') \cap [m, m']$ is empty, 
and (ii) 
the inequality relation of $\pm \Lambda_{n+1}$ and $l_{2n-2,m-1} - 1$
depends on the value of $q_{2n-2,m-1}$. 
For these reasons, we define 
\[
\alpha_{p} = \alpha_{p}(m', Q) 
:= \left\{
\begin{array}{ll}
\pm \Lambda_{n} + \max\{\pm \Lambda_{n+1}, l_{2n-2,m-1} - 1\}, 
& \mbox{ for } p=1, 
\\
\pm \Lambda_{n} + \min\{\pm \Lambda_{n+1}, l_{2n-2,m-1} - 1\},  
& \mbox{ for } p=2, 
\\
\pm \Lambda_{n} + l_{2n-2,k_{p}} - 1, 
& \mbox{ for } 3 \leq p \leq N_{2}. 
\end{array}
\right. 
\]
\begin{remark}
When $m' \in K_{6}(m')$, 
$l_{2n-2,m'}$ corresponds to $\alpha_{N_{1}}$ 
except for the case $m'=m-1$ and 
$l_{2n-2,m-1} -1 \geq \pm \Lambda_{n+1}$.  
\end{remark} 

Consider the number $l_{2n-3,p}$ for $p \in K_{7}(m')$. 
By definition \eqref{eq:K_7} of $K_{7}(m')$, 
\begin{align*}
K_{7}(m') 
&= 
\{j_{3}-1, \dots, j_{N_{1}} - 1\} 
\cup 
\{k_{N_{1}+1}, \dots, k_{N_{2}}\}. 
\end{align*}
Define 
\begin{equation*}
\beta_{p} = \beta_{p}(m', Q) 
:= \left\{
\begin{array}{ll}
\pm \Lambda_{n} + l_{2n-3,j_{p}-1}, 
& \mbox{ for } 3 \leq p \leq N_{1}, 
\\
\pm \Lambda_{n} + l_{2n-3, k_{p}}, 
& \mbox{ for } N_{1}+1 \leq p \leq N_{2}.  
\end{array}
\right. 
\end{equation*}
\begin{lemma}
These numbers satisfy 
\begin{align*}
\alpha_{1} \geq &\alpha_{2} \geq \beta_{3} > \alpha_{3} \geq \dots 
> \alpha_{N_{1}-1} \geq \beta_{N_{1}} > \alpha_{N_{1}} 
\\
&> 0 
\geq \beta_{N_{1}+1} > \alpha_{N_{1}+1} \geq \beta_{N_{1}+2} > \dots 
> 
\alpha_{N_{2}-1} \geq \beta_{N_{2}} > \alpha_{N_{2}}, 
\end{align*}
and the difference between $\beta_{p}$ and $\alpha_{p}$ is at least
two. 
\end{lemma}
\begin{proof}
Recall Remark~\ref{rem:contragredient} and the condition
\eqref{eq:5.13(4)}. 
If $p \in K_{6}(m') \cap [m-1, m'-1]$, then 
$l_{2n-2,p} - 1 
= 
\lambda_{p+1} + n - p - 1 = \Lambda_{p+1}$ 
since $q_{2n-2,p} = \lambda_{p+1}$. 
If $p \in K_{6}(m') \cap [-n-1, -m'-1]$, then 
$l_{2n-2,p} - 1 
= 
- l_{2n-2,-p} 
=
-(\lambda_{-p} + n + p) = -\Lambda_{-p}$ 
since $q_{2n-2,-p} = \lambda_{-p}$. 
We know $\Lambda_{m'} > l_{2n-2,m'} - 1 \geq \Lambda_{m'+1} - 1$ 
since $\lambda_{m'} \geq q_{2n-2,m'} \geq \lambda_{m'+1}$. 
It follows that 
\[
\alpha_{1} \geq \alpha_{2} > \dots > \alpha_{N_{1}} 
> 0 > \alpha_{N_{1}+1} > \dots > \alpha_{N_{2}}, 
\]
since the Harish-Chandra parameter 
$\Lambda = \sum_{i=1}^{n+1} \Lambda_{i} e_{i} \in \Xi_{m,\pm}$ 
satisfies 
\[
\Lambda_{1} > \cdots > \Lambda_{m-1} > \pm \Lambda_{n+1} > \Lambda_{m}
> \cdots > \Lambda_{n-1} > |\Lambda_{n}|. 
\]
Next, consider the numbers $\beta_{j}$. 
If $3 \leq p \leq N_{1}$, then 
$l_{2n-2, j_{p}-1} - 1 
\geq 
l_{2n-3, j_{p}-1} 
> 
l_{2n-2,j_{p}} 
> 
l_{2n-2,j_{p}} - 1$ 
since 
$q_{2n-2, j_{p}-1} \geq q_{2n-3, j_{p}-1} > q_{2n-2, j_{p}}$. 
If $N_{1}+1 \leq p \leq N_{2}$, then 
$l_{2n-2, k_{p}-1} - 1 
\geq 
l_{2n-3, k_{p}} 
> 
l_{2n-2,k_{p}} 
> 
l_{2n-2,k_{p}} - 1$ since 
$q_{2n-2, -k_{p}} > q_{2n-3, -k_{p}} \geq q_{2n-2, -k_{p}+1}$. 
It follows that $\alpha_{p-1} \geq \beta_{p} > \alpha_{p}$, and the
difference between $\beta_{p}$ and $\alpha_{p}$ is at least two. 
\end{proof}

In order to rewrite \eqref{eq:5.16(1)} and \eqref{eq:5.20(3)} in
a convenient form, let 
\begin{align}
& 
t_{1} :=\frac{\eta_{1}}{a_{1}}, 
\quad 
t_{2} := \mp \frac{\eta_{1} \eta_{2}}{2 a_{2}}, 
\quad 
\d_{t_{i}} := t_{i} \frac{\d}{\d t_{i}} \enskip (i=1,2) 
\quad 
\mbox{ and }
\label{eq:variable t}
\\
& 
n(Q,m';a) := 
a_{1}^{-n+1+d_{m'}(Q)} 
a_{2}^{-\sum_{p=1}^{m-1} l_{2n-1,p} 
+ \sum_{p=1}^{m-2} l_{2n-2,p} \mp (\Lambda_{n} + \lambda_{n+1}) 
- m + 2} 
\label{eq:n}
\\
& \hspace{60mm} \times 
\exp \left( \pm \frac{\eta_{2} a_{1}}{2 a_{2}}\right). 
\notag
\end{align}
Note that $t_{1} > 0$ since $\eta_{1} > 0$ and $a_{1} > 0$. 
We have $\d_{i} = -\d_{t_{i}}$ ($i = 1, 2$) and 
$\eta_{2} a_{1} / 2 a_{2} = \mp t_{2} / t_{1}$. 
Since $l_{2n-2,n} = \lambda_{n} = \Lambda_{n}$, we have the following
proposition. 
\begin{proposition}\label{proposition:modefied equations}
Assume $Q \in GT(\lambda)$ satisfies \eqref{eq:5.13(1)},
\eqref{eq:5.11} 
and \eqref{eq:5.13(4)}. 
Define 
\begin{equation*}
f(Q, m'; t) 
:= 
n(Q, m'; a)^{-1} c(Q; a).  
\end{equation*}
Then, the differential equations \eqref{eq:5.16(1)} and
\eqref{eq:5.20(3)} are expressed as 
\begin{align}
& 
(\d_{t_{1}}^{2} - \d_{t_{2}}^{2} - t_{1}^{2}) f(Q, m'; t) = 0, 
\label{eq:5.20(6)} 
\\
& 
\left\{
\prod_{p=1}^{N_{2}} 
(\d_{t_{2}} + \alpha_{p}) 
+ 
\frac{t_{2}}{t_{1}} 
(\d_{t_{1}} - \d_{t_{2}}) 
\prod_{p=3}^{N_{2}} 
(\d_{t_{2}} + \beta_{p}) 
\right\} 
f(Q, m'; t) 
= 0. 
\label{eq:5.20(7)} 
\end{align}
\end{proposition}
\begin{proposition}\label{proposition:local solutions}
Let $C_{j}$ be a loop starting and ending at $+ \infty$, 
crossing the real axis at $- \alpha_{j} - 1 < s < -\alpha_{j}$, 
and encircling all poles of 
$\Gamma(- \alpha_{p} - s)$, $p = j, \dots, N_{2}$, once in the negative
direction, but none of the poles of 
$\Gamma(\beta_{p} + s)$, $p = 3, \dots, j$, i.e. encircling the half real
axis $\{x + 0 i \in \C| x \geq - \alpha_{j}\}$ once in the negative
direction. 
Define 
\[
\begin{Bmatrix}
f_{1}^{K}(Q, m'; t) 
\\
f_{1}^{I}(Q, m'; t) 
\end{Bmatrix} 
:= 
\frac{1}{2 \pi \I} 
\int_{C_{1}} 
\frac{\prod_{p=1}^{N_{2}} \Gamma(-\alpha_{p} - s)}
{\prod_{p=3}^{N_{2}} \Gamma(1 - \beta_{p} - s)} 
\begin{Bmatrix} 
t_{2}^{s}\,  
K_{-s}(t_{1})
\\
(-t_{2})^{s}\,  I_{-s}(t_{1})
\end{Bmatrix} 
ds, 
\]
and, for $j = 2, \dots, N_{2}$, define 
\begin{align*}
& 
\begin{Bmatrix}
f_{j}^{K}(Q, m'; t) 
\\
f_{j}^{I}(Q, m'; t) 
\end{Bmatrix} 
\\
&:= 
\frac{1}{2 \pi \I} 
\int_{C_{j}} 
\frac{\prod_{p=j}^{N_{2}} \Gamma(-\alpha_{p} - s) 
\prod_{p=3}^{j} \Gamma(\beta_{p} + s)}
{\prod_{p=1}^{j-1} \Gamma(1 + \alpha_{p} + s) 
\prod_{p=j+1}^{N_{2}} \Gamma(1 - \beta_{p} - s)} 
\begin{Bmatrix} 
(- t_{2})^{s}\,  
K_{-s}(t_{1}) 
\\
t_{2}^{s}\,  
I_{-s}(t_{1})
\end{Bmatrix} 
ds.  
\end{align*}
Here, $K_{\mu}(z)$ and $I_{\nu}(z)$ are modified Bessel functions 
(cf. \cite{E}). 
Then these integrals absolutely converge for 
$(t_{1}, t_{2}) \in 
\C^{2} \setminus (\{t_{1} = 0\} \cup \{t_{2} = 0\})$, 
and they form a basis of the solution space of the system of 
equations \eqref{eq:5.20(6)} and \eqref{eq:5.20(7)}. 
Moreover, if $\alpha_{1} \not= \alpha_{2}$, 
the leading terms of these functions at $t_{2} = 0$ are 
\begin{align*}
& 
\begin{Bmatrix}
f_{1}^{K} 
\\
f_{1}^{I}
\end{Bmatrix} 
\quad \cdots \quad 
\frac{\prod_{p=2}^{N_{2}} \Gamma(\alpha_{1} -\alpha_{p})}
{\prod_{p=3}^{N_{2}} \Gamma(\alpha_{1} - \beta_{p} + 1)} 
\begin{Bmatrix}
t_{2}^{-\alpha_{1}}\,  
K_{\alpha_{1}}(t_{1})
\\
(-t_{2})^{-\alpha_{1}}\, 
I_{\alpha_{1}}(t_{1}) 
\end{Bmatrix}, 
\quad 
\mbox{and} 
\\
& 
\begin{Bmatrix} 
f_{j}^{K} 
\\
f_{j}^{I} 
\end{Bmatrix} 
\quad \cdots \quad 
\frac{\prod_{p=j+1}^{N_{2}} \Gamma(\alpha_{j} -\alpha_{p}) 
\prod_{p=3}^{j} \Gamma(\beta_{p} - \alpha_{j})}
{\prod_{p=1}^{j-1} \Gamma(1 + \alpha_{p} - \alpha_{j}) 
\prod_{p=j+1}^{N_{2}} \Gamma(\alpha_{j} - \beta_{p} + 1)} 
\begin{Bmatrix} 
(-t_{2})^{-\alpha_{j}}\,  
K_{\alpha_{j}}(t_{1})
\\
t_{2}^{-\alpha_{j}}\,  
I_{\alpha_{j}}(t_{1})
\end{Bmatrix} 
\end{align*}
if $j \geq 2$, respectively. 
If $\alpha_{1} = \alpha_{2}$, then the leading terms of $f_{1}^{K},
f_{1}^{I}$ are complicated. 
\end{proposition}
\begin{proof}
At a generic point, 
the solution space of the system of equations \eqref{eq:5.20(6)} and 
\eqref{eq:5.20(7)} is $2 N_{2}$ dimensional. 
It is easy to verify that these integrals formally satisfy the
equations \eqref{eq:5.20(6)} and \eqref{eq:5.20(7)}. 

Let $s = u + v \I$, where $u$ is a very large positive real number
and $v$ is a non-zero finite real number. 
We will see the asymptotic behavior of the integrands 
when $u \to \infty$. 
In the following of this proof, 
$C_{j}(v)$'s are positive constants which depend on $v$. 

By the asymptotic expansion of the gamma function, we know that, if
$|\arg s| < \pi$ and $|s|$ is large, 
\[
\Gamma(s + a) 
= 
s^{s-1/2+a} e^{-s} \sqrt{2 \pi} \times O(1), 
\quad 
|s| \to \infty. 
\]
Hence we have 
\begin{align*}
& 
\left|
\frac{\prod_{p=1}^{N_{2}} \Gamma(-\alpha_{p} - s)}
{\prod_{p=3}^{N} \Gamma(1 - \beta_{p} - s)} 
\right|, 
\left|
\frac{\prod_{p=j}^{N_{2}} \Gamma(-\alpha_{p} - s) 
\prod_{p=3}^{j} \Gamma(\beta_{p} + s)}
{\prod_{p=1}^{j-1} \Gamma(1 + \alpha_{p} + s) 
\prod_{p=j+1}^{N_{2}} \Gamma(1 - \beta_{p} - s)} 
\right|
\\
& 
\leq C_{1}(v) 
\exp 
\left\{
\left(
-2u - \sum_{p=1}^{N_{2}} \alpha_{p} + \sum_{p=3}^{N_{2}} \beta_{p} 
- N_{2} + 1 
\right) 
\log|u + v \I| 
+ 2 u \right\}. 
\end{align*}
Since 
\[
I_{\nu}(z) 
= 
\left(\frac{z}{2}\right)^{\nu} 
\sum_{n=0}^{\infty} 
\frac{(z/2)^{2n}}{n!\, \Gamma(\nu+n+1)} 
\quad 
\mbox{and} \quad 
K_{\nu}(z) 
= 
\frac{\pi}{2} 
\frac{I_{-\nu}(z) - I_{\nu}(z)}{\sin \nu \pi}, 
\]
and $t_{1}$ is a positive real number, 
we have 
\[
|K_{-u-v \I}(t_{1})|, 
|I_{-u-v \I}(t_{1})| 
\leq 
C_{2}(v) 
\exp 
\left\{
\left(u + \frac{1}{2} \right) \log |- u - v \I | 
- 
u (\log \frac{t_{1}}{2} + 1) 
\right\}.  
\]
Next, 
\[
|(\pm t_{2})^{u + v \I}| 
\leq C_{3}(v) |t_{2}|^{u}. 
\]
Thus the integrals in this proposition absolutely converge. 
The leading terms are obtained by the residue theorem. 
They imply that $f_{j}^{L}$, $j = 1, \dots, N_{2}$, $L = K, I$, 
are linearly independent. 
\end{proof}

\subsection{Solutions of the whole differential-difference equations}
\label{subsection:whole solutions} 
We consider whether the solutions of scalar equations
obtained in the previous subsection satisfy the whole
differential-difference equations 
$\mathcal{D}_{\tilde{\lambda}, \eta} \phi = 0$ or not. 

The equation \eqref{eq:5.15(1)}, 
with $Q$ replaced by $\tau_{m',m'} Q$, 
expresses $c(\tau_{m',m'} Q; a)$ as a sum of differentials of 
$c(Q; a)$ and $c(\tau_{0,m'} Q; a)$. 
The equation \eqref{eq:5.22}, with $j = m'$, 
expresses $c(\tau_{0,m'} Q; a)$ as a differential of $c(Q; a)$. 
By (i) eliminating $c(\tau_{0,m'} Q; a)$ from these equations, 
(ii) changing the independent variables from $a_{i}$ to $t_{i}$
and the dependent variables from $c(Q; a)$ to $f(Q; a)$, 
and (iii) simplifying the equation so obtained by using
\eqref{eq:5.20(6)} and \eqref{eq:5.20(7)}, 
we get the following shift operator. 
\begin{proposition}\label{prop:first shift operator}
Suppose $Q \in GT(\lambda)$ satisfies \eqref{eq:5.13(1)},
\eqref{eq:5.11} and \eqref{eq:5.13(4)}, 
and suppose $\tau_{m',m'} Q \in GT(\lambda)$. 
Then 
\begin{align*}
& 
f(\tau_{m',m'} Q, m'; t) 
= \mbox{\rm (nonzero constant)} \times S_{1}(m', Q) f(Q, m'; t), 
\\
& S_{1}(m', Q) 
\\
&:=
\frac{1}{t_{1} t_{2}} 
(\d_{t_{1}} + \d_{t_{2}}) 
\prod_{p = 1, \not= N_{1}}^{N_{2}}
(\d_{t_{2}} + \alpha_{p})
+ 
\frac{
\prod_{p = 3}^{N_{2}}
(\d_{t_{2}} + \beta_{p}) 
- \prod_{p = 3}^{N_{2}} (\beta_{p} - \alpha_{N_{1}} - 1)
}
{\d_{t_{2}} + \alpha_{N_{1}} + 1}. 
\end{align*}
\end{proposition}

For notational convenience, 
let $Q' := \tau_{m',m'} Q$ and denote $\alpha_{j}(Q')$,
$\beta_{j}(Q')$ by $\alpha_{j}'$, $\beta_{j}'$, respectively. 
As before, $\alpha_{j}$, $\beta_{j}$ mean $\alpha_{j}(Q)$,
$\beta_{j}(Q)$, respectively. 
Let $N_{1} = \# (K_{6}(m')(Q) \cap [m,m']) + 2$ and 
$N_{2} = \# (K_{6}(m')(Q) \cap [-n+1,-m'-1]) + N_{1}$, 
i.e. they are the numbers $N_{1}, N_{2}$ used in the previous section, 
defined for $Q$ (not $Q'$). 

Suppose $q_{2n-2,m'} = q_{2n-3,m'} < q_{2n-3,m'-1} = q_{2n-4,m'-1}$. 
Then $K_{6}(m') (Q') = K_{6}(m') (Q)$. 
Therefore, 
\begin{align*}
& 
\alpha_{j}' = \alpha_{j}, 
\quad 
\mbox{ for } j = 1, \dots, N_{1}-1, N_{1}+1, \dots, N_{2}, 
\\
& 
\alpha_{N_{1}}' = \alpha_{N_{1}} + 1, 
\\
& 
\beta_{j}' = \beta_{j}, 
\quad 
\mbox{ for } j = 3, \dots, N_{2}. 
\end{align*}
In this case, it is easy to verify that 
\[
S_{1} (m', Q) 
f_{j}^{L}(Q, m'; t) 
= 
\prod_{p=1}^{N_{2}} (\beta_{p} - \alpha_{N_{1}} - 1) 
\, 
\times 
\left\{
\begin{array}{ll} 
f_{j}^{L}(Q', m'; t)  & \mbox{if } j \leq N_{1}, 
\\
-f_{j}^{L}(Q', m'; t) & \mbox{if } j > N_{1}, 
\end{array}
\right. 
\]
for $j = 1, 2, \dots, N_{2}$ and $L = K$ or $I$. 

Suppose $m' \geq m$ and 
$q_{2n-2,m'} = q_{2n-3,m'} = q_{2n-3,m'-1} = q_{2n-4,m'-1}$. 
Then 
$K_{6}(m')(Q') = K_{6}(m')(Q) \setminus \{m'\}$. 
It follows that 
\begin{align*}
& 
\alpha_{j}' = \alpha_{j}, 
\quad 
\mbox{ for } j = 1, \dots, N_{1}-1, 
& & 
\alpha_{j}' = \alpha_{j+1}, 
\quad 
\mbox{ for } j = N_{1}, \dots, N_{2}-1, 
\\ 
& 
\beta_{j}' = \beta_{j}, 
\quad 
\mbox{ for } j = 3, \dots, N_{1}-1. 
& & 
\beta_{j}' = \beta_{j+1}, 
\quad 
\mbox{ for } j = N_{1}, \dots, N_{2}-1. 
\end{align*}
In this case, for $L = K, I$, 
\begin{align*}
& 
S_{1} (m', Q) 
f_{j}^{L}(Q, m'; t) 
= 
\prod_{p=1}^{N} (\beta_{p} - \alpha_{N_{1}} - 1) 
\times 
\begin{cases} 
f_{j}^{L}(Q', m'; t), 
&  
\mbox{for } j < N_{1} 
\\
- f_{j-1}^{L}(Q', m'; t), 
&  
\mbox{for } j > N_{1}. 
\end{cases}
\end{align*} 
On the other hand, 
the non-zero functions $S_{1} (m', Q) f_{N_{1}}^{L}(Q, m'; t)$, 
$L = K, I$, are not solutions of
\eqref{eq:5.20(6)} and \eqref{eq:5.20(7)}. 
Since 
$(\tau_{m',m'})^{q_{2n-4,m'-1} - \lambda_{m'+1}} Q_{m'}^{-} 
= 
Q_{m'}^{+}$, we get the following proposition. 
\begin{proposition}\label{prop:S_1}
For $m' \geq m$ and $L = K, I$, 
\begin{align*}
\prod_{p=0}^{q_{2n-4,m'-1} - \lambda_{m'+1}-1} 
& S_{1}(m', \tau_{m',m'}^{p} Q) 
f_{j}^{L}(Q_{m'}^{-}, m'; t)
\\
&= 
\begin{cases}
\mbox{\rm (non-zero constant)} \times 
f_{j}^{L}(Q_{m'}^{+}, m'; t), 
& 
\mbox{if } j < N_{1}, 
\\
\mbox{not a solution of \eqref{eq:5.20(6)} and \eqref{eq:5.20(7)} for
  $Q_{m'}^{+}$} 
& \mbox{if } j = N_{1}, 
\\
\mbox{\rm (non-zero constant)} \times 
f_{j-1}^{L}(Q_{m'}^{+}, m'; t), 
& 
\mbox{if } j > N_{1}. 
\end{cases}
\end{align*}
\end{proposition}

Suppose $m' \geq m$ and $\lambda_{m'} > q_{2n-4, m'-1}$. 
Consider $Q \in GT(\lambda)$ satisfying
\eqref{eq:5.13(1)}, \eqref{eq:5.11} and 
\begin{equation}\label{eq:5.13(4)'}
\left\{
\begin{array}{ll}
q_{2n-2,p} = \lambda_{p+1} \mbox{ for } p \in [m-1, m'-1], 
\\ 
q_{2n-3,p} = q_{2n-4,p} \mbox{ for } p \in [m-1, m'-2], 
\\
q_{2n-2,p} = \lambda_{p} \mbox{ for } p \in [m'+1, n-1], 
\\
q_{2n-3,p} = q_{2n-4,p-1} \mbox{ for } p \in [m', n-1], 
\\
q_{2n-2,m'} = q_{2n-3,m'-1} \in [q_{2n-4,m'-1}, \lambda_{m'}]. 
\end{array}
\right.
\end{equation}
Note that, if $q_{2n-2,m'} = q_{2n-3,m'-1} = q_{2n-4,m'-1}$, then such
$Q$ is $Q_{m'}^{+}$ defined in Definition~\ref{def:Q^pm}. 
If $q_{2n-2,m'} = q_{2n-3,m'-1} = \lambda_{m'}$, then such $Q$ is 
$Q_{m'-1}^{-}$. 

For $Q \in GT(\lambda)$ satisfying \eqref{eq:5.13(1)}, \eqref{eq:5.11} 
and \eqref{eq:5.13(4)'}, the equation \eqref{eq:5.10(2)}, with $j=m'$,
is 
\[
c(\tau_{m'-1,m'} Q; a) 
= 
\mbox{(nonzero constant)} 
\times 
\frac{a_{1}}{\eta_{1}} 
(\mathcal{D}_{2}^{\pm}(Q) + l_{2n-2,m'}) 
c(Q; a). 
\]
As before, let 
\begin{align*}
& 
f(Q_{m'}^{+}, m'; t) 
= 
n(Q_{m'}^{+}, m', a)^{-1} c(Q_{m'}^{+}; a), 
\\ 
& 
f(Q_{m'-1}^{-}, m'-1; t) 
= 
n(Q_{m'-1}^{-}, m'-1, a)^{-1} c(Q_{m'-1}^{-}; a). 
\end{align*}
Since 
$(\tau_{m'-1,m'})^{\lambda_{m'} - q_{2n-4,m'-1}} 
Q_{m'}^{+} 
= 
Q_{m'-1}^{-}$, 
we obtain a shift operator 
\begin{align*}
& 
f(Q_{m'-1}^{-}, m'-1; t) 
= 
\mbox{(non-zero constant)} \times S_{2}(m') f(Q_{m'}^{+}, m'; t), 
\\
& 
S_{2}(m') := 
\prod_{q=q_{2n-4,m'-1}}^{\lambda_{m'}-1} 
(\d_{t_{2}} \pm \Lambda_{n} - q - n + m').  
\end{align*}
Recall that, for
\[
K_{6}(m')(Q_{m'}^{+}) 
= 
\{m-1\} \cup 
\{j_{3}, \dots, j_{N_{1}}\} \cup \{k_{N_{1}+1}, \dots, k_{N_{2}}\},
\] 
$K_{7}(m')(Q_{m'}^{+})$ is defined to be 
\[
K_{7}(m')(Q_{m'}^{+}) 
= 
\{j_{3}-1, \dots, j_{N_{1}}-1\} \cup \{k_{N_{1}+1}, \dots, k_{N_{2}}\}. 
\]
Suppose $Q_{m'-1}^{-} \not= Q_{m'}^{+}$. 
Since 
$K_{6}(m'-1)(Q_{m'-1}^{-}) = K_{6}(m')(Q_{m'}^{+}) \cup \{-m'\}$, 
we have 
\begin{align*}
& 
K_{6}(m'-1)(Q_{m'-1}^{-}) 
= 
\{m-1\} \cup \{j_{3}, \dots, j_{N_{1}}\} \cup 
\{-m', k_{N_{1}+1}, \dots, k_{N_{2}}\}, 
\\
& 
K_{7}(m'-1)(Q_{m'-1}^{-}) 
= 
\{j_{3}-1, \dots, j_{N_{1}}-1\} \cup 
\{-m', k_{N_{1}+1}, \dots, k_{N_{2}}\}. 
\end{align*}
For simplicity, 
denote $\alpha_{j}(Q_{m'}^{+})$,
$\beta_{j}(Q_{m'}^{+})$ by $\alpha_{j}'$, $\beta_{j}'$, 
and 
$\alpha_{j}(Q_{m'-1}^{-})$,
$\beta_{j}(Q_{m'-1}^{-})$ by $\alpha_{j}''$, $\beta_{j}''$, respectively. 
They are related as follows: 
\begin{align*}
&
\alpha_{p}'' = \alpha_{p}' \quad \mbox{for } 1 \leq p \leq N_{1}, 
& & 
\beta_{p}'' = \beta_{p}' \quad \mbox{for } 3 \leq p \leq N_{1}, 
\\
& 
\alpha_{N_{1}+1}'' = \pm \Lambda_{n} - \lambda_{m'} - n + m' 
& & 
\beta_{N_{1}+1}'' = \pm \Lambda_{n} - q_{2n-4,m'-1} - n + m' + 1 
\\
& 
\alpha_{p}'' = \alpha_{p-1}' \quad \mbox{for } 
N_{1}+2 \leq p \leq N_{2}+1, 
& & 
\beta_{p}'' = \beta_{p-1}' \quad \mbox{for } 
N_{1}+2 \leq p \leq N_{2}+1, 
\end{align*}
and they are ordered as follows: 
\begin{align*}
\alpha_{1}'& >  \alpha_{2}' \geq \beta_{3}' > \alpha_{3}' \geq \dots 
> \alpha_{N_{1}-1}' \geq \beta_{N_{1}}' > \alpha_{N_{1}}' 
\\
&> 0 
> \beta_{N_{1}+1}'' > \alpha_{N_{1}+1}'' 
\geq \beta_{N_{1}+1}' > \alpha_{N_{1}+1}' \geq \beta_{N_{1}+2}' > \dots 
> 
\alpha_{N_{2}-1}' \geq \beta_{N_{2}}' > \alpha_{N_{2}}'. 
\end{align*}
Since 
\begin{align*}
S_{2}&(m') (\pm t_{2})^{s}
\\
&=
\frac{\Gamma(\beta_{N_{1}+1}''+s)}{\Gamma(1+\alpha_{N_{1}+1}''+s)} 
(\pm t_{2})^{s}
= 
(-)^{\beta_{N_{1}+1}''-\alpha_{N_{1}+1}''-1} 
\frac{\Gamma(-\alpha_{N_{1}+1}''-s)}{\Gamma(1-\beta_{N_{1}+1}''-s)} 
(\pm t_{2})^{s}, 
\end{align*}
we have the following proposition. 
\begin{proposition}\label{prop:S_2}
For $L = K$ or $I$,  
$S_{2}(m') 
f_{j}^{L}(Q_{m'}^{+}, m'; t)$ is a non-zero constant multiple of 
\begin{enumerate}
\item
$f_{j}^{L}(Q_{m'-1}^{-}, m'-1; t)$ if $j = 1, \dots, N_{1}$, 
\item
$f_{j+1}^{L}(Q_{m'-1}^{-}, m'-1; t)$ 
if $j = N_{1}+1, \dots, N_{2}$. 
\end{enumerate}
\end{proposition}

Propositions~\ref{prop:S_1}, \ref{prop:S_2} enable us to judge whether
$f_{j}^{L}$, $j = 1, \dots, N_{2}$, $L = K, I$, generates the whole
solution of $\mathcal{D}_{\tilde{\lambda}, \eta} \phi = 0$. 

By Theorem~\ref{thm:corner}, $c(Q_{n-1}^{+}; a)$, therefore
$f(Q_{n-1}^{+}; t)$, determines all the $c(Q; a)$ containing the same
$\vect{q}_{1}, \dots, \vect{q}_{2n-4}$ parts as $Q_{n-1}^{+}$. 
Let $\alpha_{j} = \alpha_{j}(Q_{n-1}^{+})$. 
This is the leading exponent of $f_{j}^{L}(Q_{n-1}^{+}; t)$ 
at $t_{2} = 0$. 
Since $\alpha_{1} = \pm (\Lambda_{n} + \Lambda_{n+1})$ and 
$\alpha_{2} = \pm \Lambda_{n} + l_{2n-2,m-1} -1$, 
alternative use of Proposition~\ref{prop:S_1} and
Proposition~\ref{prop:S_2} implies that, if $j = 1, 2$ and $L = K, I$, 
\begin{align*}
f_{j}^{L}(Q_{n-1}^{+}, n-1; t) 
&\overset{\mbox{\scriptsize Prop.\ref{prop:S_1}}}{\longrightarrow}
f_{j}^{L}(Q_{n-2}^{-}, n-2; t) 
\overset{\mbox{\scriptsize Prop.\ref{prop:S_2}}}{\longrightarrow}
f_{j}^{L}(Q_{n-2}^{+}, n-2; t) 
\longrightarrow 
\dots 
\\
& \dots 
\overset{\mbox{\scriptsize Prop.\ref{prop:S_1}}}{\longrightarrow}
f_{j}^{L}(Q_{m-1}^{-}, m-1; t) 
\end{align*}
On the other hand, if $j \geq 3$, there exists 
$m' \in K_{6}(n-1)(Q_{n-1}^{+}) \cap [m, n-2]$ such that 
$\alpha_{j}(Q_{n-1}^{+}) = \pm \Lambda_{n} + l_{2n-2,m'} - 1$. 
In this case, 
\begin{align*}
& f_{j}^{L}(Q_{n-1}^{+}, n-1; t) 
\overset{\mbox{\scriptsize Prop.\ref{prop:S_1}}}{\longrightarrow}
f_{j}^{L}(Q_{n-2}^{-}, n-2; t) 
\overset{\mbox{\scriptsize Prop.\ref{prop:S_2}}}{\longrightarrow}
f_{j}^{L}(Q_{n-2}^{+}, n-2; t) 
\longrightarrow 
\dots 
\\
& \dots 
\overset{\mbox{\scriptsize Prop.\ref{prop:S_2}}}{\longrightarrow} 
f_{j}^{L}(Q_{m'}^{-}, m'; t) 
\overset{\mbox{\scriptsize Prop.\ref{prop:S_1}}}{\longrightarrow}
\mbox{not a solution of \eqref{eq:5.20(6)} and \eqref{eq:5.20(7)} for
  $Q_{m'}^{+}$}. 
\end{align*}
This implies that only $f_{j}^{L}(Q_{n-1}^{+}, n-1; t)$, 
$j = 1, 2$, $L = K, I$, can generate
a solution of the whole differential-difference equations. 
It follows that the constant $C$ in Theorem~\ref{thm:Bernstein} is at
most four. 
In the special case when $m = n$, we can check the compatibility of
the equations in Lemma~\ref{lemma:5.8}. 
Therefore, this constant $C$, which is independent of $m$, is just
four. 

\begin{theorem}\label{thm:last theorem, algebraic}
Let $\Lambda \in \Xi_{m,\pm}$, $m = 2, \dots, n$. 
The functions $f_{1}^{K}(Q_{n-1}^{+}, n-1; t)$, 
$f_{1}^{I}(Q_{n-1}^{+}, n-1; t)$, $f_{2}^{K}(Q_{n-1}^{+}, n-1; t)$, 
$f_{2}^{I}(Q_{n-1}^{+}, n-1; t)$, 
with $Q_{n-1}^{+}$ defined in Definition~\ref{def:Q^pm} 
completely determines the non-zero solutions of 
$\mathcal{D}_{\tilde{\lambda}, \eta} \phi = 0$. 
\end{theorem}

\subsection{Continuous Whittaker models} 

So far, we have investigated the space 
$\Hom_{\lie{g}, K_{\R}}(\pi_{\Lambda}^{\ast}, 
C^{\infty}(G_{\R}/N_{\R}; \eta))$. 
Here, we specify the subspace of continuous intertwining
operators 
$\Hom_{G_{\R}}^{\infty}((\pi_{\Lambda}^{\ast})_{\infty}, 
C^{\infty}(G_{\R}/N_{\R}; \eta))$. 
Note that the latter space is isomorphic to 
$\Wh_{-\eta}^{\infty}(\pi_{\Lambda}^{\ast})$. 
\begin{proposition}\label{prop:continuous intertwining space}
Suppose $G_{\R} = Spin(2n, 2)$. 
Let $\pi_{\Lambda}$ be the discrete series representation with the
Harish-Chandra parameter $\Lambda$. 
\begin{enumerate}
\item
Suppose $\Lambda \in \Xi_{m,+}$ and $\Lambda' \in \Xi_{m,-}$, 
$m = 2, \dots, n$. 
If $\Wh_{-\eta}^{\infty}(\pi_{\Lambda}^{\ast}) \not= \{0\}$, 
then $\Wh_{-\eta}^{\infty}(\pi_{\Lambda'}^{\ast}) = \{0\}$. 
\item
Let $\eta$ and $\eta'$ be non-degenerate unitary characters defined as
in \eqref{eq:character}. 
Suppose $\eta_{2} \eta_{2}' < 0$. 
If $\Wh_{-\eta}^{\infty}(\pi_{\Lambda}^{\ast}) \not= \{0\}$, 
then $\Wh_{-\eta'}^{\infty}(\pi_{\Lambda}^{\ast}) = \{0\}$. 
\item
Suppose $\Lambda \in \Xi_{m,\pm}$, $m = 2, \dots, n$. 
If $\Wh_{-\eta}^{\infty}(\pi_{\Lambda}^{\ast}) \not= \{0\}$, then 
\[
\dim \Wh_{-\eta}^{\infty}(\pi_{\Lambda}^{\ast}) 
= 
\sum_{\genfrac{}{}{0pt}{}{
\lambda_1 \geq \mu_1 \geq \lambda_2 \geq \dots 
\geq \lambda_{m-2} \geq \mu_{m-2} \geq \lambda_{m-1}}
{\lambda_m \geq \mu_1' \geq \lambda_{m+1} \geq \dots 
\geq \lambda_{n-1} \geq \mu_{n-m}' \geq |\lambda_n|} }
\dim
V_{(\mu_1,\dots,\mu_{m-2},\mu_1',\dots,\mu_{n-m}')}^{Spin(2n-3,\C)}. 
\]
\end{enumerate}
\end{proposition}
\begin{proof}
For our case $G_{\R} = Spin(2n, 2)$, there are two principal nilpotent
$G_{\R}$-orbits on $\lier{g}$, 
and $W_{G_{\R}} \simeq \mathfrak{S}_{2} \ltimes (\Z/2\Z)^{2}$. 
Therefore, (3) follows from Theorems~\ref{thm:matu-2},
\ref{thm:Bernstein}, Remark~\ref{rem:contragredient} and
Theorem~\ref{thm:last theorem, algebraic}. 

Via the Kostant-Sekiguchi correspondence (\cite{S}), 
these two orbits correspond to the two nilpotent $K$-orbits on
$\lie{p}$ generated by 
$X_{- e_{1} \pm e_{n+1}} + X_{e_{n} \mp e_{n+1}} 
+ X_{- e_{n} \mp e_{n+1}}$ 
(see the proof of Proposition~\ref{prop:GK dim}). 
Schmid and Vilonen (\cite{SV}) proved that the associated cycle of a
Harish-Chandra $(\lie{g}, K_{\R})$-module and the wave front cycle of
it are related under the Kostant-Sekiguchi correspondence. 
It follows that, for $m = 2, \dots, n$, the wave front set 
of the discrete series $\pi_{\Lambda}$ with $\Lambda \in \Xi_{m,+}$ and 
that of $\pi_{\Lambda'}$ with $\Lambda' \in \Xi_{m,-}$ are different
principal nilpotent $G_{\R}$-orbits. 
Therefore, (1) is a consequence of Theorem~\ref{thm:matu-2} (1). 
Recall the identification of a unitary character with an element
of $\I (\lier{n}/[\lier{n}, \lier{n}])^{\ast} 
\subset 
\I \lier{g}^{\ast} 
\simeq 
\I \lier{g}$ (cf. \S \ref{subsection:Whittaker models}). 
It is easy to check that $\eta$ and $\eta'$ are contained in different
principal nilpotent $G_{\R}$-orbits multiplied by $\I$. 
Therefore, (2) also follows from Theorem~\ref{thm:matu-2} (1). 
\end{proof}

By a theorem of Wallach (\cite{W}), 
if $\psi \in \Hom_{G_{\R}}^{\infty}((\pi_{\Lambda}^{\ast})_{\infty}, 
C^{\infty}(G_{\R}/N_{\R}; \eta))$, 
then $\psi(v)(g)$, 
$v \in (\pi_{\Lambda}^{\ast})_{\infty}$, $g \in G_{\R}$, 
must be a moderate-growth function. 
We show that, if $\mp \eta_{2} > 0$, 
then the function $f_{1}^{K}$ defined in
Proposition~\ref{proposition:local solutions} generates a rapidly
decreasing Whittaker function. 

Recall the definition \eqref{eq:variable t} of $t_{1}, t_{2}$. 
Since $a_{1}, a_{2} > 0$, $\eta_{1} > 0$ and $\mp \eta_{2} > 0$, 
we have $t_{1} > 0$ and $t_{2} > 0$. 
Let 
\begin{align*}
& 
S_{3} = 
\prod_{p=3}^{N_{2}} \prod_{q=\alpha_{p}+1}^{\beta_{p}-1} 
(-\d_{t_{2}} - q), 
& 
&
f_{0}(t) 
= 
\frac{1}{2 \pi \I} 
\int_{C_{1}} \Gamma(-\alpha_{1}-s) \Gamma(-\alpha_{2}-s) 
t_{2}^{s}\,  
K_{-s}(t_{1})\, ds. 
\end{align*} 
Then $f_{1}^{K} = S_{3} f_{0}$, so 
we show that $f_{0}$ is a rapidly decreasing function. 
Recall an integral formula 
\[
K_{\nu}(z) 
= 
\frac{1}{2} 
\left(\frac{z}{2}\right)^{\nu} 
\int_{0}^{\infty} 
\exp\left(- u - \frac{z^{2}}{4u}\right) u^{-\nu-1} du 
\]
of $K_{\nu}(z)$. 
Then $f_{0}$ is 
\begin{align*}
&
\frac{1}{2 \pi \I} 
\int_{C_{1}} \Gamma(-\alpha_{1}-s) \Gamma(-\alpha_{2}-s) 
t_{2}^{s}\,  
\left(
\frac{1}{2} 
\left(\frac{t_{1}}{2}\right)^{-s} 
\int_{0}^{\infty} 
\exp\left(- u - \frac{t_{1}^{2}}{4u}\right) u^{s-1} du
\right)\, ds
\\
&=
\int_{0}^{\infty} 
\exp\left(- u - \frac{t_{1}^{2}}{4u}\right) 
\left(
\frac{1}{4 \pi \I} 
\int_{C_{1}} \Gamma(-\alpha_{1}-s) \Gamma(-\alpha_{2}-s) 
\left(\frac{2 t_{2} u}{t_{1}}\right)^{s}\, ds 
\right) \frac{du}{u}. 
\end{align*}
By residue calculus, the inner integral is expressed by a $K$-Bessel
function, 
and then we get 
\[
f_{0} 
= 
\int_{0}^{\infty} 
\exp\left(- u - \frac{t_{1}^{2}}{4u}\right) 
\left(\frac{2 t_{2} u}{t_{1}}\right)^{-(\alpha_{1}+\alpha_{2})/2} 
K_{\alpha_{1} - \alpha_{2}} 
\left(2 \sqrt{\frac{2 t_{2} u}{t_{1}}}\right)
\frac{du}{u}. 
\]
This is essentially the same as the function $h_{r,0}$
treated in \cite[Theorem~4.4]{HO}. 
This function is proved to be a rapidly decreasing function there. 
It is not hard to see that this function generates a rapidly
decreasing solution of $\mathcal{D}_{\tilde{\lambda}, \eta} \phi = 0$. 
Therefore, the intertwining operator corresponding to this solution
is an element of 
$\Hom_{G_{\R}}^{\infty}((\pi_{\Lambda}^{\ast})_{\infty}, 
C^{\infty}(G_{\R}/N_{\R}; \eta))$. 
This result, 
together with Proposition~\ref{prop:continuous intertwining space}, 
implies the following theorem. 
\begin{theorem}\label{thm:continuous intertwining space}
Suppose $G_{\R} = Spin(2n, 2)$. 
Let $\pi_{\Lambda}$ be the discrete series representation with the
Harish-Chandra parameter $\Lambda \in \Xi_{m,\pm}$, 
$m = 2, \dots, n$. 
If $\mp \eta_{2} > 0$, then 
\begin{align*}
\dim &
\Hom_{G_{\R}}^{\infty}((\pi_{\Lambda}^{\ast})_{\infty}, 
C^{\infty}(G_{\R}/N_{\R}; \eta))
\\
&= 
\sum_{\genfrac{}{}{0pt}{}{
\lambda_1 \geq \mu_1 \geq \lambda_2 \geq \dots 
\geq \lambda_{m-2} \geq \mu_{m-2} \geq \lambda_{m-1}}
{\lambda_m \geq \mu_1' \geq \lambda_{m+1} \geq \dots 
\geq \lambda_{n-1} \geq \mu_{n-m}' \geq |\lambda_n|} }
\dim
V_{(\mu_1,\dots,\mu_{m-2},\mu_1',\dots,\mu_{n-m}')}^{Spin(2n-3,\C)}. 
\end{align*}
Each continuous intertwining operator corresponds to $f_{1}^{K}$
defined in Theorem~\ref{proposition:local solutions}. 
On the other hand, if $\mp \eta_{2} < 0$, then 
\[
\Hom_{G_{\R}}^{\infty}((\pi_{\Lambda}^{\ast})_{\infty}, 
C^{\infty}(G_{\R}/N_{\R}; \eta)) 
= 
\{0\}. 
\]
\end{theorem}


\begin{thebibliography}{9}

\bibitem{BB} 
A. Beilinson and J. Bernstein, 
Localisation de $\lie{g}$-modules, 
C. R. Acad. Sci. Paris Ser. I Math. \textbf{292} (1981), 15--18. 

\bibitem{C1}
J.T. Chang, 
Characteristic cycles of holomorphic discrete series, 
Trans. Amer. Math. Soc. \textbf{334} (1992), 213--227. 


\bibitem{C2}
J.T. Chang, 
Characteristic cycles of discrete series 
for $\R$-rank one groups, 
Trans. Amer. Math. Soc. \textbf{341} (1994), 603--622. 

\bibitem{E}
A. Erd\'elyi et al, 
{\it Higher Transcendental Functions, volume I}, 
McGraw-Hill, 1953. 


\bibitem{GT}
I.M. Gelfand and M.L. Tsetlin, 
Finite-dimensional representations of the group of orthogonal
  matrices, Dokl. Akad, Nauk SSSR \textbf{71} (1950), 1017--1020. 

\bibitem{HO}
T. Hayata and T. Oda, 
An explicit integral representation of Whittaker functions for
  the representations of the discrete series -- the cases of $SU(2,2)$
  --, J. Math. Kyoto Univ. \textbf{37-3} (1997), 519--530. 

\bibitem{Kn}
A. W. Knapp, 
Representation theory of Semisimple Groups: An Overview Based on
Examples, 
Princeton University Press, Princeton, (1986). 

\bibitem{Kr} 
H. Kraljevi\v{c}, 
Representation of the universal covering group of the group
  $SU(n,1)$, 
Glasnik Mathemati\v{c}ki \textbf{8(28)} No. 1 (1973), 23--72. 

\bibitem{M1}
H. Matumoto, 
Whittaker vectors and the Goodman-Wallach operators, 
Acta math. \textbf{161} (1988), 183--241. 

\bibitem{M2}
H. Matumoto, 
$C^{-\infty}$-Whittaker vectors corresponding to a principal nilpotent
orbit of a real reductive linear Lie group, and wave front set, 
Compositio Math. \textbf{82} (1992), 189--244. 

\bibitem{SV} 
W. Schmid and K. Vilonen, 
Characteristic cycles and wave front cycles of representations of
reductive Lie groups, 
Annals of Math. \textbf{151} (2000), 1071--1118. 

\bibitem{S} 
J. Sekiguchi, 
Remarks on nilpotent orbits of a symmetric pair, 
J. Math. Soc. Japan \textbf{39} (1987), 127--138. 

\bibitem{T}
K. Taniguchi, Discrete series Whittaker functions of $SU(n,1)$ and
$Spin(2n,1)$, J. Math. Sci. Univ. Tokyo {\bf 3} (1996), no.~2,
331--377.

\bibitem{Vogan}
D.A. Vogan Jr., 
Associated varieties and unipotent representations, 
in Harmonic analysis on reductive groups : 
Proceedings of the Bowdoin conference 1989, Bowdoin College, (1989), 
315--388. 

\bibitem{W} 
N. R. Wallach, 
Asymptotic expansions of generalized matrix entries of representations
of real reductive groups, 
Lie group representations, I, 287--369, Lecture Notes in
Math., \textbf{1024}, Springer Verlag, Berlin, 1983. 

\bibitem{Yamashita}
H. Yamashita, 
Embeddings of discrete series into induced representations of
  semisimple Lie groups, I -- General theory and the case of $SU(2,2)$
  --, Japan. J. Math. \textbf{16-1} (1990), 31--95. 


\end{thebibliography}
\end{document}